\documentclass[a4paper,11pt,twoside]{amsart}  
\DeclareSymbolFontAlphabet{\mathcal}{symbols}
\usepackage[all]{xy}
\usepackage{amsmath}
\usepackage[psamsfonts]{amssymb}
\usepackage[psamsfonts]{eucal}
\usepackage{amsthm}
\usepackage[scaled]{helvet}
\usepackage{geometry}
\usepackage{color}

\usepackage{tikz}
\usetikzlibrary{arrows}
\usepackage{enumerate}
\theoremstyle{plain}
\newtheorem{theorem}{Theorem}

\newtheorem{proposition}{Proposition}[section]

\newtheorem{lemma}[proposition]{Lemma}
\newtheorem{corollary}[proposition]{Corollary}

\theoremstyle{definition}
\newtheorem{definition}[proposition]{Definition}
\newtheorem{example}[proposition]{Example}

\theoremstyle{remark}
\newtheorem{remark}[proposition]{Remark}

\newcommand{\secref}[1]{Section~\ref{#1}}

\newcommand{\thmref}[1]{Theorem~\ref{#1}}
\newcommand{\propref}[1]{Proposition~\ref{#1}}
\newcommand{\lemref}[1]{Lemma~\ref{#1}}

\newcommand{\remref}[1]{Remark~\ref{#1}}
\newcommand{\exemref}[1]{Example~\ref{#1}}

\newcommand{\defref}[1]{Definition~\ref{#1}}

\def\R{{\mathbb R}}

\def\ov{\overline}
\def\ob{\underline}

\def\cD{{\mathcal D}}
\def\cE{{\mathcal E}}
\def\cF{{\mathcal F}}

\def\cH{{\mathcal H}}

\def\cL{{\mathcal L}}

\def\cN{{\mathcal N}}

\def\cP{{\mathcal P}}

\def\cU{{\mathcal U}}

\def\cZ{{\mathcal Z}}
\def\tp{\mathtt p}
\def\tv{{\mathtt v}}
\def\C{\mathbb{C}}

\def\F{\mathbb{F}}

\def\N{\mathbb{N}}

\def\R{\mathbb{R}}
\def\Z{\mathbb{Z}}

\def\ker{{\rm Ker\,}}

\def\im{{\rm Im\,}}

\def\id{{\rm id}}

\def\tC{{\widetilde{C}}}

\def\tN{{\widetilde{N}}}
\def\tdelta{{\widetilde{\Delta}}}

\def\fil{{{\pmb\Delta}^{[n]}_\cF}{ -Sets}}

\def\GM{{\rm GM}}
\def\TW{{\rm TW}}
\def\ds{\boldsymbol{\mathcal P}}

\def\ffs{{filtered face set}}
\def\ffss{{filtered face sets}}

\def\ev{{\tt eval}}
\def\res{{\tt Res}}
\def\cPloose{{\cP_{{\rm loose}}^n}}
\def\sq{{\rm Sq}}
\def\sqg{{\rm Sq}_{\rm\scriptscriptstyle G}}

\def\IN{{\rm {\mathbf{IN}}}}
\def\cov{{\rm Cov}}
\def\bSS{{\rm{\mathbf{S}}}}

\def\Th{{\rm Th}}
\def\da{{\delta^a}}
\def\dca{{\delta^{ca}}}
\def\db{{\delta^b}}
\def\dcb{{\delta^{cb}}}
\def\dab{{\delta^{a\ast b}}}
\def\dcab{{\delta^{c(a\ast b)}}}
\def\dcacb{{\delta^{ca\otimes cb}}}
\def\dcatb{{\delta^{ca\otimes b}}}
\def\lhs{{\tt LHS}}
\def\rhs{{\tt RHS}}
\def\sd{{\rm sd}}

\title[Intersection cohomology and Steenrod squares]{Steenrod squares on Intersection cohomology and a conjecture of M.~Goresky and W.~Pardon} 
\date{\today}
\date{\today}

\author{David Chataur}
\address{D\'epartement de Math\'ematiques\\
         UMR 8524 et F\'ed\'eration CNRS Nord-Pas-de-Calais FR 2956\\
         Universit\'e de Lille~1\\
         59655 Villeneuve d'Ascq Cedex\\
         France}
\email{David.Chataur@math.univ-lille1.fr}

\author{Martintxo Saralegi-Aranguren}
\address{Laboratoire de Math{\'e}matiques de Lens\\  
      EA 2462 et F\'ed\'eration CNRS Nord-Pas-de-Calais FR 2956\\
      Universit\'e d'Artois\\
         SP18, rue Jean Souvraz\\
          62307 Lens Cedex\\
         France}
\email{saralegi@euler.univ-artois.fr}

\author{Daniel Tanr\'e}
\address{D\'epartement de Math\'ematiques\\
         UMR 8524 et F\'ed\'eration CNRS Nord-Pas-de-Calais FR 2956\\
         Universit\'e de Lille~1\\
         59655 Villeneuve d'Ascq Cedex\\
         France}
\email{Daniel.Tanre@univ-lille1.fr}

\thanks{The third author is partially supported by the MICINN grant MTM2010-18089,  ANR-11-BS01-002-01 ``HOGT" and ANR-11-LABX-0007-01  ``CEMPI''}

 
\begin{document} 
\begin{abstract} 
We prove a conjecture raised by M.~Goresky and W.~Pardon, 
concerning the range of validity of the perverse degree of Steenrod squares in intersection cohomology. 
This answer turns out to be  of importance for the definition of characteristic classes
 in the framework of intersection cohomology.

For this purpose, we present a  construction of ${\rm cup}_{i}$-products on the cochain complex, 
built on the blow-up of some singular simplices and introduced in a previous work. 
We extend to this setting the classical properties of the associated Steenrod squares, 
including Adem and Cartan relations, for any loose perversities. 
In the case of a PL-pseudomanifold and range $2\ov{p}$, 
we prove that our definition coincides with M.~Goresky's definition. We  show also that our Steenrod squares are  topological invariants which do not depend on the choice of a stratification of $X$. 
 
Several examples of concrete computation of perverse Steenrod squares are given, 
including the case of isolated singularities and, more especially, 
we describe the Steenrod squares on the Thom space of a vector bundle, 
as a function of the Steenrod squares of the base space and the Stiefel-Whitney classes of the bundle. 
We detail also an example of a non-trivial square, $\sq^2\colon H_{\ov{p}}\to H_{\ov{p}+2}$, 
whose information is lost if we consider it as taking values in $H_{2\ov{p}}$, 
showing the interest of the Goresky-Pardon conjecture.
\end{abstract}

\maketitle
%%%%%%%%%%%%%%%%%

Intersection cohomology was introduced by M.~Goresky and R.~MacPherson in \cite{MR572580} and \cite{MR696691}, in order to adapt Poincar\'e duality to singular manifolds and extend characteristic classes to this paradigm.
Steenrod squares on the intersection cohomology of a pseudomanifold, $X$,  were already  defined and studied by M. Goresky in \cite{MR761809}. For that, he uses a sheaf introduced by Deligne and proves that the Steenrod construction of ${\rm cup}_{i}$-products induces a morphism,
$\sq^i_{G}\colon H^r_{\ov{p}}(X;\F_{2})\to H^{r+i}_{2\ov{p}}(X;\F_{2})$, for any Goresky-MacPherson perversity $\ov{p}$ such that $2\ov{p}(\ell)\leq \ell-2$ for any $\ell$ and with $\F_{2}$ the field with two elements.

Here, we consider the blow-up, $\tN^*(X)$, of the normalized cochain complex on a filtered version of the singular simplicial set associated to $X$. 
This notion of blow-up,  defined in \cite{2012arXiv1205.7057C} and recalled in \secref{sec:perversity},  comes from a version adapted to differential forms already existent in \cite{MR1143404}. 
The elements of $\tN^*(X)$ have a perverse degree (see \defref{def:tranversedegree}) which allows the definition of a complex, $\tN^*_{\ov{p}}(X)$, for  any loose perversity $\ov{p}$. 
In \cite{2012arXiv1205.7057C}, we have proved that the blow-up,
$\tC^*(X)$, gives the Goresky-MacPherson intersection cohomology of the pseudomanifold $X$, 
for the complementary perversity, when we are working over a field. 
With \propref{prop:CandN}, the blow-up  $\tN^*(X)$ inherits this property; we denote its cohomology by 
$H^*_{\TW,\bullet}(X;\F_{2})$.

When the coefficients of $\tN^*(X)$ are in $\F_{2}$, we define a structure of ${\rm cup}_{i}$-products,
$\cup_{i}\colon \tN^*_{\ov{p}}(X)\otimes \tN^*_{\ov{q}}(X)\to \tN^*_{\ov{p}+\ov{q}}(X)$,
for any loose perversities $\ov{p}$, $\ov{q}$. This is done following the work of C.~Berger and B.~Fresse in \cite{MR2075046} (see also \cite{MR0281196}):
we consider a normalized, homogeneous Bar resolution, $\cE(2)$, of the symmetric group $\Sigma_{2}$ and prove that there exists a $\Sigma_{2}$-equivariant cochain map,
$\psi_{2}\colon \cE(2)\otimes \tN^*_{\ov{p}}(X)\otimes \tN^*_{\ov{q}}(X)\to \tN^*_{\ov{p}+\ov{q}}(X)
$. Such a map is called \emph{a structure of perverse $\cE(2)$-algebra on $\tN^*_{\bullet}(X)$;} its construction comes from the existence of a diagonal on $\cE(2)$, established in \cite{MR2075046}. Moreover, we prove in \thmref{thm:NXE2} that the ${\rm cup}_{i}$-products arising from the existence of $\psi_{2}$ verify the two following  properties,
$a\cup_{|a|}a=a$ and $a\cup_{i}a'=0$, if $i\geq \min(|a|,|a'|)$ where $|a|$, $|a'|$ are the respective degrees of $a$ and $a'$. 
 
The definition of perverse $\cE(2)$-algebras can be extended to perverse $\cE(n)$-algebras, for any $n$.  As this work is concerned with Steenrod squares, we consider only perverse $\cE(2)$-algebras over $\F_{2}$. Nevertheless, it is clear that our methods of proof can be enhanced to give a structure of perverse $E_{\infty}$-algebras over $\Z$ on $\tN^*_{\bullet}(X)$.  We will come back on these points in a forthcoming paper.

As usual, Steenrod squares are defined on $H^k_{\TW,\ov{p}}(X;\F_{2})$ by
$\sq^i(a)=a\cup_{k-i}a$. Using  May's presentation of Steenrod squares in \cite{MR0281196}, we see that the classical properties of Steenrod squares are direct consequences of the structure of perverse $\cE(2)$-algebra. We collect them, together with Adem and Cartan relations, in \thmref{thm:steenrodsquare}. (One may observe that the proof of the Adem relation on a tensor product needs a brief incursion in the world of perverse $\cE(4)$-algebras over $\F_{2}$.)

In \thmref{thm:steenrodsquare}, we also answer positively to the problem asked by M.~Goresky in \cite[Page~493]
{MR761809} and to the conjecture made by M.~Goresky and W.~Pardon in \cite[Conjecture 7.5]{MR1014465}. 
This problem concerns the range of  the perversities: with the definition of Steenrod squares 
via the ${\rm cup}_{i}$-products, it is clear that $\sq^i$ sends $H^k_{\TW,\ov{p}}(X;\F_{2})$
 into $H^{k+i}_{\TW,2\ov{p}}(X;\F_{2})$. 
 We prove that, in fact, there is a lifting as a map, 
 $\sq^i\colon H^k_{\TW,\ov{p}}(X;\F_{2})\to H^{k+i}_{\TW,\cL(\ov{p},i)}(X;\F_{2})$, 
where $\cL(\ov{p},i)$ is the loose perversity defined by 
$\cL(\ov{p},i)(\ell)=$\linebreak
$\min(2\ov{p}(\ell), \ov{p}(\ell)+i)$,
which is exactly  
\cite[Conjecture 7.5]{MR1014465}.
This reveals an important fact because it allows the lifting of Wu classes in  intersection cohomology, in a lower part of the poset of perversities.

In \thmref{thm:Goreskyandblowup}, we prove that our definition of Steenrod squares coincides with Goresky's definition introduced in \cite{MR761809}. 
For doing that, we transform the blow-up, $\tN^*_{\bullet}$, into a sheaf $\IN^*_{\bullet}$ on $X$ and prove that $\IN^*_{\bullet}$ is isomorphic to the Deligne sheaf, in the derived category of sheaves on $X$. 
The rest of the proof comes from a unicity theorem for Steenrod squares defined on an injective sheaf, established by M.~Goresky, \cite{MR761809}.

We end this part of the work with examples of concrete computation of perverse Steenrod squares, 
beginning with the case of isolated singularities. 
From it, we are able to write the Steenrod squares on the intersection cohomology of the Thom space associated to a
 vector bundle, as a function of the Steenrod squares of the base space and the Stiefel-Whitney classes of the bundle. 
We detail also an example of a non-trivial square, 
$\sq^2\colon H_{\TW,\ov{p}}(X;\F_{2})\to H_{\TW,\cL(\ov{p},2)}(X;\F_{2})$, 
whose information
 is lost if we consider it as values in $H_{\TW,2\ov{p}}$, showing the interest of the Goresky-Pardon conjecture. 
This last example can also be seen as a tubular neighborhood of a stratum, 
which is the first step in the study of intersection cohomology of pseudomanifolds.

In \thmref{thm:topinvariance}, we prove that Steenrod squares,
$\sq^i\colon H^r_{\TW,\ov{p}}(X;\F_{2})\to H^{r+i}_{\TW,\cL(\ov{p},i)}(X;\F_{2})$,
are topological invariants when $X$ is a PL-pseudomanifold. This completes the result of \cite{MR761809}
that the Steenrod squares are topological invariants, as homomorphisms
$H^r_{\ov{p}}(X;\F_{2})\to H_{2\ov{p}}^{r+i}(X;\F_{2})$.
The proof is combinatorial, using  the  description of Steenrod squares made by Steenrod in 
\cite{MR0022071}.

We emphasize now some  particularities which are important in the process of the proof of the Goresky-Pardon conjecture.
The main point is that our technique allows an explicit construction of the ${\rm cup}_{i}$-products at the level of cochain complexes, without requiring the derived category for their definition. 
In the context of filtered objects, observe first that the notion of filtered singular simplices is a natural one,
see \remref{rem:ffspseudomanifold}. 

The second modus operandi is the blow-up of these simplices. 
 In differential geometry, a blow-up is the replacement of a sub manifold $N$ of a manifold $M$ by the boundary of a tubular neighborhood of $N$ in $M$.
Its simplicial version can be illustrated as follows in the case of
$\Delta=\Delta^{j_{0}}\ast\Delta^{j_{1}}$: we cut off a small open neighborhood of $\Delta^{j_{0}}$ in $\Delta$ to get 
$\widetilde{\Delta}=c\Delta^{j_{0}}\times \Delta^{j_{1}}$. For instance, 

%=========================================================================
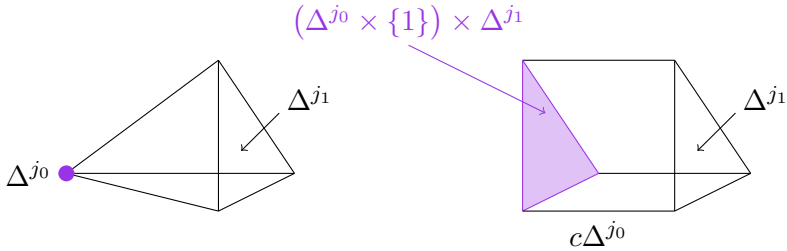
\definecolor{zzttqq}{rgb}{0.6,0.2,0.9}

\makebox{}

\medskip

\begin{tikzpicture}
\draw [color=black] (0,0)-- (1,0.5);
\draw [color=black] (1,0.5)--  (0,2);
\draw [color=black]  (0,2)-- (0,0);
\draw [color=black]  (1,0.5)-- (-2,0.5);
\draw [color=black]  (0,0)-- (-2,0.5);
\draw [color=black]  (0,2)-- (-2,0.5);
\fill [color=zzttqq] (-2,0.5) circle (3pt);
\draw[color=black] (-2.5,0.5) node {$\Delta^{j_0}$};
\draw[color=black] (-0.5,-1.2) node {$\Delta = \Delta^{j_0} *\Delta^{j_1}$  };
\draw[color=black] (2.5,-1.3) node {has for blow-up };
\draw[color=black] (5.5,-1.2) node {$\widetilde\Delta = c \Delta^{j_0} \times  \Delta^{j_1}$. };
\draw[color=black] (1.2,1.5) node {$\Delta^{j_1}$};
\draw[color=zzttqq] (2.5,2.5) node {$\left(\Delta^{j_0} \times \{1\} \right)\times \Delta^{j_1} $};
\draw[->] (0.8,1.3) -- (0.3,0.8);
\draw[color=zzttqq, ->] (2.5,2.2) -- (4.3,1.3);
\draw [color=black] (6,0)-- (7,0.5);
\draw [color=black] (7,0.5)--  (6,2);
\draw [color=black]  (6,2)-- (6,0);
\fill[color=red,fill=zzttqq,fill opacity=0.3] (4,0) -- (5,0.5) -- (4,2) -- (4,0) -- cycle;
\draw [color=zzttqq] (4,0)-- (5,0.5);
\draw [color=black] (5,0.5)--  (7,0.5);
\draw [color=black] (6,2)--  (4,2);
\draw [color=zzttqq]  (4,2)-- (4,0);
\draw [color=black]  (6,0)-- (4,0);
\draw [color=zzttqq]  (5,0.5)-- (4,2);
\draw[color=black] (7.2,1.5) node {$\Delta^{j_1}$};
\draw[->] (6.8,1.3) -- (6.3,0.8);
\draw[color=black] (5,-0.3) node {$c\Delta^{j_0}$};
\end{tikzpicture}

%=========================================================================

In the general case of $\Delta=\Delta^{j_{0}}\ast\cdots\ast\Delta^{j_{n}}$, we use an inductive process which consists in cutting off a small open neighborhood of the smallest stratum. As an illustration, 

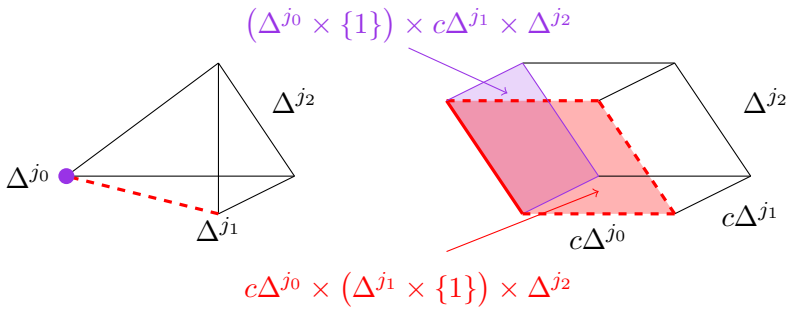
\vspace{1cm}

\begin{tikzpicture}
\draw [color=black] (0,0)-- (1,0.5);
\draw [color=black] (1,0.5)--  (0,2);
\draw [color=black]  (0,2)-- (0,0);
\draw [color=black]  (1,0.5)-- (-2,0.5);
\draw [color=red,very thick,dashed]  (0,0)-- (-2,0.5);
\draw [color=black]  (0,2)-- (-2,0.5);
\fill [color=zzttqq] (-2,0.5) circle (3pt);
\draw[color=black] (-2.5,0.5) node {$\Delta^{j_0}$};
\draw[color=black] (0,-0.2) node {$\Delta^{j_1}$};
\draw[color=black] (-0.5,-2) node {$\Delta = \Delta^{j_0} *\Delta^{j_1} *\Delta^{j_2}$  };
\draw[color=black] (3,-2.1) node {has for blow-up };
\draw[color=black] (6.5,-2) node {$\widetilde\Delta = c \Delta^{j_0} \times c\Delta^{j_1} \times \Delta^{j_2}$. };
\draw[color=black] (1,1.5) node {$\Delta^{j_2}$};
\draw[color=zzttqq] (2.5,2.5) node {$\left(\Delta^{j_0} \times \{1\} \right)\times c\Delta^{j_1} \times \Delta^{j_2}$};
\draw[color=zzttqq, ->] (2.5,2.2) -- (3.8,1.6);
\draw[color=red] (2.5,-1) node {$ c\Delta^{j_0}  \times \left(\Delta^{j_1} \times \{1\} \right)\times \Delta^{j_2}$};
\draw[color=red, ->] (3,-0.5) -- (5,0.3);
\draw [color=black] (6,0)-- (7,0.5);
\draw [color=black] (7,0.5)--  (6,2);
\draw [color=black]  (5,1.5)-- (6,2);
\fill[color=black,fill=zzttqq,fill opacity=0.2,dashed] (4,0) -- (5,0.5) -- (4,2) -- (3,1.5) -- (4,0) -- cycle;
\fill[color=black,fill=red,fill opacity=0.3] (4,0) -- (6,0)-- (5,1.5)  -- (3,1.5) -- (4,0) -- cycle;
\draw [color=zzttqq] (4,0)-- (5,0.5);
\draw [color=black] (5,0.5)--  (7,0.5);
\draw [color=black] (6,2)--  (4,2);
\draw [color=red,very thick]  (4,0)-- (3,1.5);
\draw [color=red,very thick,dashed]  (6,0)-- (4,0);
\draw [color=red,very thick,dashed]  (3,1.5)-- (5,1.5);
\draw [color=zzttqq]  (3,1.5)-- (4,2);
\draw [color=red,very thick,dashed] (6,0 )--  (5,1.5);
\draw [color=zzttqq]  (5,0.5)-- (4,2);
\draw[color=black] (7.2,1.5) node {$\Delta^{j_2}$};
\draw[color=black] (5,-0.3) node {$c\Delta^{j_0}$};
\draw[color=black] (7,0) node {$c\Delta^{j_1}$};
\end{tikzpicture}

The faces containing $\Delta^{j_{i}}\times\{1\}$ as a factor, which play a fundamental role in the definition of the perverse degree (see \defref{def:tranversedegree}), have been shadowed in the previous drawings.

The motivation for such process occurs when one determines the intersection cohomology of a pseudomanifold with differential forms : as these forms cannot be defined on the singular strata, the only possibility is to  define them on the regular part and ask for some control in the neighborhood of strata. That is what we do here for cochains.
As observed in \cite{MR3046315}, by G.~Friedman and J.E.~McClure, the classical way for the definition of a cup-product (with back and front faces) does not fit with perverse degrees. But, one advantage of the blow-up is that we can define the cup-product (and more generally the ${\rm cup}_{i}$-products) stratum after stratum, on each factor of the product 
$c\Delta^{j_{0}}\times \cdots\times c\Delta^{j_{n-1}}\times \Delta^{j_{n}}$,
 from the classical definition and in a compatible way with the perverse structure. 
Finally,  this procedure reveals itself of an easy use and does not lose any information in cohomology; it gives the same structure on cohomology as Goresky's definition, as it is established in \secref{sec:goresky}.

\tableofcontents
 In \secref{sec:perversity}, we recall  basic notions concerning \ffss~and their intersection cohomology. \secref{sec:E2algebras} is devoted to the construction
of a structure of perverse $\cE(2)$-algebra on the blow-up, $\tN^*(X)$, which corresponds to the building of ${\rm cup}_{i}$-products. In \secref{sec:steenrod}, we establish the main properties of perverse Steenrod squares, including the proof of the perverse range conjecture of M.~Goresky and W.~Pardon. The comparison between our definition and Goresky's definition of Steenrod squares, in the case of a pseudomanifold, is done in \secref{sec:goresky}. The particular case of isolated singularities and the treatment of Steenrod squares in the intersection cohomology of a Thom space are presented in \secref{sec:isolated}. 
An example  of a square, $\sq^2$, in the intersection cohomology of the total space of a fibration whose fiber is a cone is given in 
\secref{sec:the example}. This example shows the interest of having a range of perversity in $\cL(\ov{p},i)$ instead of $2\ov{p}$.
Finally, \secref{sec:topinvariance} is devoted to the topological invariance of our Steenrod squares.

All the cohomology groups appearing in this text are over the field with two elements, $\F_{2}$. If 
there is no ambiguity, we simplify the notation $H^*(X;\F_{2})$ in $H^*(X)$.

We thank the anonymous referee for her/his comments and suggestions which have contributed to improve the organization and the writing.

%%%%%%%%%%%%%%%
\section{Blow-up and perversity}\label{sec:perversity}

In this section, we recall the basics of a simplicial version of intersection cohomology, already introduced in \cite{2012arXiv1205.7057C}.

Let $\Delta^k$ be the standard simplex of $\R^{k+1}$, whose vertices, $v_{0},\ldots,v_{k}$, verify $v_{i}=(t_{0},\ldots,t_{k})$, with $t_{j}=0$ if $j\neq i$ and $t_{i}=1$. Let $\delta_{i}\colon\{0,1,\ldots,k-1\}\to \{0,1,\ldots,k\}$
defined by
$$\delta_{i}(j)=\left\{
\begin{array}{lcl}
j&\text{ if }&j<i,\\
j+1&\text{ if }&j\geq i.
\end{array}\right.$$
Such maps generate linear applications, still denoted
$\delta_{i}\colon \Delta^{k-1}\to\Delta^k$ and defined by
$\delta_{i}(v_{j})=v_{\delta_{i}(j)}$. More generally, any map
$\sigma\colon \{0,1,\ldots,\ell\}\to \{0,1,\ldots,k\}$
generates a linear application 
$\sigma\colon \Delta^{\ell}\to \Delta^k$.

\medskip
 \emph{We fix an integer $n$} 
 and consider the category  ${\pmb\Delta}^{[n]}_\cF$ whose
\begin{itemize}
\item objects are the joins, $\Delta=\Delta^{j_0}\ast\Delta^{j_1}\ast\cdots\ast\Delta^{j_n}$, where $\Delta^{j_i}$  is the simplex of dimension $j_i$, possibly empty, with the conventions $\Delta^{-1}=\emptyset$ and  $\emptyset\ast X=X$,
\item maps are the $\sigma\colon \Delta=\Delta^{j_0}\ast\Delta^{j_1}\ast\cdots\ast\Delta^{j_n}\to\Delta'=\Delta^{k_0}\ast\Delta^{k_1}\ast\cdots\ast\Delta^{k_n}$, 
of the shape $\sigma=\ast_{i=0}^n\sigma_{i}$, with 
$\sigma_{i}\colon \{0,1,\ldots,j_{i}\}\to \{0,1,\ldots,k_{i}\}$
 an injective order-preserving map for each~$i$.
\end{itemize}

The category
${\pmb\Delta}^{[n],+}_\cF$ is the full subcategory of ${\pmb\Delta}^{[n]}_\cF$ whose objects are the joins $\Delta^{j_0}\ast\Delta^{j_1}\ast\cdots\ast\Delta^{j_n}$ with
$\Delta^{j_n}\neq\emptyset$, i.e., $j_n\geq 0$.
To any such element, we associate its \emph{blow-up} which is the map$$\mu\colon \widetilde{\Delta}=c\Delta^{j_0}\times\cdots\times c\Delta^{j_{n-1}}\times \Delta^{j_n}\to
\Delta=\Delta^{j_0}\ast\cdots\ast\Delta^{j_n},$$
defined by
\begin{eqnarray*}
\mu([y_0,s_0],\ldots,[y_{n-1},s_{n-1}],y_n)&=&
s_0y_0+(1-s_0)s_1y_1+\cdots\\
&&
+(1-s_0)\cdots (1-s_{n-2})s_{n-1}y_{n-1}\\&&
+(1-s_0)\cdots (1-s_{n-2})(1-s_{n-1})y_n,
\end{eqnarray*}
where $y_i\in\Delta^{j_i}$ and $[y_i,s_i]\in c\Delta^{j_i}=(\Delta^{j_{i}}\times [0,1])/(\Delta^{j_{i}}\times \{0\})$. 
The prism $\widetilde{\Delta}$ is sometimes also called the blow-up of~$\Delta$.

Observe that this blow-up is well defined thanks to the restriction to the subcategory ${\pmb\Delta}^{[n],+}_\cF$. 
In the topological setting (see \remref{rem:ffspseudomanifold}) 
this restriction means that we do not consider simplices entirely included in the singular part.

\begin{definition}\label{def:filteredfaceset}
A \emph{filtered face set,} of formal dimension $n$, is a contravariant functor, $\ob{K}$, from the category ${\pmb\Delta}^{[n]}_\cF$ to the category of sets, i.e., $(j_0,\ldots,j_n)\mapsto \ob{K}_{(j_0,\ldots,j_n)}$.
The restriction of the filtered face set, $\ob{K}$, to ${\pmb\Delta}^{[n],+}_\cF$ is denoted $\ob{K}_+$. 

If $\ob{K}$ and $\ob{K}'$ are filtered face sets, a \emph{filtered face map,} $\ob{f}\colon\ob{K}\to\ob{K}'$, is a natural transformation between the two functors $\ob{K}$ and $\ob{K}'$.
We denote by 
$\fil$
the category of filtered face sets.
\end{definition}

To any simplicial set, $Y$, we can associate the $\F_{2}$-vector space $C_{d}(Y)$ generated by the $d$-dimensional
simplices of $Y$. The normalized chain complex, $N_{d}(Y)$, is the quotient of $C_{d}(Y)$ by the degeneracies $\mathfrak s_{i}$, 
$$N_{d}(Y)=C_{d}(Y)/{\mathfrak s}_{0}C_{d-1}(Y)+\cdots+{\mathfrak s}_{d-1}C_{d-1}(Y).$$
We consider also the duals
$N^*(Y)=\hom_{\F_{2}}(N_{*}(Y),\F_{2})$ and
$C^*(Y)=\hom_{\F_{2}}(C_{*}(Y),\F_{2})$.

\smallskip
Any face operator,
$\delta_{i}\colon\Delta^{j_{\ell}}\to \Delta^{j_{\ell}+1}$,
for some $\ell\in\{0,\ldots,n-1\}$,
induces a chain map,
$\delta_{i}^*\colon 
N^*(c\Delta^{j_0})\otimes\cdots\otimes N^*(c\Delta^{j_{\ell}+1})\otimes\cdots\otimes N^*(\Delta^{j_n})
\to
N^*(c\Delta^{j_0})\otimes\cdots\otimes N^*(c\Delta^{j_{\ell}})\otimes\cdots\otimes N^*(\Delta^{j_n})$, defined by the identity on the factors in $\Delta^{j_{i}}$ for $i\neq\ell$. 

We denote also by 
$\delta_{i}\colon 
\Delta^{j_{0}}\ast\cdots\ast\Delta^{j_{\ell}}\ast\cdots\ast \Delta^{j_{n}}
\to
\Delta^{j_{0}}\ast\cdots\ast\Delta^{j_{\ell}+1}\ast\cdots\ast \Delta^{j_{n}}$
the operator defined by $\delta_{i}\colon\Delta^{j_{\ell}}\to \Delta^{j_{\ell}+1}$ and the identity maps.
For any simplex,\linebreak
 $\sigma\colon \Delta^{j_{0}}\ast\cdots\ast\Delta^{j_{\ell}+1}\ast\cdots\ast\Delta^{j_{n}}\to \ob{K}_{+}$, we define a  simplex,
$\partial_{i}\sigma\colon \Delta^{j_{0}}\ast\cdots\ast\Delta^{j_{\ell}}\ast\cdots\ast\Delta^{j_{n}}\to \ob{K}$,
by $\partial_{i}\sigma=\sigma\circ\delta_{i}$, and a complex,
$$\tN^*_\sigma=N^*(c\Delta^{j_0})\otimes\cdots\otimes N^*(c\Delta^{j_{n-1}})\otimes N^*(\Delta^{j_n}).$$
These previous considerations on face operators can easily be adapted to the case $\ell=n$.

\smallskip
A \emph{global section (or {cochain})} on $\ob{K}$ is a function which assigns to each simplex $\sigma\in\ob{K}_{+}$ an element $c_\sigma\in \tN^*_\sigma$ such that
$c_{\partial_i\sigma}=\delta_{i}^*(c_\sigma)$ for all $\sigma\in\ob{K}_{+}$  and all $\delta_{i}\in{\pmb\Delta}^{[n],+}_\cF$.
(The restriction to
${\pmb\Delta}^{[n],+}_\cF$
implies $\Delta^{j_{n}}\neq \emptyset$.)

The space of global sections is denoted by $\tN^*(\ob{K})$  and  called  \emph{the blow-up}  of $N^*$ over the \ffs~$\ob{K}$. Global sections have an extra degree, called the \emph{perverse degree}, that we describe now. 

\medskip
Let $\sigma\colon\Delta^{j_0}\ast\cdots\ast \Delta^{j_n}\to\ob{K}_{+}$ and
$\ell\in \{1,\ldots,n\}$ such that  $\Delta^{j_{n-\ell}}\neq\emptyset$. 
For any cochain $c_{\sigma}\in 
N^*(c\Delta^{j_0})\otimes\cdots\otimes N^*(c\Delta^{j_{n-1}})\otimes N^*(\Delta^{j_n})$,
its restriction
\begin{equation}\label{equa:blowup}
c_{\sigma,n-\ell}\in 
N^*(c\Delta^{j_0})\otimes\cdots\otimes N^*(\Delta^{j_{n-\ell}}\times\{1\})\otimes\cdots\otimes N^*(c\Delta^{j_{n-1}})\otimes N^*(\Delta^{j_n})
\end{equation}
can be written
$c_{\sigma,n-\ell}=\sum_k c'_{\sigma,n-\ell}(k)\otimes c''_{\sigma,n-\ell}(k)
 $, with 
 \begin{itemize}
\item $c'_{\sigma,n-\ell}(k)\in N^*(c\Delta^{j_0})\otimes\cdots\otimes N^*(c\Delta^{j_{n-\ell-1}})\otimes
N^*(\Delta^{j_{n-\ell}}\times\{1\})$ and
\item $c''_{\sigma,n-\ell}(k)\in  N^*(c\Delta^{j_{n-\ell+1}})\otimes\cdots\otimes N^*(\Delta^{j_n})$.
\end{itemize}
Observe that each term of the tensor product in Formula~(\ref{equa:blowup}) has a finite canonical basis and the decomposition of $c_{\sigma,n-\ell}$ can be canonically chosen in function of the associated basis of the tensor product.

\begin{definition}\label{def:tranversedegree}
If $c_{\sigma,n-\ell}\neq 0$, the \emph{$\ell$-perverse degree,} $\|c_{\sigma}\|_{\ell}$, of  
$c_{\sigma}$ is equal to 
$$\|c_\sigma\|_{\ell}=\sup_k\left\{ |c''_{\sigma,n-\ell}(k)|\text{ such that } c'_{\sigma,n-\ell}(k)\neq 0\right\},$$
where $|c''_{\sigma,n-\ell}(k)|$ denotes the usual degree of the cochain $c''_{\sigma,n-\ell}(k)$.

If $c_{\sigma,n-\ell}= 0$ or $\Delta^{j_{n-\ell}}=\emptyset$, we set
$\|c_{\sigma}\|_{\ell}=-\infty$.

The
  \emph{perverse degree}
 of a  global section $c\in \tN^*({\ob{K}})$  
is the $n$-tuple
$$\|c\|=(\|c\|_1,\ldots,\|c\|_{n}),$$  
where $\|c\|_\ell$ is the supremum of the $\|c_{{\sigma}}\|_\ell$ for all ${\sigma}\in {\ob{K}_{+}}$.
\end{definition}

Intersection cohomology requires a notion of perversity that we introduce now, following the convention of \cite{MR800845}.

\begin{definition}
A \emph{loose perversity} is a map $\ov{p}\colon \N\to\Z$, $i\mapsto \ov{p}(i)$, such that $\ov{p}(0)=0$. 
A \emph{perversity} is a loose perversity such that
$\ov{p}(i)\leq\ov{p}(i+1)\leq\ov{p}(i)+1$, for all $i\in \N$.
A \emph{Goresky-MacPherson perversity} (or \emph{GM-perversity}) is a perversity such that $\ov{p}(1)=\ov{p}(2)=0$.
\end{definition}

If $\ov{p}_1$ and $\ov{p}_2$ are two loose perversities, we set $\ov{p}_1\leq \ov{p}_2$ if we have $\ov{p}_1(i)\leq\ov{p}_2(i)$, for all $i\in\N$.  The poset of all loose perversities is denoted $\cPloose$. 

The lattice of GM-perversities, denoted $\cP^n$, admits a maximal element, $\ov{t}$, called the \emph{top perversity} and defined by $\ov{t}(i)=i-2$, if $i\geq 2$, $\ov{t}(0)=\ov{t}(1)=0$.

To these posets, we  add an element, $\ov{\infty}$, which is the constant map to $\infty$. We call it the \emph{infinite perversity} despite the fact that it is not a perversity in the sense of the previous definition. Finally, we set
$\hat{\cP}^n=\cP^n\cup\{\ov{\infty}\}$
and
$\hat{\cP}^n_{{\rm loose}}=\cPloose\cup\{\ov{\infty}\}$.

\begin{definition}\label{def:localsystemadmissible}
Let $\ov{p}$ be a  loose perversity.  A global section $c\in \tN^*({\ob{K}})$  is \emph{$\ov{p}$-admissible} if 
$\|c\|_{i}\leq \ov{p}(i)$, for any $i\in\{1,\ldots,n\}$. A  global section $c$ is of \emph{$\ov{p}$-intersection} if $c$ and its differential, $\delta c$, are $\ov{p}$-admissible. 

We denote by $\tN^*_{\ov{p}}({\ob{K}})$ the complex 
 of  global sections of $\ov{p}$-intersection and by 
 $H_{\ov{p}}^*({\ob{K}};\tN)$  its homology.
\end{definition} 

By using the same process with $C^*$ in place of $N^*$, we obtain a second complex of global sections of $\ov{p}$-intersection,
$\tC^*_{\ov{p}}({\ob{K}})$, of homology  $H_{\ov{p}}^*({\ob{K}};\tC)$. Directly from 
\cite[Theorem A]{2012arXiv1205.7057C},
we get an isomorphism between these two cohomologies.

\begin{proposition}\label{prop:CandN}
Let $\ob{K}$ be a \ffs~and $\ov{p}$ be a loose perversity. The canonical surjection, $C_{*}(-)\to N_{*}(-)$, induces a quasi-isomorphism,
$\tN^*_{\ov{p}}(\ob{K})\to \tC^*_{\ov{p}}(\ob{K})$,
 and therefore an isomorphism
$H_{\ov{p}}^*({\ob{K}};\tN)\cong H_{\ov{p}}^*({\ob{K}};\tC)$.
\end{proposition}

If there is no ambiguity, we denote by $H_{\TW,\ov{p}}^*({\ob{K}})$ this common value and called it the \emph{Thom-Whitney cohomology} (henceforth TW-cohomology) of ${\ob{K}}$ with coefficients in $\F_{2}$ for the loose perversity $\ov{p}$.

The topological objects corresponding to the \ffss~are locally conical, stratified topological spaces. We consider here
only the case of pseudomanifolds defined as follows.

\begin{definition}\label{def:pseudo}
An \emph{$n$-dimensional topological pseudomanifold} is a nonempty topological space with a filtration by closed subsets,
$$\emptyset=X_{-1}\subseteq X_{0}\subseteq \cdots\subseteq X_{n-2}=X_{n-1}\subsetneqq X_{n}=X,$$
such that, for all $i$, 
$X_i\backslash X_{i-1}$ is an $i$-dimensional metrizable topological  manifold or the empty set. Moreover, for each point  $x \in X_i \backslash X_{i-1}$, $i\neq n$, there exist,
\begin{enumerate}[(a)]
\item an open neighborhood, $V$, of $x$ in $X$, endowed with the induced filtration,
\item an open neighborhood, $U$, of $x$ in  $X_i\backslash X_{i-1}$, 
\item a compact topological pseudomanifold, $L=(L_{j})_{0\leq j\leq n-i-1}$,  of dimension $n-i-1$, whose cone, $\mathring{c}L=(L\times [0,1[)/(L\times\{0\})$, is endowed with the conic filtration, i.e., $(\mathring{c}L)_{i}=\mathring{c}L_{i-1}$, for $i\geq 0$,% 
\item a   homeomorphism, $\varphi \colon U \times \mathring{c}L\to V$, 
such that
\begin{enumerate}[(1)]
\item $\varphi(u,\tv)=u$, for any $u\in U$, with $\tv$  the cone point,
\item $\varphi(U\times \mathring{c}L_j)=V\cap X_{i+j+1}$, for any $j\in \{0,\ldots,n-i-1\}$.
\end{enumerate}
\end{enumerate}
The  couple $(V,\varphi)$ is called a \emph{conic chart} of $x$
 and the filtered  space, $L$, the \emph{link} of $x$. 
\end{definition}

This definition makes sense with an induction on the dimension, starting from pseudomanifolds of dimension 0 which are discrete topological spaces, by definition. 
Also, one can prove that the subspace
$X_{n}\backslash X_{n-2}$ is dense.

\begin{remark}\label{rem:ffspseudomanifold}
The set of filtered singular simplices is the  bridge between pseudomanifolds and the more general notion of \ffss. More precisely, for any pseudomanifold, $X$, we  define (see \cite[Example 1.5]{2012arXiv1205.7057C}) the singular \ffs~by
$$\ob{\rm ISing}^{\cF}(X)_{j_0,\ldots,j_n}=\{\sigma\colon \Delta^{j_0}\ast\cdots\ast \Delta^{j_n}\to X \mid \sigma^{-1}X_{i}= \Delta^{j_0}\ast\cdots\ast \Delta^{j_i}\}.$$
Such simplex is called \emph{filtered.}
\end{remark}

 If $X$ is a pseudomanifold and $\ob{K}=\ob{\rm ISing}^{\cF}(X)$, we use  the notations $\tN^*_{\ov{p}}(X)$, $\tC^*_{\ov{p}}(X)$ and $H^*_{\TW,\ov{p}}(X)$ for the Thom-Whitney complexes and their cohomology. (As $X_{n-1}=X_{n-2}$, the case $i=1$ in \defref{def:localsystemadmissible}, is vacuous in this setting.) 
 
 \medskip
 We end this section with a reminder of Goresky-MacPherson cohomology (with coefficients in $\F_{2}$) and its link
 with the blow-up. Let $\ov{p}$ be a loose perversity. A filtered simplex,
 $\sigma\colon\Delta=\Delta^{j_{0}}\ast\cdots\ast \Delta^{j_{n}}\to X$,
 has a perverse degree, $\|\sigma\|=(\|\sigma\|_{0},\ldots,\|\sigma\|_{n})$,
 where $\|\sigma\|_{\ell}=\dim (\Delta^{j_{0}}\ast\cdots\ast\Delta^{j_{n-\ell}})$,
 with the convention $\|\sigma\|_{\ell}=-\infty$ if $\sigma^{-1} X_{n-\ell}=\emptyset$.

 A $\ov{p}$-admissible simplex of $X$ is a filtered simplex,
 $\sigma\colon\Delta=\Delta^{j_{0}}\ast\cdots\ast \Delta^{j_{n}}\to X$,
 such that
 $\|\sigma\|_{\ell}\leq \dim\Delta-\ell+\ov{p}(\ell)$,
 for any $\ell\in\{1,\ldots,n\}$. 
 A $\ov{p}$-admissible chain is a linear combination of $\ov{p}$-admissible simplices. A chain, $c$, is of $\ov{p}$-intersection if $c$ and its boundary, $\partial c$, are $\ov{p}$-admissible. Denote by
 $C_{*}^{\GM,\ov{p}}(X)$,   $N_{*}^{\GM,\ov{p}}(X)$ the complexes of $\ov{p}$-intersection chains, by
 $C^*_{\GM,\ov{p}}(X)=\hom(C_{*}^{\GM,\ov{p}}(X),\F_{2})$,
 $N^*_{\GM,\ov{p}}(X)=\hom(N_{*}^{\GM,\ov{p}}(X),\F_{2})$
 their dual and by
 $H^*_{\GM,\ov{p}}(X)=H(C^*_{\GM,\ov{p}}(X))=H(N^*_{\GM,\ov{p}}(X))$
 their homology, called
the Goresky-MacPherson (\cite{MR572580}) intersection cohomology of $X$ (henceforth GM-cohomology)
 with coefficients in $\F_{2}$. 
This cohomology is isomorphic to the original Goresky-MacPherson cohomology in the case of a pseudomanifold, $X$, 
and a GM-perversity, $\ov{p}$, see \cite[Proposition A.29]{2012arXiv1205.7057C} and \cite{MR800845}.

\smallskip
The GM and TW cohomologies are related in \cite[Theorem B]{2012arXiv1205.7057C} that we recall here.

\begin{proposition}\label{prop:GMetTW}
Let $X$ be a pseudomanifold, $\ov{p}$ and $\ov{q}$ be two perversities such that $\ov{q}\geq 0$ and $\ov{p}(i)+\ov{q}(i)=i-2$ for any $i\in\{2,\ldots,n\}$. Then there is an isomorphism between the GM and the TW cohomologies,
$H^*_{\TW,\ov{q}}(X)\cong H^*_{\GM,\ov{p}}(X)$.
 \end{proposition} 

%%%%%%%%%%%%%%%%%%
\section{Perverse $\cE(2)$-algebras and \ffss}\label{sec:E2algebras}

Steenrod  squares are built from an action of a normalized homogeneous Bar resolution, $\cE(2)$, of the symmetric group $\Sigma_{2}$, on the normalized singular cochains. 
This is the way the non-commutativity of the cup-product is controlled up to higher coherent homotopies.
This action enriches the multiplicative structure given by the cup-product.
We first review it in order to adapt this construction to the perverse setting.

 Recall that the resolution $\cE(2)$ of $\F_{2}$ as $\Sigma_{2}$-module is defined by
$$\ldots\rightarrow \cE(2)_{i}\stackrel{d}{\rightarrow} \cE(2)_{i-1}\to\cdots$$
with $\cE(2)_{i}=\F_{2}(e_{i},\tau_{i})$, $d e_{i}=d\tau_{i}=e_{i-1}+\tau_{i-1}$.
(As we are using cochain complexes, $\cE(2)$ is negatively graded.)
From the isomorphism
$\Sigma_{2}\cong \{e_{i},\,\tau_{i}\}$ with $\tau_{i}$ the generator of $\Sigma_{2}$,
the action (on the left) of $\Sigma_{2}$ defines a natural action on $\cE(2)$.
This action is extended to the tensor product $\cE(2)\otimes\cE(2)$ as a diagonal action.
 Moreover, the
complex $\cE(2)$ is equipped with a $\Sigma_{2}$-equivariant diagonal,
$\cD\colon \cE(2)\to \cE(2)\otimes \cE(2)$, defined by
$$\cD(e_{i})=\sum_{j=0}^i e_{j}\otimes \tau^j.e_{i-j},$$
with $\tau.e_{k}=\tau_{k}$, $\tau.\tau_{k}=e_{k}$.
 This diagonal is essential for the definition of the structure of $\cE(2)$-algebra on $\tN^*(\ob{K})$.
 Finally, observe that, for any vector space $V$, there is a $\Sigma_{2}$-action on $\hom_{\F_{2}}(V^{\otimes 2},V)$, defined by $(\tau.f)(v_{1}\otimes v_{2})=f(v_{2}\otimes v_{1})$.
 
\begin{definition}\label{def:E2algebra}
An \emph{$\cE(2)$-algebra structure} on a cochain complex, $A^*$, is a cochain map,
$\psi \colon \cE(2)\otimes A^{\otimes 2}\to A$, which is $\Sigma_{2}$-equivariant as map from $\cE(2)$ to $\hom_{\F_{2}}(A^{\otimes 2},A)$.
\end{definition}
If we denote $\psi(e_{i}\otimes x_{1}\otimes x_{2})$ by $ x_{1}\cup_{i}x_{2}$, the previous definition is equivalent to
\begin{enumerate}
\item $\psi(\tau_{i}\otimes x_{1}\otimes x_{2})=\psi(e_{i}\otimes x_{2}\otimes x_{1})=x_{2}\cup_{i}x_{1}$,
\item together with the Leibniz condition:\\
$\delta(x_{1}\cup_{i} x_{2})= x_{1}\cup_{i-1}x_{2}+x_{2}\cup_{{i-1}}x_{1}+ \delta x_{1}\cup_{i}x_{2}+x_{1}\cup_{i}\delta x_{2}$.
\end{enumerate}
This means that \emph{an $\cE(2)$-algebra structure is given by a cochain map, called $\text{cup}_{i}$-product, $\cup_{i} \colon A^r\otimes A^s\to A^{r+s-i}$, satisfying the previous Leibniz condition.}

Let $L$ be a simplicial set. 
In \cite{MR2075046}, C. Berger et B. Fresse prove the existence of an $\cE(2)$-action on the normalized 
cochain complex of $L$, i.e., the existence of a cochain map
$$\psi_{L}\colon \cE(2)\otimes N^*(L)^{\otimes^2}\to N^*(L),\;e_{i}\otimes x_{1}\otimes x_{2}\mapsto x_{1}\cup_{i}x_{2},$$
which satisfies the requirements of \defref{def:E2algebra}. 
As it is established by May  in \cite{MR0281196},  classical properties of ${\rm cup}_{i}$-products are 
a direct consequence of this $\cE(2)$-algebra structure, except two of them that we quote in the next definition.
(Mention that $N^*(L)$ satisfies these two additional properties, see \cite{MR2075046}.)

\begin{definition}\label{def:niceE2algebra}
An {$\cE(2)$-algebra,} $A^*$, is \emph{nice} if it verifies the two next properties, for all 
$x,\,x'\in A$ of respective degrees $|x|$ and $|x'|$,
\begin{enumerate}[(i)]
\item $x\cup_{|x|}x=x$,
\item  $x\cup_{i}x'=0$ if $i>\min(|x|,|x'|)$.
\end{enumerate}
\end{definition}

\noindent
Observe  the useful next property of nice $\cE(2)$-algebras.

\begin{lemma}\label{lem:petitlemmenice}
Let $A$ be a nice $\cE(2)$-algebra. If $a\in A^d$ and $b\in A^d$, we have
$$a\cup_{d}b=b\cup_{d}a.$$
\end{lemma}

\begin{proof}
Property (ii) of \defref{def:niceE2algebra} and Leibniz rule imply
\begin{eqnarray*}
\delta(a\cup_{d+1}b)
&=& 0\\
&=& a\cup_{d}b+b\cup_{d}a+\delta a\cup_{d+1}b+a\cup_{d+1}\delta b\\
&=& a\cup_{d}b+b\cup_{d}a.
\end{eqnarray*}
\end{proof}

We recall now from \cite{MR2075046} the construction of \emph{the tensor product of 
$\cE(2)$-algebras.}
Let $\psi_{i}\colon \cE(2)\otimes A_{i}^{\otimes^2}\to A_{i}$
be $\cE(2)$-algebras for $i=0,1$. 
We use the diagonal $\cD$ of $\cE(2)$ for the construction of an $\cE(2)$-action on the tensor product 
$A_{0}\otimes A_{1}$, as the following composite, denoted by $\Phi$,
$$\xymatrix{
\cE(2)\otimes (A_{0}\otimes A_{1})^{\otimes^2}\ar[r]^{Sh}&
 \cE(2)\otimes A_{0}^{\otimes^2}\otimes A_{1}^{\otimes^2}
\ar[rr]^-{\cD\otimes\id\otimes\id}&&
\cE(2)\otimes \cE(2)\otimes A_{0}^{\otimes^2}\otimes A_{1}^{\otimes^2}
\ar[d]^{Sh}\\
&A_{0}\otimes A_{1}
&& \cE(2)\otimes A_{0}^{\otimes^2}\otimes \cE(2)\otimes A_{1}^{\otimes^2},
\ar[ll]_-{\psi_{{0}}\otimes \psi_{{1}}}
}$$
where $Sh$ are appropriate shuffle maps. 
We have to verify that the map $\Phi$ satisfies the two conditions stated after \defref{def:E2algebra}. 
Assertion (2) is the compatibility with the differentials which is direct here, 
because $\Phi$ is the composite of maps that are compatible with the differentials.
Thus, we are reduced to Assertion (1). Recall from the definition of the diagonal of $\cE(2)$,
$$\cD(e_{i})=
\sum_{j=0}^ie_{j}\otimes \tau^j. e_{i-j} \text{ and }
\cD(\tau_{i})=
\sum_{j=0}^i\tau. e_{j}\otimes \tau^{j+1}. e_{i-j}=
\sum_{j=0}^i \tau_{j}\otimes \tau^j. \tau_{i-j}.$$
A computation from the definition of $\Phi$ gives,
\begin{eqnarray*}
\Phi(\tau_{i}\otimes a_{0}\otimes a_{1}\otimes b_{0}\otimes b_{1})
&=&
\sum_{j=0}^i
\psi_{0}(\tau_{j}\otimes a_{0}\otimes b_{0})\otimes
\psi_{1}(\tau^j.\tau_{i-j}\otimes a_{1}\otimes b_{1}),\\
\Phi(e_{i}\otimes b_{0}\otimes b_{1}\otimes a_{0}\otimes a_{1})
&=&
\sum_{j=0}^i
\psi_{0}(e_{j}\otimes b_{0}\otimes a_{0})\otimes
\psi_{1}(\tau^j.e_{i-j}\otimes b_{1}\otimes a_{1}).\\
\end{eqnarray*}
If $j$ is even, we have
$$\psi_{1}(\tau^j.e_{i-j}\otimes b_{1}\otimes a_{1})=
\psi_{1}(e_{i-j}\otimes b_{1}\otimes a_{1})=
\psi_{1}(\tau_{i-j}\otimes a_{1}\otimes b_{1})=
\psi_{1}(\tau^j.\tau_{i-j}\otimes a_{1}\otimes b_{1}).$$
A similar computation in the case $j$ odd gives
$$\Phi(\tau_{i}\otimes a_{0}\otimes a_{1}\otimes b_{0}\otimes b_{1})
=
\Phi(e_{i}\otimes b_{0}\otimes b_{1}\otimes a_{0}\otimes a_{1}).$$

\medskip
Consider now a family of $\cE(2)$-algebras, $\psi_{i}\colon \cE(2)\otimes A_{i}^{\otimes^2}\to A_{i}$, 
with $i=0,\ldots,n$. 
As $\cD\colon \cE(2)\to \cE(2)\otimes \cE(2)$ is the diagonal of a Bar resolution, it is a cochain map, coassociative
(\cite{MR2075046}) and we may iterate it as,
$$\cD^2(e_{i})=\sum_{j=0}^i\cD(e_{j})\otimes \tau^j.e_{i-j}=\sum_{j=0}^i \sum_{k=0}^j e_{k}\otimes \tau^k.e_{j-k}\otimes \tau^j.e_{i-j}.$$
If we set $i_{1}=k$, $i_{2}=j-k$, $i_{3}=i-j$, this last expression can be written as
$$\cD^2(e_{i})=
\sum_{(i_{1},i_{2},i_{3})\text{ with }i_{1}+i_{2}+i_{3}=i} e_{i_{1}}\otimes \tau^{i_{1}}.e_{i_{2}}\otimes \tau^{i_{1}+i_{2}}.e_{i_{3}}.
$$
More generally, an induction gives,
$$\cD^{n-1}(e_{i})=\sum_{(i_{1},\ldots,i_{n})\text{ with }i_{1}+\cdots+i_{n}=i} e_{i_{1}}\otimes
\tau^{i_{1}}.e_{i_{2}}\otimes\cdots\otimes \tau^{i_{1}+\cdots+i_{n-1}}.e_{i_{n}}.
$$
As in the previous case of two $\cE(2)$-algebras, the action of  $\cE(2)$ on $\otimes_{i=0}^n A_{i}$
 is obtained from appropriate shuffle maps and the iteration $\cD^{n-1}$ of the diagonal. 
 By using the notation in cup$_{i}$-products, this structure is defined as the map
\begin{equation}\label{equa:tensorE2}
\xymatrix{
\cE(2)\otimes (\otimes_{i=0}^n A_{i})^{\otimes^2}
\ar[r]^-{\Phi}&
\otimes_{i=0}^n A_{i}
}\end{equation}
which sends the element $e_{i}\otimes (\otimes_{i=0}^n x_{i}) \otimes (\otimes_{i=0}^n y_{i})$ to 
$$\sum_{(i_{1},\ldots,i_{n})\text{ with }i_{1}+\cdots+i_{n}=i} (x_{1}\cup_{i_{1}}y_{1})\otimes
(x_{2}\cup^{i_{1}}_{i_{2}}y_{2})\otimes
\cdots\otimes 
(x_{n}\cup^{i_{1}+\cdots+i_{n-1}}_{i_{n}}y_{n}),$$
where we set, for $j\geq 0$,
\begin{equation}\label{equa:taucup}
x\cup_{i}^jy=\left\{\begin{array}{ll}
x\cup_{i}y&\text{ if } j \text{ is even,}\\
y\cup_{i}x&\text{ if } j \text{ is odd.}
\end{array}\right.
\end{equation}
Up to shuffle maps, $\Phi$ is obtained as a composite and tensor product of equivariant cochain maps; thus it satisfies 
the requirements of \defref{def:E2algebra}. Moreover, as we establish below,  the tensor product of nice $\cE(2)$-algebras is a nice $\cE(2)$-algebra.

\begin{lemma}\label{lem:produitnice}
Any tensor product of  nice $\cE(2)$-algebras is a nice $\cE(2)$-algebra for the product structure coming from the diagonal  of $\cE(2)$.
\end{lemma}

\begin{proof}
By coassociativity of the diagonal of $\cE(2)$, it is sufficient to reduce the proof to the case of the tensor product of two nice $\cE(2)$-algebras, $A$ and~$B$.

Let $x=\sum_{k} a_{k}\otimes b_{k}\in (A\otimes B)^d$
and
$x'=\sum_{\ell} a'_{\ell}\otimes b'_{\ell}\in (A\otimes B)^{d'}$ with $d\leq d'$. We set $f=d+m$ with $m\geq 0$. One compute
$$x\cup_{f}x'=\sum_{f_{1}+f_{2}=f}\sum_{k,\ell} (a_{k}\cup_{f_{1}}a'_{\ell})
\otimes
(b_{k}\cup^{f_{1}}_{f_{2}}b'_{\ell}).$$
If the element $(a_{k}\cup_{f_{1}}a'_{\ell})
\otimes
(b_{k}\cup^{f_{1}}_{f_{2}}b'_{\ell})$ of this sum is not equal to zero, we must have
$$f_{1}\leq \min(|a_{k}|,|a'_{\ell}|) \text{ and }
f_{2} \leq \min (|b_{k}|,|b'_{\ell}|),$$
which implies 
$f=f_{1}+f_{2}=d+m\leq |a_{k}|+|b_{k}|=d$ and $m=0$. We have established Property (ii) of \defref{def:niceE2algebra}.
As for Property (i), we consider
$$x\cup_{d}x=\sum_{f_{1}+f_{2}=d}\sum_{k,k'} (a_{k}\cup_{f_{1}}a_{k'})\otimes (b_{k}\cup^{f_{1}}_{f_{2}}b_{k'}).$$
As above, if the element $(a_{k}\cup_{f_{1}}a_{k'})
\otimes
(b_{k}\cup^{f_{1}}_{f_{2}}b_{k'})$ of this sum is not equal to zero, we must have
$$f_{1}\leq \min(|a_{k}|,|a_{k'}|) \text{ and }
f_{2} \leq \min (|b_{k}|,|b_{k'}|).$$
Suppose $\min(|a_{k}|,|a_{k'}|)=|a_{k}|$, then we have $|b_{k'}|\leq |b_{k}|$, because $|a_{k}|+|b_{k}|=|a_{k'}|+|b_{k'}|$, and also
$d=|a_{k}|+|b_{k}|= f_{1}+f_{2}\leq |a_{k}|+|b_{k'}|$, which imply
$|b_{k}|=|b_{k'}|$. Therefore, the non-zero elements of this sum must be of the shape
$(a_{k}\cup_{d-r}a_{k'})\otimes (b_{k}\cup^{d-r}_{r}b_{k'})$ 
with
$|a_{k}|=|a_{k'}|=d-r$, $|b_{k}|=|b_{k'}|=r$.
With \lemref{lem:petitlemmenice}, if $a_{k}\neq a_{k'}$, the same term appears twice, as
 $(a_{k}\cup_{d-r}a_{k'})\otimes (b_{k}\cup^{d-r}_{r}b_{k'})$ 
and as
$(a_{k'}\cup_{d-r}a_{k})\otimes (b_{k'}\cup^{d-r}_{r}b_{k})$. Their sum is equal to zero. 
With the same argument applied to the case $b_{k}\neq b_{k'}$, we have reduced the previous expression to
\begin{eqnarray*}
x\cup_{d}x
&=&
\sum_{k}(a_{k}\cup_{d-r}a_{k})\otimes (b_{k}\cup_{r}b_{k})\\
&=&
\sum_{k}a_{k}\otimes b_{k}=x,
\end{eqnarray*}
and Property (i) of \defref{def:niceE2algebra} is established.
\end{proof}

We come back to the intersection setting and recall (\cite{2012arXiv1205.7057C}) that a \emph{perverse cochain complex} is a functor defined on $\hat{\cP}^n$, with values in the category of cochain complexes. A functor from 
$\hat{\cP}^n_{{\rm loose}}$
 with values in the category of cochain complexes is called a \emph{generalized perverse cochain complex}. For instance, if $\ob{K}$ is a \ffs, the association $\ov{p}\mapsto \tN^*_{\ov{p}}(\ob{K})$ is a (generalized) perverse cochain complex and this association is natural in $\ob{K}$.

\begin{definition}\label{def:perverseE2algebra}
Let $A^*_{\bullet}$ be a generalized perverse cochain complex. We denote by $\varphi_{\ov{p}}^{\ov{q}}\colon A_{\ov{p}}^*\to A_{\ov{q}}^*$ the morphism associated to $\ov{p}\leq \ov{q}$.
A \emph{perverse $\cE(2)$-algebra structure} on $A^*_{\bullet}$ is a family of cochain maps,
$\psi_{\ov{p},\ov{q}}\colon \cE(2)\otimes A_{\ov{p}}^*\otimes A_{\ov{q}}^*\to A_{\ov{p}+\ov{q}}^*$, 
verifying
\begin{enumerate}[(i)]
\item a \emph{compatibility condition with perversities:} for any loose perversities, $\ov{p}_{1}$, $\ov{q}_{1}$, $\ov{p}_{2}$, $\ov{q}_{2}$, with $\ov{p}_{1}\leq \ov{p}_{2}$ and $\ov{q}_{1}\leq \ov{q}_{2}$, the following diagram is commutative
$$\xymatrix{
\cE(2)\otimes A_{\ov{p}_{1}}^*\otimes A_{\ov{q}_{1}}^*
\ar[rr]^-{\psi_{\ov{p}_{1},\ov{q}_{1}}}
\ar[d]_{\id\otimes \varphi_{\ov{p}_{1}}^{\ov{p}_{2}}\otimes  \varphi_{\ov{q}_{1}}^{\ov{q}_{2}}}
&&
A_{\ov{p}_{1}+\ov{q}_{1}}^*
\ar[d]^{\varphi_{\ov{p}_{1}+\ov{q}_{1}}^{\ov{p}_{2}+\ov{q}_{2}}}
\\
\cE(2)\otimes A_{\ov{p}_{2}}^*\otimes A_{\ov{q}_{2}}^*.
\ar[rr]^-{\psi_{\ov{p}_{2},\ov{q}_{2}}}
&&
A_{\ov{p}_{2}+\ov{q}_{2}}^*
}$$
\item a \emph{$\Sigma_{2}$-equivariance} as map from $\cE(2)$ to $(\hom(A^*_{\ov{p}}\otimes A^*_{\ov{q}},A^*_{\ov{p}+\ov{q}}))_{\ov{p},\ov{q}}$ with the following $\Sigma_{2}$-action on the codomain:\\
To any family $\eta_{\ov{p},\ov{q}}\colon A^*_{\ov{p}}\otimes A^*_{\ov{q}}\to A^*_{\ov{p}+\ov{q}}$, we associate the family $(\tau\eta)_{\ov{p},\ov{q}}\colon A^*_{\ov{p}}\otimes A^*_{\ov{q}}\to A^*_{\ov{p}+\ov{q}}$, defined by
$(\tau\eta)_{\ov{p},\ov{q}}(x_{1}\otimes x_{2})=\eta_{\ov{q},\ov{p}}(x_{2}\otimes x_{1})$.
\end{enumerate} 
\end{definition}
Equivalently, a \emph{perverse $\cE(2)$-algebra structure} on $A^*_{\bullet}$ is  entirely determined by maps, 
called \emph{perverse $\text{cup}_{i}$-products,} 
$\cup_{i} \colon A^r_{\ov{p}}\otimes A^s_{\ov{q}}\to A^{r+s-i}_{\ov{p}+\ov{q}}$, 
satisfying the previous Leibniz condition and the compatibility conditions with the poset structure of perversities. 
(The two settings are related by $x\cup_{i}y=\psi_{\ov{p},\ov{q}}(e_{i}\otimes x\otimes y)$.)
\emph{Nice perverse $\cE(2)$-algebras} are defined as in \defref{def:niceE2algebra}.

When $A^*_{\bullet}$ is a perverse cochain complex and the sum $\ov{p}+\ov{q}$ replaced by the sum of GM-perversities,
$\ov{p}\oplus\ov{q}$, in \defref{def:perverseE2algebra}, 
we say that $A^*_{\bullet}$ is a \emph{GM-perverse $\cE(2)$-algebra.}

\medskip
Let $\ob{K}$ be a \ffs~and $\sigma\colon \Delta=\Delta^{j_0}\ast\Delta^{j_1}\ast\cdots\ast\Delta^{j_n}\to \ob{K}$. 
With the tensor product of $\cE(2)$-algebras recalled in (\ref{equa:tensorE2}) and the structure of nice $\cE(2)$-algebra defined on the normalized cochain complex in \cite{MR2075046}, we get a structure of nice $\cE(2)$-algebra on
the tensor product
$\tN^*(\ob{K})_{\sigma}=N^*(c\Delta^{j_{0}})\otimes\cdots\otimes N^*(c\Delta^{j_{n-1}})\otimes N^*(\Delta^{j_{n}})$.
The next theorem establishes the compatibility of this structure with the perverse degrees.

\begin{theorem}\label{thm:NXE2}
Let $\ob{K}$ be a \ffs~and $\ov{p}$ be a loose perversity. The generalized perverse cochain complex, $\ov{p}\mapsto \tN^*_{\ov{p}}(\ob{K})$, is a nice perverse $\cE(2)$-algebra, natural in $\ob{K}$, for the filtered face maps.
\end{theorem}

Recall that a continuous map, $f\colon X=(X_{j})_{0\leq j\leq n}\to Y=(Y_{j})_{0\leq j\leq n}$, between pseudomanifolds is 
a \emph{stratum preserving stratified map} if, for any stratum $S'$ of $Y'$, $f^{-1}(S')$ is a union of strata of $X$ 
and $f^{-1}(Y_{n-\ell})=X_{n-\ell}$, for any $\ell\geq 0$. 
As any stratum preserving, stratified map induces a \ffs~map, $\ob{\rm ISing}^{\cF}(X)\to \ob{\rm ISing}^{\cF}(Y)$, 
(see \cite[Example 1.5]{2012arXiv1205.7057C}) the next result is a direct consequence of \thmref{thm:NXE2}.

\begin{corollary}\label{cor:NXE2pseudo}
Let $X$ be a pseudomanifold and $\ov{p}$ be a loose perversity. The generalized perverse cochain complex, $\ov{p}\mapsto \tN^*_{\ov{p}}(\ob{\rm ISing}^{\cF}(X))$,
is a nice perverse $\cE(2)$-algebra, natural in $X$ by stratum preserving stratified maps.
\end{corollary}

\begin{proof}[Proof of \thmref{thm:NXE2}]
A cochain $c\in \tN^*(\ob{K})$ associates to any simplex,
$\sigma\colon \Delta=\Delta^{j_0}\ast\cdots\ast\Delta^{j_n}\to\ob{K}_{+}$,
an element $c_{\sigma}\in N^*(c\Delta^{j_{0}})\otimes\cdots\otimes N^*(c\Delta^{j_{n-1}})\otimes N^*(\Delta^{j_{n}})$. 

If we set $(c\cup_{i}c')_{\sigma}=c_{\sigma}\cup_{i}c'_{\sigma}$, by naturality of the structure of $\cE(2)$-algebra on 
$N^*(c\Delta^{j_{0}})\otimes\cdots\otimes N^*(c\Delta^{j_{n-1}})\otimes N^*(\Delta^{j_{n}})$,
we get a global section
$c\cup_{i}c'\in \tN^*(\ob{K})$. More precisely, we have a $\Sigma_{2}$-equivariant cochain map,
$$\cE(2)\otimes \tN^*(\ob{K})^{\otimes 2}\to \tN^*(\ob{K}),$$
entirely defined by $e_{i}\otimes c\otimes c'\mapsto c\cup_{i}c'$, which gives to $\tN^*(\ob{K})$ a structure of $\cE(2)$-algebra. The niceness of this structure is a direct consequence of \lemref{lem:produitnice}.

The naturality in $\ob{K}$ comes from the naturality of the $\cE(2)$-algebra structure on a tensor product,
already mentioned, 
and from the naturality of the association,
$\ob{K}\mapsto \tN^*_{\ov{p}}(\ob{K})$,
see \cite[Proposition 1.36]{2012arXiv1205.7057C}.

We study now the behavior of this structure with the perverse degree. 
%%%%%%%%%%%%%%%%%%%%
The perversity degree being a local notion, we consider $c$ and $c'$ in
$N^*(c\Delta^{j_{0}})\otimes\cdots\otimes N^*(c\Delta^{j_{n-1}})\otimes N^*(\Delta^{j_{n}})$, with $j_{n}\geq 0$, and $\ell\in\{1,\ldots,n\}$ such that $\Delta^{j_{n-\ell}}\neq\emptyset$. 
We denote by $c_{n-\ell}$ and $c'_{n-\ell}$ the respective restrictions of $c$ and $c'$ to
$N^*(c\Delta^{j_{0}})\otimes\cdots\otimes N^*(\Delta^{j_{n-\ell}}\times\{1\})\otimes\cdots\otimes
N^*(c\Delta^{j_{n-1}})\otimes N^*(\Delta^{j_{n}})$. 

We decompose $c$, $c'$ in
$c=\sum_{s=0}^mc_{0}^s\otimes\cdots\otimes c_{n}^s$, 
$c'=\sum_{t=0}^{m'}c'^t_{0}\otimes\cdots\otimes c'^t_{n}$ and their restriction in
$c_{n-\ell}=\sum_{s=0}^mc_{0}^s\otimes\cdots\otimes \iota_{n-\ell}^*c_{n-\ell}^s\otimes \cdots\otimes c_{n}^s$, 
$c'_{n-\ell}=\sum_{t=0}^mc'^t_{0}\otimes\cdots\otimes \iota_{n-\ell}^*c'^t_{n-\ell}\otimes \cdots\otimes c'^t_{n}$,
where $\iota^*_{n-\ell}$ is induced by the  inclusion $\Delta^{j_{n-\ell}}\times\{1\}\hookrightarrow c\Delta^{j_{n-\ell}}$.
By definition, we have
$$\| c\|_{\ell}=\sup_{s}\{|c_{n-\ell+1}^s\otimes\cdots\otimes c_{n}^s| \text{ such that } c^s_{0}\otimes\cdots\otimes\iota^*_{n-\ell}c_{n-\ell}^s\neq 0\}.$$

Let  $$\xymatrix{
N^*(c\Delta^{j_{0}})\otimes\cdots \otimes N^*(c\Delta^{j_{n-\ell}})\otimes\cdots\otimes N^*(c\Delta^{j_{n-1}})\otimes N^*(\Delta^{j_{n}})\ar[d]^{\hat{\iota}^*_{n-\ell}=\id\otimes \iota^*_{n-\ell}\otimes\id}\\
N^*(c\Delta^{j_{0}})\otimes\cdots\otimes N^*(\Delta^{j_{n-\ell}}\times\{1\})\otimes\cdots\otimes
N^*(c\Delta^{j_{n-1}})\otimes N^*(\Delta^{j_{n}}).
}$$
As the $\text{cup}_{i}$-product is natural, we have
$\hat{\iota}^*_{n-\ell}(c\cup_{i}c')=\hat{\iota}^*_{n-\ell}(c)\cup_{i}\hat{\iota}^*_{n-\ell}(c')$. 
\begin{itemize}
\item If $\hat{\iota}^*_{n-\ell}(c)=0$ or $\hat{\iota}^*_{n-\ell}(c')=0$, we have
$\hat{\iota}^*_{n-\ell}(c)\cup_{i}\hat{\iota}^*_{n-\ell}(c')=0$ and thus
$$\|c\cup_{i}c'\|_{\ell}=-\infty.$$
\item Suppose now $\hat{\iota}^*_{n-\ell}(c)\neq 0$ and $\hat{\iota}^*_{n-\ell}(c')\neq 0$. 
By definition of the 
$\text{cup}_{i}$-product, $\hat{\iota}^*_{n-\ell}(c)\cup_{i}\hat{\iota}^*_{n-\ell}(c')$ is a sum of tensor products whose elements  are of two kinds:
\begin{enumerate}
\item $c^s_{j}\cup_{f_{j}}c'^t_{j}$, with $j\neq n-\ell$, or
\item $\hat{\iota}^*_{n-\ell}(c^s_{n-\ell})\cup_{f_{n-\ell}} \hat{\iota}^*_{n-\ell}(c'^t_{n-\ell})$.
\end{enumerate}
\end{itemize}
As $|c^s_{j}\cup_{f_{j}}c'^t_{j}|\leq |c^s_{j}|+|c'^t_{j}|-f_{j}$, the cochain degree decreases and we obtain, for each $\ell$,
$$\|c\cup_{i}c'\|_{\ell}\leq \|c\|_{\ell}+\|c'\|_{\ell},$$
by definition of the perverse degree, see \defref{def:tranversedegree}.
Therefore, we have
$$\|c\cup_{i}c'\|\leq \|c\|+\|c'\|.$$
Now, the rule of Leibniz implies
$$\|\delta (c\cup_{i}c')\|\leq \max(\|\delta c\|+\|c'\|,\|\delta c'\|+\|c\|,\|c\|+\|c'\|).$$
Thus, if $\|c\|\leq \ov{p}$, $\|\delta c\|\leq \ov{p}$, $\|c'\|\leq \ov{q}$ and $\|\delta c'\|\leq \ov{q}$, we have $\|c\cup_{i}c'\|\leq \ov{p}+\ov{q}$ and 
$\|\delta (c\cup_{i}c')\|\leq \ov{p}+\ov{q}$. This implies that the $\cE(2)$-algebra structure on $\tN^*(\ob{K})$ induces equivariant cochain maps
$$\cE(2)\otimes \tN^*_{\ov{p}}(\ob{K})\otimes \tN^*_{\ov{q}}(\ob{K})\to \tN^*_{\ov{p}+\ov{q}}(\ob{K}).$$
That means: $\tN^*_{\bullet}(\ob{K})$ is a perverse $\cE(2)$-algebra.
\end{proof}

%%%%%%%%%%%%%%%%%%%
\section{Steenrod perverse squares}\label{sec:steenrod}

From the existence of perverse ${\rm cup}_{i}$-products, we define Steenrod squares, as in the classical case. In the next statement, when $i>0$, the fact that the loose perversity image of $\sq^i$ is $\cL(\ov{p},i)$, defined by
$\cL(\ov{p},i)(\ell)=\min(2\ov{p}(\ell), \ov{p}(\ell)+i)$, answers positively a conjecture of M.~Goresky and W.~Pardon, see \cite[Conjecture 7.5]{MR1014465}. More explicitly, we prove the existence of a dashed arrow which lifts the square $\sq^i$,
$$\xymatrix{
&H^{r+i}_{\TW,\cL(\ov{p},i)}\ar[d]\\
H^r_{\TW,\ov{p}}\ar[r]^-{\sq^i}\ar@{-->}[ur]&
H^{r+i}_{\TW,2\ov{p}}.
}$$
We still denote by $\sq^i$ this lifting.
\begin{theorem}\label{thm:steenrodsquare}
Let $\ob{K}$ be a \ffs~and $\ov{p}$, $\ov{q}$ be loose perversities
The perverse $\text{cup}_{i}$-products induce natural perverse squares, defined by $\sq^i(x)=x\cup_{|x|-i}x$, for $x\in H^{|x|}_{\TW,\ov{p}}(\ob{K})$, which satisfy the following properties.
\begin{enumerate}[(1)]
\item If $i<0$, then $\sq^i(x)=0$.
\item If $i\geq 0$, then we have
$$\sq^i\colon H^r_{\TW,\ov{p}}(\ob{K})\to H^{r+i}_{\TW,\cL(\ov{p},i)}(\ob{K}),
$$
where $\cL(\ov{p},i)=\min(2\ov{p}, \ov{p}+i)$ and 
\begin{enumerate}[(i)]
\item $\sq^i(x)=0$ if $i>|x|$,
\item $\sq^{|x|}(x)=x^2$,
\item $\sq^0=\id$.
\item If $x\in H^{|x|}_{\TW,\ov{p}}(\ob{K})$, $y\in H^{|y|}_{\TW,\ov{q}}(\ob{K})$,  one has the (internal) \emph{Cartan formula,}
$$\sq^i(x\cup y)=\sum_{i_{1}+{i_{2}}=i} \sq^{i_{1}}(x)\cup \sq^{i_{2}}(y)\in H^{|x|+|y|+i}_{\TW,\ov{r}}(\ob{K}),$$  
with $\ov{r}=\min (2\ov{p}+2\ov{q},\ov{p}+\ov{q}+i)$ and $\cup=\cup_{0}$.
\item For any pair $(i,j)$, with $i<2j$, one has the \emph{Adem relation,}
$$\sq^i\sq^j=\sum_{k=0}^{[i/2]}\left(
\begin{array}{c}
j-k-1\\
i-2k
\end{array}\right) \sq^{i+j-k}\sq^k$$
and $\sq^i\sq^j$ sends $H_{\TW,\ov{p}}^*$  into $H_{\TW,\ov{r}}^{*+i+j}$, with
$\ov{r}=\min(4\ov{p},2\ov{p}+i,\ov{p}+i+j)$.
\end{enumerate}
\end{enumerate}
\end{theorem}

Before proving this theorem, we establish a technical property on the tensor product of two nice $\cE(2)$-algebras, 
which is the keystone in the proof of \thmref{thm:steenrodsquare}.

\begin{lemma}\label{lem:goreskypardon}
Let $A$ and $B$ be two nice $\cE(2)$-algebras and $A\otimes B$ their tensor product equipped with the $\cE(2)$-algebra structure coming from the diagonal of $\cE(2)$. Let $x$, $x'$ in $A$, $y$, $y'$ in $B$ such that $|x|+|y|=|x'|+|y'|=d$, $|y|\leq r$ and $|y'|\leq r$. Then,  for any $k\in \{0,\ldots, d-i\}$ such that
$(x\cup_{d-k-i}x')\otimes (y\cup^{d-k-i}_{k}y')\neq 0$, we have
$|y\cup^{d-k-i}_{k}y'|\leq r+i$.
\end{lemma}

\begin{proof}
Suppose ${d-k-i}$ even. 
If $(x\cup_{d-k-i}x')\otimes (y\cup_{k}y')\neq 0$, we must have
$k\leq \min(|y|,|y'|)$ and $d-k-i\leq \min(|x|,|x'|)$,
which implies
$$d-i-\min(|x|,|x'|)\leq k.$$
Suppose $\min(|x|,|x'|)=|x|$. Then we have
\begin{eqnarray*}
|y|+|y'|-d+i+\min(|x|,|x'|)
&=&
|y|+|y'|-(|x|+|y|)+i+|x|\\
&=& |y'|+i,
\end{eqnarray*}
which implies
$$|y\cup_{k}y'|\leq |y|+|y'|-k\leq |y|+|y'|-d+i+\min(|x|,|x'|)\leq |y'|+i \leq r+i.$$
A similar argument gives the result in the case $\min(|x|,|x'|)=|x'|$.
Also, the proof is analogous to the previous one if ${d-k-i}$ is odd, since $|y'\cup_{k}y|\leq |y'|+|y|-k$.
\end{proof}

Directly from the definition of ${\rm cup}_{k}$-products, the inequalities $|y|\leq r$
and $|y'|\leq r$
imply
$|y\cup_{k}y'|\leq 2r$. Thus, the bound $|y\cup_{k}y'|\leq r+i$ obtained in \lemref{lem:goreskypardon} is exactly what is needed for the proof of the Goresky-Pardon conjecture, as we show in the beginning of the next proof.

\begin{proof}[Proof of \thmref{thm:steenrodsquare}]
Let $i\geq 0$.
From their definition as particular ${\rm cup}_{i}$-products, the Steenrod squares have their image in the intersection cohomology with loose perversity $2\ov{p}$. We  prove first that the loose perversity $2\ov{p}$ can be replaced by
$\cL(\ov{p},i)$.
We take over the arguments and the method used at the end of the proof of \thmref{thm:NXE2} by considering a cocycle
$c\in N^*(c\Delta^{j_{0}})\otimes\cdots\otimes N^*(c\Delta^{j_{n-1}})\otimes N^*(\Delta^{j_{n}})$, $\ell\in\{1,\ldots,n\}$, such that $\Delta^{j_{n-\ell}}\neq\emptyset$, and the restriction $c_{n-\ell}$ of $c$ to
$N^*(c\Delta^{j_{0}})\otimes\cdots\otimes N^*(\Delta^{j_{n-\ell}}\times\{1\})\otimes\cdots\otimes
N^*(c\Delta^{j_{n-1}})\otimes N^*(\Delta^{j_{n}})$.  Observe first that, by naturality, we have
$(c\cup_{|c|-i}c)_{n-\ell}=c_{n-\ell}\cup_{|c_{n-\ell}|-i}c_{n-\ell}$.

$\bullet$ If $c_{n-\ell}=0$, we have $(c\cup_{|c|-i}c)_{n-\ell}=0$ and $\|c\cup_{|c|-i}c\|_{\ell}=-\infty$. 

$\bullet$ If $c_{n-\ell}\neq 0$, we decompose it in a canonical form,
$c_{n-\ell}=\sum_{s}c'^s_{n-\ell}\otimes c''^s_{n-\ell}\in A\otimes B$,
with
$A=N^*(c\Delta^{j_{0}})\otimes\cdots\otimes N^*(\Delta^{j_{n-\ell}}\times\{1\})$
and
$B=N^*(c\Delta^{j_{n-\ell+1}})\otimes\cdots\otimes
N^*(c\Delta^{j_{n-1}})\otimes N^*(\Delta^{j_{n}})$. Using \lemref{lem:goreskypardon}, we know that
$(c'^s_{n-\ell}\cup_{|c_{n-\ell}|-k-i} c'^t_{n-\ell})\otimes
(c''^s_{n-\ell}\cup^{|c_{n-\ell}|-k-i}_{k}c''^t_{n-\ell})\neq 0$
implies
$|c''^s_{n-\ell}\cup^{|c_{n-\ell}|-k-i}_{k}c''^t_{n-\ell}|\leq \ov{p}(\ell)+i$,
for any pair of indices, $(s,t)$, in the writing of $c_{n-\ell}$. This implies
$\|c\cup_{|c|-i}c\|\leq \ov{p}+i$, as announced.

The condition on the perversity of the differential of $c\cup_{|c|-i} c$ is immediate here because~$c$ is a cocycle, and
the naturality follows from the fact that the lifting already exists at the level of the spaces of cocycles.

\medskip
The list ({\it{1}}), ({\it{2})-(\it{i})}, ({\it{2})-(\it{ii})}, ({\it{2})-(\it{iii})} 
of properties is a direct consequence of
 \thmref{thm:NXE2} and \cite[Section~5]{MR0281196}. 
 
 \medskip
 Let $A$ and $B$ be two nice $\cE(2)$-algebras. By definition of the diagonal action of $\cE(2)$ on the tensor product,
 we have a  Cartan external formula,
 $$\sq^i(a\otimes b)=\sum_{i_{1}+{i_{2}}=i} \sq^{i_{1}}(a)\otimes \sq^{i_{2}}(b),$$
 for
 $a\in A$
 and
 $b\in B$.  
 In our case, each factor, $A$ and $B$, satisfies the Cartan internal formula. Therefore, the 
Cartan internal formula on $A\otimes B$ is a direct consequence of the next equalities: 
 \begin{eqnarray*}
 \sq^i((a\otimes b)\cup (a'\otimes b')) &=_{(1)}&
 \sq^i((a\cup a')\otimes (b\cup b'))\\
 &=_{(2)}&
 \sum_{i_{1}+i_{2}=i}\sq^{i_{1}}(a\cup a')\otimes \sq^{i_{2}}(b\cup b')\\
 &=_{(3)}&
 \sum_{j_{1}+j_{2}+k_{1}+k_{2}=i}
 (\sq^{j_{1}}(a)\cup \sq^{j_{2}}(a'))\otimes (\sq^{k_{1}}(b)\cup\sq^{k_{2}}(b'))
 \end{eqnarray*}
 and
{\footnotesize  \begin{eqnarray*}
 \sum_{i_{1}+i_{2}=i}\sq^{i_{1}}(a\otimes b)\cup \sq^{i_{2}}(a'\otimes b')
 &=_{(2)}&
 \sum_{s_{1}+s_{2}+t_{1}+t_{2}=i}(\sq^{s_{1}}(a)\otimes\sq^{s_{2}}(b))\cup
 (\sq^{t_{1}}(a')\otimes \sq^{t_{2}}(b'))\\
 &=_{(1)}&
 \sum_{s_{1}+s_{2}+t_{1}+t_{2}=i}
 (\sq^{s_{1}}(a)\cup \sq^{t_{1}}(a'))\otimes (\sq^{s_{2}}(b)\cup \sq^{t_{2}}(b')),
 \end{eqnarray*}}\noindent
where  $=_{(1)}$ comes from the definition of the cup-product on a tensor product, 
 $=_{(2)}$ from the application of the  Cartan external formula and
 $=_{(3)}$ from the Cartan internal formula on each factor. 
   
 \medskip
For the Adem's formula ({\it{2})-(\it{v})}, we need to recall some properties in order to track the perversity conditions. 
The classical proof uses the Bar resolution, $\cE(4)$, of $\F_{2}$ as a $\Sigma_{4}$-module, and the existence of a $\Sigma_{4}$-equivariant cochain map,
$\cE(4)\otimes N^*(L)^{\otimes 4}\to N^*(L)$, for any simplicial set $L$, called an $\cE(4)$-algebra. 
As these objects appear just in this part of proof, we do not recall them in detail, 
referring to \cite[Section~1]{MR2075046}. We mention only the points related to the control of perversities. 

Denote by $\omega\colon \cE(2)\otimes\cE(2)\otimes\cE(2)\to \cE(4)$ the cochain map induced by the wreath product
$\Sigma_{2}\times \Sigma_{2}\times \Sigma_{2}\to \Sigma_{4}$. Let $A$ be an $\cE(2)$ and an $\cE(4)$-algebra whose structure maps are respectively denoted $\psi_{2}$ and $\psi_{4}$. By definition, we say that $A$ is an Adem-object (\cite{MR0281196}) if there is a commutative diagram
$$
\xymatrix@C=2cm{
\cE(2)\otimes\cE(2)^{\otimes 2}\otimes A^{\otimes 4}\ar[r]^-{\omega\otimes \id}\ar[d]_{\rm{Sh}}&
\cE(4)\otimes A^{\otimes 4}\ar[r]^-{\psi_{4}}&A\\
\cE(2)\otimes (\cE(2)\otimes A^{\otimes 2})^{\otimes 2}\ar[r]_-{\id\otimes \psi_{2}^{\otimes 2}}&
\cE(2)\otimes A^{\otimes 2}\ar[ru]_-{\psi_{2}}&
}$$
where {\rm {Sh}} is the appropriate shuffle map. 
 
 Let $\Delta=\Delta^{j_{0}}\ast \cdots\ast \Delta^{j_{n}}$ and
 $A=N^*(c\Delta^{j_{0}})\otimes\cdots\otimes N^*(c\Delta^{j_{n-1}})\otimes N^*(\Delta^{j_{n}})$.
Because  $N^*(L)$ is an Adem-object for any simplicial set $L$ and because 
the tensor product of two nice $\cE(2)$-algebras which are Adem-objects is an Adem-object
(\cite[Lemma 4.2, Page 174]{MR0281196}),
$A$ is an Adem-object. 

In \thmref{thm:NXE2}, we prove that $\psi_{2}$ restricts to a map
$\cE(2)\otimes A_{\ov{p}}\otimes A_{\ov{q}}\to A_{\ov{p}+\ov{q}}$.  
Exactly the same argument can be used for $\psi_{4}$,
 replacing $c\cup_{i}c'$ by
$\psi_{4}(\alpha_{i}\otimes c_{1}\otimes c_{2}\otimes c_{3}\otimes c_{4})$
for each $\alpha_{i}\in \cE(4)$, in the last part of the proof of \thmref{thm:NXE2}.
 Thus $\psi_{4}$ restricts to a map
$\cE(4)\otimes A_{\ov{p}_{1}}\otimes A_{\ov{p}_{2}}\otimes A_{\ov{p}_{3}}\otimes A_{\ov{p}_{4}}
\to A_{\ov{p}_{1}+\ov{p}_{2}+\ov{p}_{3}+\ov{p}_{4}}
$ and we get an  Adem formula for intersection cohomology.

Successive applications of \lemref{lem:goreskypardon} show that the non-zero terms
 in the right-hand side of the Adem relation
belong to intersection cohomology in perversities less than, or equal to,
$\min(4\ov{p},2\ov{p}+2j,2\ov{p}+i,\ov{p}+i+j)\leq
\min(4\ov{p},2\ov{p}+i,\ov{p}+i+j)
$, since $i<2j$. The same argument applied to the left-hand side implies that the non-zero terms
belong also to intersection cohomology in the same range of perversities.
\end{proof}

\begin{remark}\label{rem:GMperversities}
Previous definitions and results can be adapted to the context of GM-perversities. By restricting to GM-perversities $\ov{p}$ and $\ov{q}$ such that $\ov{p}+\ov{q}\leq \ov{t}$, the $\text{cup}_{i}$-products are defined by
$$\cup_{i}\colon A_{\ov{p}}^r\otimes A_{\ov{q}}^s\to A_{\ov{p}\oplus\ov{q}}^{r+s-i},$$
where the sum
$\ov{p}\oplus\ov{q}$ is taken in the lattice $\cP^n$, see \cite{MR2544388} or \cite[Section 2.1]{2012arXiv1205.7057C}.
The  Steenrod squares introduced in \secref{sec:steenrod},
$$\sq^i\colon H^r_{\TW,\ov{p}}\to H^{r+i}_{\TW,\ov{r}},$$ are 
therefore defined for GM-perversities $\ov{p}$, $\ov{r}$ such that
$\min(2\ov{p},\ov{p}+i)\leq \ov{r}$.
\end{remark}
%%%%%%%%%%%%%%%%%%%%%%%

\section{Comparison with Goresky's construction}\label{sec:goresky}

As this section is concerned with isomorphisms between different definitions of Steenrod squares in
intersection cohomology, in some crucial points, we keep all the information in the notations of cohomology groups.

In \cite{MR696691} (see also \cite[Chapter V]{MR2401086}), the intersection cohomology on a pseudomanifold, $X$, 
is introduced by the use of  a sheaf due to Deligne. 
The Deligne's sheaf, $\ds_{\ov{p}}$, is defined by a sequence of truncations starting from the constant sheaf on 
$X_{n}\backslash X_{n-2}$. As we are not using this specific construction, 
we do not recall it, sending the reader to the previous references.

In \cite{MR761809}, M. Goresky has already defined Steenrod squares, $\sqg^i$, on the intersection cohomology, 
$H^*(X;\ds_{\ov{p}})$,
of a topological pseudomanifold, $X$, in the case of a GM-perversity~$\ov{p}$. 
In this section, we prove that the two Steenrod squares, $\sq^i$ and $\sqg^i$,  coincide.

Recall the \ffs~$\ob{\rm ISing}^{\cF}(X)$ introduced in \remref{rem:ffspseudomanifold}.
 The next result connects Goresky's definition of Steenrod squares on 
 $H^*(X;\cP_{\ov{p}})$
to our definition of Steenrod squares on the TW-cohomology of the \ffs~$\ob{\rm ISing}^{\cF}(X)$, 
denoted $H^*_{\TW,\ov{p}}(X)$.

\begin{theorem}\label{thm:Goreskyandblowup}
Let $X$ be an $n$-dimensional topological pseudomanifold. For any GM-perversity $\ov{q}$,  there exists an
isomorphism $\theta^{*}_{\ov{q}}\colon H^*_{\TW,\ov{q}}(X)\to H^*(X;\ds_{\ov{q}})$. 
Moreover, if $\ov{p}$ is a GM-perversity  
such that $2\ov{p}\leq \ov{t}$, then the following diagram commutes,
$$\xymatrix{
H^{r}_{\TW,\ov{p}}(X)\ar[r]^-{\sq^i}\ar[d]_{\theta^r_{\ov{p}}}&
H^{r+i}_{\TW,\cL(\ov{p},i)}(X)\ar[r]&
H^{r+i}_{\TW,2\ov{p}}(X)\ar[d]^{\theta^{r+i}_{2\ov{p}}}\\
H^r(X;\ds_{\ov{p}})\ar[rr]^{\sqg^i}&&
H^{r+i}(X;\ds_{2\ov{p}}).
}$$
\end{theorem}

The previous statement implies that 
$\theta^{r+i}_{\cL(\ov{p},i)}\circ \sq^i\circ (\theta^r_{\ov{p}})^{-1}\colon H^r(X;\ds_{\ov{p}})\to H^{r+i}(X;\ds_{\cL(\ov{p},i)})$
is a lift of the Steenrod squares defined by Goresky,
$\sqg^i\colon H^r(X;\ds_{\ov{p}})
\to H^{r+i}(X;\ds_{2\ov{p}})$.
Therefore \emph{the Goresky-Pardon conjecture has a positive answer.}

\medskip
From the functor $\tN^*$, we  define a presheaf on $X$ by
$$IN^*_{\ov{p}}(U)=\tN_{\ov{p}}^*(\ob{{\rm ISing}}^{\cF}(U)),$$
for any open set $U$ of $X$. 
Denote by $\cov(U)$ the directed set of open covers of $U$, ordered by inclusions. For any $\cU\in\cov(U)$, 
$\ob{\rm ISing}^{\cF,\cU}(U)$  is the sub-\ffs~of $\ob{{\rm ISing}}^{\cF}(U)$ whose elements have a support included in an element of $\cU$. The sheafification of $IN^*_{\ov{p}}$ is given by
$$\IN^*_{\ov{p}}(U)=\lim_{\cU\in \cov(U)}\tN_{\ov{p}}^*(\ob{{\rm ISing}}^{\cF,\cU}(U)),$$
see \cite[Exemple 3.9.1.]{MR0345092} in the case of singular cochains.
The ${\rm cup}_{i}$-products introduced in  \secref{sec:steenrod} on $\tN^*_{\bullet}(\ob{{\rm ISing}}^{\cF,\cU}(U))$ induce ${\rm cup}_{i}$-products on $\IN^*_{\bullet}(U)$, by definition of the last one as a direct limit.

\smallskip
\thmref{thm:Goreskyandblowup} is a direct consequence of Lemmas \ref{lem:prefaisceaucup} and \ref{lem:goresky=}. 
First, we connect the definition of Steenrod squares on $\ob{\rm ISing}^{\cF}(X)$ with a definition involving the sheaf $\IN^*_{\bullet}$ on $X$.

\begin{lemma}\label{lem:prefaisceaucup}
For any $n$-dimensional topological pseudomanifold, $X$, and any  GM-perversity $\ov{p}$, we have a commutative diagram,
$$\xymatrix{
H^r_{\TW,\ov{p}}(X)
\ar[r]^-{\sq^i}\ar[d]_{\cong}&
H^{r+i}_{\TW,\cL(\ov{p},i)}(X)
\ar[d]^-{\cong}\\
H^r(X;\IN_{\ov{p}})\ar[r]^-{\sq^i}&
H^{r+i}(X;\IN_{\cL(\ov{p},i)}),
}$$
in which vertical maps are quasi-isomorphisms induced by the canonical map $IN^*_{\bullet}\to \IN^*_{\bullet}$.
\end{lemma}

\begin{proof}
For any $\cU\in \cov(X)$, there is a restriction map,
 $r_{\cU}\colon IN^*_{\ov{p}}(X)\to \tN^*_{\ov{p}}(\ob{\rm ISing}^{\cF,\cU}(X))$,
 compatible with the inclusions of open covers. This gives the morphism,
 $$IN^*_{\ov{p}}(X)\to 
\Gamma(X,\IN^*_{\ov{p}}):= \lim_{\cU\in \cov(X)}\tN_{\ov{p}}^*(\ob{{\rm ISing}}^{\cF,\cU}(X)),$$
 induced by the canonical map $IN^*_{\bullet}\to \IN^*_{\bullet}$.
By taking the direct limit of the quasi-isomorphisms of \lemref{lem:Upetits}, we get an isomorphism
$$H^*\left(\lim_{\cU\in \cov(U)}\tN_{\ov{p}}^*(\ob{\rm ISing}^{\cF,\cU}(U))\right)
\cong
H^*\left(\tN_{\ov{p}}^*(\ob{\rm ISing}^{\cF}(U))\right)
=H^*_{\TW,\ov{p}}(U).
$$

In a second step, by following the lines of \cite[Exemple 3.9.1.]{MR0345092}, we prove that the sheaf, $\IN^*_{\ov{p}}$, is soft. 
The elements of $IN^0_{\ov{0}}(U)$ are $\ov{0}$-admissible vertices; they are the vertices of  the regular part and the map
$N^0(U)\to IN^0_{\ov{0}}(U)$ can be considered as the restriction to the regular part. Also, in this degree~0, the presheaves $N^0$ and $IN^0_{\ov{0}}$ are clearly sheaves and  $N^0(U)\to IN^0_{\ov{0}}(U)$, is a
 morphism of sheaves of rings. 
Observe also that $IN^*_{\ov{p}}(U)$ is an $IN^0_{\ov{0}}(U)$-module for the cup-product. 
 As the sheaf 
 $N^0$ is soft, and as (see \cite[Th\'eor\`eme 3.7.1.]{MR0345092}) any sheaf of modules over a soft sheaf of rings is soft, we deduce the softness of  $\IN^*_{\ov{p}}$. Thus, the hypercohomology is the cohomology of the space of sections of the sheaf and we  get a series of isomorphisms,
 $$H^*(X;\IN^*_{\ov{p}})\cong H^*(\Gamma(X,\IN^*_{\ov{p}}))\cong 
 H^*(\tN_{\ov{p}}^*(\ob{\rm ISing}^{\cF}(X)))=H^*_{\TW,\ov{p}}(X).$$
By definition of the ${\rm cup}_{i}$-products on $\IN^*_{\bullet}$, the following diagram commutes,
$$\xymatrix{
IN^r_{\ov{p}}(X)\otimes IN^s_{\ov{q}}(X)
\ar[r]^-{\cup_{i}}\ar[d]_{\simeq}&
IN^{r+s-i}_{\ov{p}\oplus\ov{q}}(X)\ar[d]^{\simeq}\\
\Gamma(X,\IN^r_{\ov{p}})\otimes \Gamma(X,\IN^s_{\ov{q}})\ar[r]^-{\cup_{i}}&
\Gamma(X,\IN^{r+s-i}_{\ov{p}\oplus\ov{q}}).
}$$
With the properties already established, the vertical maps are quasi-isomorphisms induced by the canonical map $IN^*_{\bullet}\to \IN^*_{\bullet}$. The stated result is now a consequence of the definition of Steenrod squares from  ${\rm cup}_{i}$-products.
\end{proof}

\begin{lemma}\label{lem:Upetits}
Let $X$ be an $n$-dimensional pseudomanifold  and $\cU$ be an open cover of~$X$.
The canonical inclusion,
$\iota\colon\ob{\rm ISing}^{\cF,\cU}(X)\to \ob{\rm ISing}^{\cF}(X)$,
induces an isomorphism in intersection cohomology, for any GM-perversity $\ov{p}$. 
\end{lemma}

\begin{proof}
With \propref{prop:CandN}, we can replace $\tN^*(-)$ by  the blow-up $\tC^*(-)$, already studied in \cite{2012arXiv1205.7057C}.
Let $\ov{q}$ be the GM-perversity defined by $\ov{p}(k)+\ov{q}(k)=k-2$. 
Recall from 
\cite[Theorem~B]{2012arXiv1205.7057C},
 the existence of a quasi-isomorphism,
$\ev\colon \tC^*_{\ov{p}}(\ob{K})
\to
\hom(C_{*}^{\GM,\ov{q}}(\ob{K}),\F_{2})$,
defined for any \ffs, $\ob{K}$, as follows:\\
for any $\Phi\in \tC^*_{\ov{p}}(\ob{K})$, 
$\sigma\colon \Delta^{j_{0}}\ast\cdots\ast\Delta^{j_{n}}\to \ob{K}$,
we have $\Phi_{\sigma}=\sum_{j}\Phi_{0,\sigma,j}\otimes\cdots\otimes \Phi_{n,\sigma,j}\in
C^*(c\Delta^{j_{0}})\otimes\cdots\otimes C^*(\Delta^{j_{n}})$
and we set
$$\ev (\Phi)(\sigma)=\sum_{j}
\Phi_{0,\sigma,j}([c\Delta^{j_{0}}])\cdot\ldots\cdot
\Phi_{n,\sigma,j}([\Delta^{j_{n}}]),$$
where $[-]$ is the maximal simplex. 
By using it for $\ob{K}=\ob{\rm ISing}^{\cF}(X)$ and $\ob{K}=\ob{\rm ISing}^{\cF,\cU}(X)$, we get the following diagram, whose commutativity follows directly from the 
definitions of maps,
$$\xymatrix{
H^*_{\TW,\ov{p}}(\ob{\rm ISing}^{\cF}(X))\ar[r]^-{\ev^*}\ar[d]_{\iota^*_{\TW}}&
 H^*_{\GM,\ov{q}}(\ob{\rm ISing}^{\cF}(X))\ar[d]^{\iota^*_{\GM}}\\
 H^*_{\TW,\ov{p}}(\ob{\rm ISing}^{\cF,\cU}(X))\ar[r]^-{\ev^*}&
 H^*_{\GM,\ov{q}}(\ob{\rm ISing}^{\cF,\cU}(X)).
}$$
We know that the two evaluation maps, $\ev^*$, are quasi-isomorphisms and
 we have to prove that the map, $\iota^*_{\TW}$, induced by the inclusion, $\iota$, is an isomorphism. 
 With the commutativity of the previous diagram, and the fact that the homology is over a field, it is sufficient to prove that 
 $$\iota_{\GM,*}\colon  
 H_*^{\GM,\ov{q}}(\ob{\rm ISing}^{\cF,\cU}(X))\to
 H_*^{\GM,\ov{q}}(\ob{\rm ISing}^{\cF}(X))
 $$
  is an isomorphism. Set $C_*^{\ov{q}}(X)=C_{*}^{\GM,\ov{q}}(\ob{\rm ISing}^{\cF}(X))$. Recall from \cite[Lemma~A.16]{2012arXiv1205.7057C}, the existence of 
  a chain map, which is the classical subdivision,
  $\sd\colon C_*^{\ov{q}}(X)\to C_*^{\ov{q}}(X)$, and, for any integer $m$, the existence of a homomorphism,  $T\colon C_*^{\ov{q}}(X)\to C_{*+1}^{\ov{q}}(X)$, such that $\partial T+T \partial=\id -\sd^m$.  
By construction, for any element $c\in C_*^{\ov{q}}(X)$, there is an integer $m$ such that $\sd^m c\in C_{*}^{\GM,\ov{q}}(\ob{\rm ISing}^{\cF,\cU}(X))$.
Moreover, if $c\in C_{*}^{\GM,\ov{q}}(\ob{\rm ISing}^{\cF,\cU}(X))$ then $Tc\in C_{*}^{\GM,\ov{q}}(\ob{\rm ISing}^{\cF,\cU}(X))$. Also, if $c$ is a cycle, $\sd^m c$ is a cycle also and the two homology classes $[c]$ and $[\sd^m c]$ are equal. This implies the surjectivity and the injectivity of $\iota_{\GM,*}$ through a classical argument.
\end{proof}
 
 The second step in the proof of \thmref{thm:Goreskyandblowup} is the comparison of the two definitions of Steenrod squares, respectively associated  to the sheaf $\IN^*_{\bullet}$ and to the Deligne sheaf $\ds^*_{\bullet}$. This is a consequence of the comparison of the two associated ${\rm cup}_{i}$-products, done in the next lemma.
 
 \begin{lemma}\label{lem:goresky=}
Let $X$ be an $n$-dimensional topological pseudomanifold and 
 let $\ov{p}$, $\ov{q}$ be two GM-perversities, such that $\ov{p}\oplus \ov{q}\leq\ov{t}$, 
 where $\ov{p}\oplus\ov{q}$ is the smallest GM-perversity, $\ov{r}$, such that $\ov{p}+\ov{q}\leq \ov{r}$.
Then, for any $i$, there is a commutative square in the derived category of sheaves on $X$, 
linking the two $\text{cup}_i$-products,
$$\xymatrix@C=2cm{\IN^*_{\ov{p}}(X)
\otimes \IN^*_{\ov{q}}(X)
\ar[r]^-{\cup_{i}}\ar@{~>}[d]&
\IN^*_{\ov{p}\oplus\ov{q}}(X)
\ar@{~>}[d]\\
\ds^*_{\ov{p}}(X)\otimes \ds^*_{\ov{q}}(X)\ar[r]^-{\cup_{i}}&
\ds^*_{\ov{p}\oplus\ov{q}}(X),
}$$
and such that vertical arrows are isomorphisms.
 \end{lemma}

\begin{proof}
Let $\bSS^*$ be a differential graded sheaf on the pseudomanifold $X$. We denote by $\bSS^*_{k}$ the restriction of $\bSS^*$ to the open set $X\backslash X_{n-k}$, for $k\in \{2,\ldots, n+1\}$. Recall the conditions (AX1) of \cite[V.2.3]{MR2401086}:
\begin{enumerate}[(a)]
\item $\bSS^*$ is bounded, $\bSS^i=0$ for $i<0$ and $\bSS^*_{2}$ is quasi-isomorphic to the ordinary singular cohomology.
\item For any $k\in \{2,\ldots, n\}$ and any $x\in X_{n-k}\backslash X_{n-k-1}$, we have $\cH^i(\bSS)_{x}=0$ if $i> \ov{p}(k)$. 
\item The attachment map, $\alpha_{k}\colon \bSS^*_{k+1}\to Ri_{k^*}\bSS^*_{k}$,
induced by the canonical inclusion
$X\backslash X_{n-k}\to X\backslash X_{n-k-1}$,
is a quasi-isomorphism up to $\ov{p}(k)$.
\end{enumerate}
If $\bSS^*$ is soft, from \cite[Remark 2.3.]{MR2168981}, we may replace condition (c) by the following equivalent one:\\
(c$'$) for any $k\in\{2,\ldots,n\}$, $j\leq \ov{p}(k)$ and $x\in X_{n-k}\backslash X_{n-k-1}$, 
the restriction map induces an isomorphism,
$${\varinjlim}_{U_{x}}H^j(\Gamma(U_{x};\bSS^*))\xrightarrow[]{\cong}
{\varinjlim}_{U_{x}} H^j(\Gamma(U_{x}\backslash X_{n-k};\bSS^*)),
$$
where $U_{x}$ varies into a cofinal family of neighborhoods of $x$ in $X\backslash X_{n-k-1}$. 

\smallskip
On the regular part, the sheaf $\IN^*_{\bullet}$ is the sheafification of $N^*$ and thus computes the singular cohomology.
Therefore, condition (a) is satisfied for $\IN^*$. 
In order to prove that the sheaf $\IN^*_{\bullet}$ satisfies the axioms (b) and (c$'$), we use the isomorphism established in \lemref{lem:prefaisceaucup},
$$H^*(X;\IN_{\bullet})\cong H^*_{\TW,\bullet}(X).$$
Let $x\in X_{n-k}\backslash X_{n-k-1}$. 
The cohomology $\cH^*(\IN_{\bullet})_{x}$ is determined by the following isomorphisms,
$$\cH^*(\IN_{\bullet})_{x}={\varinjlim}_{U_{x}}
H^*(\Gamma(U_{x};\IN_{\bullet}))\cong {\varinjlim}_{U_{x}}H^*_{\TW,\bullet}(U_{x}),$$
where the direct limits are taken over the open neighborhoods $U_{x}$ of $x$. 
(The first equality is the definition of the stalk at a point.)
Moreover, these limits can also be obtained from a restriction to a cofinal family of trivializing open neighborhoods, 
$U_{x}\cong \R^{n-k}\times cL$, where $L$ is the link of $x$. Axiom (b) follows now from 
$H^*_{\TW,\ov{p}}(\R^{n-k}\times cL)=H^*_{\TW,\ov{p}}(cL)=0$, if $*>\ov{p}(k)$, 
see \cite[Corollary 1.47]{2012arXiv1205.7057C}.

The verification of (c$'$) is quite similar. As noticed in \cite[Proof of Theorem 7.1.]{MR2168981}, 
we are reduce to analyze the map,
$${\varinjlim}_{U_{x}}H^*(U_{x};\IN_{\bullet})\to {\varinjlim}_{U_{x}}H^*(U_{x}\backslash X_{n-k};\IN_{\bullet}),$$
where the direct limit is taken over a cofinal family  of  trivializing open neighborhoods of~$x$,
$U_{x}\cong \R^{n-k}\times cL$. We consider the following commutative diagram, whose horizontal maps are
induced by the canonical inclusions and vertical maps are isomorphisms,
$$\xymatrix{
H^*(U_{x};\IN_{\bullet})\ar[r]\ar[d]_{\cong}
&
H^*(U_{x}\backslash X_{n-k};\IN_{\bullet})\ar[d]_{\cong}
&
\\
H^*(\R^{n-k}\times cL;\IN_{\bullet})\ar[r]\ar[d]_{\cong}
&
H^*(\R^{n-k}\times (cL-\{\tv\});\IN_{\bullet})\ar[d]_{\cong}
&
\\
H^*(cL;\IN_{\bullet})\ar[r]
&
H^*(cL-\{\tv\};\IN_{\bullet})\ar[r]^-{\cong}
&
H^*(L;\IN_{\bullet}).
}$$
Finally, we note that the composite at the bottom is an isomorphism when $*\leq \ov{p}(k)$, as shows the
classical computation of the intersection cohomology of a cone. Modulo the vertical isomorphisms, this is exactly the axiom (c$'$).

Therefore, the sheaf $\IN^*_{\bullet}$ satisfies conditions (AX1) and, by \cite[Theorem 2.5]{MR2401086}, 
there exists a quasi-isomorphism between $\IN^*_{\bullet}$ and $\ds^*_{\bullet}$ (see also \cite{MR696691}). 
As a consequence, 
these two sheaves have a common injective resolution and we may apply to it the uniqueness of 
${\rm cup}_{i}$-products established by M.~Goresky in \cite[Proposition 3.6]{MR761809}.
\end{proof}

From the previous results on ${\rm cup}_{i}$-products, we get
an isomorphism of algebras of cohomology, with coefficients in $\F_{2}$.  

\begin{corollary}\label{cor:cupproducts}
If $X$ is an $n$-dimensional pseudomanifold, there are isomorphisms of perverse algebras,
$$H^*_{\TW,\bullet}(X)
\cong
H^*(X;\IN_{\bullet})
\cong
H^*(X;\ds_{\bullet}).
$$
Moreover, if $X$ is compact and PL, one has also an isomorphism of algebras,
$$H^*(X;\ds_{\bullet})\cong
H_{n-*}^{\ov{t}-{\bullet}}(X;\F_{2}),$$
with  the intersection product on the last term.
\end{corollary}

\begin{proof}
The two first isomorphisms are consequences of the previous results on ${\rm cup}_{i}$-products. The last one is established by G.~Friedman in \cite{MR2529162}.
\end{proof}

If we are interested only by the cup-product, $\cup_{0}$, we may consider versions of the sheaves, $\IN$ and $\ds$, over any field.
In this case, the previous corollary is still true  for any field and not only for $\F_{2}$. With more work of this type, one should be able also to show the existence of an isomorphism between our definition of cup-product and the definition of G. Friedman and J. E. McClure (\cite{MR3046315}).

%%%%%%%%%%%%%%%%%%%%%%
\section{Pseudomanifolds with  isolated singularities}\label{sec:isolated}

In this section, we determine Steenrod squares on the intersection cohomology of 
pseudomanifolds with isolated singularities. 
In this case, if the pseudomanifold is of dimension~$n$, the perversity $\ov{p}$ is determined by one number, $\ov{p}(n)$. 
Recall now that the intersection cohomology of a cone $cY$ on a space $Y$ is given by
$H^r_{\TW,\ov{p}}(cY)= H^r(Y)$, if $r\leq \ov{p}(n)$ and 0 otherwise.

\begin{proposition}\label{prop:isolated}
Let $\ov{p}$ be a GM-perversity and $X$ be an $n$-dimensional pseudomanifold obtained from a triangulated manifold with boundary, 
$(W,\partial W)$, by attaching cones on the connected components, $(\partial_{u}W)_{u\in I}$, of $\partial W$, i.e., 
$X$ is the push out
$$\xymatrix{
\partial W=\sqcup_{u\in I}\partial_{u} W\ar[r]^-{\iota}\ar[d]&W\ar[d]\\
\sqcup_{u\in I} c(\partial_{u} W)\ar[r]&X.
}$$
We filter the pseudomanifold $X$ by $\emptyset\subset \{\tv_{u}\mid u\in I\}\subset X$, where $\tv_{u}$ is the cone point of $c(\partial_{u} W)$.
Then, the following properties are satisfied.
\begin{enumerate}[(i)]
\item The cochain complex, $\tN^*_{\ov{p}}(X)$, is quasi-isomorphic to the pullback in the category of 
cochain complexes, 
$N^*(W)\oplus_{N^*(\partial W)}\tau_{\leq\ov{p}(n)}N^*(\partial W)$, 
where $\tau_{\leq\ov{p}(n)}N^*(\partial W)$ is the usual truncation (see \cite[Page~52]{MR2401086}),
$$
(\tau_{\leq \ov{p}(n)}N^*(\partial W))^r=
\left\{
\begin{array}{ccl}
N^r(\partial W)&\text{if}&r<\ov{p}(n),\\
\cZ N^{\ov{p}(n)}(\partial W)&\text{if}&r=\ov{p}(n),\\
0&\text{if}&r>\ov{p}(n),
\end{array}\right.
$$
in which $\cZ$ denotes the vector space of cocycles.
Moreover, the GM-perverse $\cE(2)$-algebra, $\ov{p}\mapsto \tN^*_{\ov{p}}(X)$,
is quasi-isomorphic to the pullback in the category of GM-perverse $\cE(2)$-algebras, defined by
$\ov{p}\mapsto N^*(W)\oplus_{N^*(\partial W)}\tau_{\leq\ov{p}(n)}N^*(\partial W)$,
with the $\cE(2)$-algebra structure on $N^*(-)$ defined in \cite{MR2075046}.
\item The intersection cohomology of $X$ is determined by
$$H^k_{\TW,\ov{p}}(X)=\left\{\begin{array}{lcl}
H^k(W)
&\text{if}&
k\leq \ov{p}(n),\\
\ker (H^k(W)\to H^k(\partial W))% 
&\text{if}&
k=\ov{p}(n)+1,\\
H^k(W,\partial W)
&\text{if}&
k>\ov{p}(n)+1.
\end{array}\right.$$
\item If $(\alpha,\iota^*\alpha)\in N^*(W)\oplus_{N^*(\partial W)}\tau_{\leq\ov{p}(n)}N^*(\partial W)$ is a cocycle of $\ov{p}$-intersection and  $i$ is a positive integer, we have
$$\sq^i(\alpha,\iota^*\alpha)=(\sq^i\alpha, \iota^*\sq^i\alpha)\in H^{*+i}_{\TW,\cL(\ov{p},i)}(X).$$
\end{enumerate}
\end{proposition}

\begin{proof}
(i) Starting from a triangulation of $(W,\partial W)$, we may suppose that $X$, $W$, $\partial W$ and
 $\sqcup_{u\in I}c(\partial_{u}W)$ are triangulated in such a way that any simplex of the triangulation of $X$ is filtered, 
 for the filtration $\emptyset\subset \{\tv_{u}\mid u\in I\}\subset X$.

Let $Y$ be one of the spaces above and $Y^\tau$ be the associated triangulated space. 
In \cite[Chapter 3 and Chapter 5]{IHGreg}, G. Friedman proves that the cochains 
$C^*_{\GM,\ov{p}}(Y)$ and $C^*_{\GM,\ov{p}}(Y^\tau)$ are quasi-isomorphic for any GM-perversity $\ov{p}$.
Let $\ov{p}$ and $\ov{q}$ be two GM-perversities such that $\ov{p}(k)+\ov{q}(k)=k-2$. There exists a quasi-isomorphism between 
$C^*_{\GM,\ov{q}}(Y)$ and $\tC^*_{\ov{p}}({Y})$, see \cite[Theorem B]{2012arXiv1205.7057C} or \propref{prop:GMetTW}.
Recall also from \propref{prop:CandN} the existence of a quasi-isomorphism between
$\tC^*_{\ov{p}}({Y})$
and
$\tN^*_{\ov{p}}({Y})$.
Thus, the isomorphism,
$$C^*_{\GM,\ov{q}}(X^\tau)\cong C^*(W^\tau)\oplus_{(\oplus_{u\in I}C^*(\partial_{u}W^\tau))}(\oplus_{u\in I}C^*_{\GM,\ov{q}}(c(\partial_{u}W)^\tau)),$$
obtained by construction of the triangulations, gives quasi-isomorphisms,
\begin{eqnarray*}
C^*_{\GM,\ov{q}}(X)
&\simeq &
C^*(W)\oplus_{C^*(\partial W)}(\oplus_{u\in I}C^*_{\GM,\ov{q}}(c(\partial_{u}W)))\\
&\simeq&
C^*(W)\oplus_{C^*(\partial W)}(\oplus_{u\in I}\;\tau_{\leq \ov{t}(n)-\ov{q}(n)}C^*(\partial_{u}W))\\
&\simeq&
C^*(W)\oplus_{C^*(\partial W)} \tau_{\leq \ov{t}(n)-\ov{q}(n)}C^*(\partial W)\\
&\simeq&
N^*(W)\oplus_{N^*(\partial W)} \tau_{\leq \ov{t}(n)-\ov{q}(n)}N^*(\partial W).
\end{eqnarray*}
Therefore
we have obtained a quasi-isomorphism,
$$\tN^*_{\ov{p}}(X)\simeq
N^*(W)\oplus_{N^*(\partial W)} \tau_{\leq \ov{p}(n)}N^*(\partial W).
$$

We investigate now the structure of $\cE(2)$-algebra. 
In \cite{MR2075046}, C. Berger and B. Fresse prove that a restriction map,
$N^*(Y)\to N^*(Z)$,
induced by an inclusion $Z \hookrightarrow Y$,
is a morphism of $\cE(2)$-algebras.
Therefore, we obtain functors from the lattice of GM-perversities (and $\ov{\infty}$) to
GM-perverse $\cE(2)$-algebras, defined by
$\ov{p}\mapsto \tN^*_{\ov{p}}(X)$,
$\ov{p}\mapsto N^*(W)$,
$\ov{p}\mapsto N^*(\partial W)$,
$\ov{p}\mapsto \tau_{\leq \ov{p}(n)}N^*(\partial W)$.
Restriction maps define GM-perverse $\cE(2)$-algebra maps between $\tN^*_{\ov{p}}(X)$ and the
three other GM-perverse $\cE(2)$-algebras. From them, we obtain a GM-perverse $\cE(2)$-algebra map
\begin{equation}\label{equa:modelisolated}
\tN^*_{\ov{p}}(X)\to
N^*(W)\oplus_{N^*(\partial W)} \tau_{\leq \ov{p}(n)}N^*(\partial W),
\end{equation}
whose codomain is a pullback in the category of GM-perverse $\cE(2)$-algebras, see \cite{MR2075046}. We have proved above that
this last map is a quasi-isomorphism for each $\ov{p}$ and the first item of the statement is established.

\medskip
(ii) An element of the previous sum is of the type $(\alpha,\iota^*\alpha)$, with $\iota^*\alpha$ of degree less than, or equal to, $\ov{p}(n)$. This means that, if  $\alpha$  is of degree $k$, we must have
$$\left\{
\begin{array}{lcl}
\iota^*\alpha=0
&\text{if}&
k>\ov{p}(n),\\
\iota^*\alpha \;\text{is a cocycle}
&\text{if}&
k=\ov{p}(n),\\
\text{no condition}
&\text{if}&
k<\ov{p}(n).
\end{array}\right.$$
This implies immediately $H^k_{\TW,\ov{p}}(X)=H^k(W)$ if $k\leq \ov{p}(n)$ and that $H^k_{\TW,\ov{p}}(X)=H^k(W,\partial W)$ if $k>\ov{p}(n)+1$. In degree $k=\ov{p}(n)+1$, the $\ov{p}$-intersection cohomology of $X$ is formed of the elements of $H^k(W)$ which are in the image of $H^k(W,\partial W)$, i.e., the kernel of $H^k(W)\to H^k(\partial W)$.

\medskip
(iii) The quasi-isomorphisms between 
$\tN^*_{\ov{p}}(X)$ and $N^*(W)\oplus_{N^*(\partial W)}\tau_{\leq\ov{p}(n)}N^*(\partial W)$
 defining a map of GM-perverse $\cE(2)$-algebras, they are
compatible with the ${\rm cup}_{i}$-products, and the right-hand complex of (\ref{equa:modelisolated}) can be used for the determination of ${\rm cup}_{i}$-products, i.e., we have
$$(\alpha,\iota^*\alpha)\cup_{i}(\beta,\iota^*\beta)=(\alpha\cup_{i}\beta,\iota^*\alpha\cup_{i}\iota^*\beta),$$
from which we deduce the announced formula for Steenrod squares.
\end{proof}

\begin{remark}\label{rem:conjectureisolated}
\emph{This remark gives a direct proof of the Goresky-Pardon conjecture in the case of isolated singularities.} 
Let $(\alpha,\iota^*\alpha)$ be a  cocycle in $N^k(W)\oplus_{N^k(\partial W)}\tau_{\leq\ov{p}(n)}N^k(\partial W)$.
The perverse degree of the Steenrod square,
$\sq^j(\alpha,\iota^*\alpha)=(\alpha,\iota^*\alpha)\cup_{k-j}(\alpha,\iota^*\alpha)$, verifies
$$\|(\alpha,\iota^*\alpha)\cup_{k-j}(\alpha,\iota^*\alpha)\|
\leq_{(1)}
|\iota^*\alpha\cup_{k-j}\iota^*\alpha|
 \leq_{(2)}
 k+j
 \leq_{(3)}
  \ov{p}(n)+j,$$
  where
  \begin{itemize}
\item $\leq_{(1)}$ comes from the fact that the perverse degree of a cochain is less than, or equal, to its usual degree,
\item $\leq_{(2)}$ is a consequence of $|a\cup_{i}b|\leq |a|+|b|-i$,
\item $\leq_{(3)}$ uses $\iota^*\alpha=0$ if $k>\ov{p}(n)$.
\end{itemize}
\end{remark}

\begin{remark}\label{rem:coherenceofSq}
The fact that the image of $H^*_{\TW,\ov{p}}(X)$ by $\sq^i$ is in perversity 
$\cL(\ov{p},i)=\min(2\ov{p},\ov{p}+i)$
 is perfectly in phase with the characterization of the intersection cohomology of $X$, 
 written in \propref{prop:isolated}.{\it (ii)}. 
 This remark follows from the next observations for a cocycle $(\alpha,\iota^*\alpha)\in N^k(W)\oplus_{N^k(\partial W)}\tau_{\leq\ov{p}(n)}N^k(\partial W)$.
\begin{itemize}
\item If $k\leq \ov{p}(n)$, then, by definition of the Steenrod squares in $H^*(W)$, we have
$|\sq^i(\alpha)|= k+i\leq \ov{p}(n)+i$.
\begin{itemize}
\item If $i\leq \ov{p}(n)$, this implies
$|\sq^i(\alpha,\iota^*\alpha)|\leq \cL(\ov{p},i)(n)$.
\item If $i> \ov{p}(n)$, we have $\sq^i(\alpha,\iota^*\alpha)=(\sq^i\alpha,\iota^*\sq^i\alpha)=0$.
\end{itemize}
\item If $k>\ov{p}(n)$, then $\iota^*\alpha=0$ and 
$|\sq^i(\alpha,\iota^*\alpha)|=k+i> \ov{p}(n)+i\geq \cL(\ov{p},i)(n)$. 
\end{itemize}
In conclusion, $\sq^i$ respects the c\ae sur\ae~in the determination of the perverse cohomologies, $H^*_{\TW,\ov{p}}(X)$ and $H^*_{\TW,\cL(\ov{p},i)}(X)$.
Moreover, in degrees $k\leq \ov{p}(n)$, \emph{the Steenrod squares on $H^k_{\TW,\ov{p}}(X)$
 coincide with the Steenrod squares on $H^k(W)$.}
\end{remark}

\begin{example}[\emph{Steenrod squares on the intersection cohomology of the suspension of a manifold}]
Let $X$ be an $(n-1)$-dimensional manifold and $\ov{p}$ be a GM-perversity. 
The following pushout defines $\Sigma X$, as in \propref{prop:isolated}, 
$$\xymatrix{
X_{-1}\sqcup X_{1}\ar[r]^{\iota_{-1}\sqcup \iota_{1}}\ar[d]&
X\times [-1,1]\ar[d]\\
cX_{-1}\sqcup cX_{1}\ar[r]&
\Sigma X,
}$$
where 
$\iota_{1}\colon X_{1}=X\times \{1\}\to X\times [-1,1]$
and
$\iota_{-1}\colon X_{-1}=X\times \{-1\}\to X\times [-1,1]$
are the canonical injections.
From (i) of \propref{prop:isolated}, we know that $\tN^*_{\ov{p}}(\Sigma X)$ is quasi-isomorphic to 
the cochain complex
\begin{equation}\label{equa:suspension}
\begin{array}{lcl}
N^*(X\times [-1,1])^{<\ov{p}(n)}
&\oplus&
\{\alpha\in N^{\ov{p}(n)}(X\times [-1,1])\mid d\iota_{1}^*(\alpha)=d\iota_{-1}^*(\alpha)=0\}\\
&\oplus&
(\ker\iota_{1}^*\cap\ker\iota_{-1}^*)^{>\ov{p}(n)},
\end{array}
\end{equation}
in which the superscript refers to the degree.
For instance, $(A)^{<k}$ is the set of elements of $A$ of degree less than $k$.

The suspension, $\Sigma X$, can also be obtained as a cofiber,
$X_{1}\sqcup X_{-1}\to X\times [-1,1]\to \Sigma X$,
which gives a short exact sequence,
$$0\rightarrow
(\ker\iota_{1}^*\cap\ker\iota_{-1}^*)\hookrightarrow
N^*(X\times [-1,1])\xrightarrow[]{(\iota_{1}^*,\iota_{-1}^*)}
N^*(X_{1})\oplus N^*(X_{-1})\rightarrow 0.$$
The morphism of $\cE(2)$-algebras (\cite{MR2075046}),
$N^*(\Sigma X)\to N^*(X\times [-1,1])$,
lifts  as a quasi-isomorphism of cochain complexes,
$N^*(\Sigma X)\to (\ker\iota_{1}^*\cap\ker\iota_{-1}^*)$. 
From (\ref{equa:suspension}) and the previous observation, we deduce the intersection cohomology of the suspension $\Sigma X$, as
$$
H^k_{\TW,\ov{p}}(\Sigma X)=\left\{
\begin{array}{cl}
H^k(X)&
\text{if } k\leq \ov{p}(n),\\
0&
\text{if } k= \ov{p}(n)+1,\\
H^k(\Sigma X)=H^{k-1}(X)&
\text{if } k> \ov{p}(n)+1.
\end{array}\right.
$$
With \remref{rem:coherenceofSq}, we know that, \emph{in degrees $k\leq \ov{p}(n)$, the Steenrod squares on 
$H^k_{\TW,\ov{p}}(\Sigma X)$ coincide with the Steenrod squares on $H^k(X)$.}
Moreover, the intersection of kernels being endowed with the induced structure of $\cE(2)$-algebra of $N^*(X\times [-1,1])$,
the quasi-isomorphism
$N^*(\Sigma X)\to (\ker\iota_{1}^*\cap\ker\iota_{-1}^*)$
 is a morphism of $\cE(2)$-algebras, see \cite{MR2075046}.
Thus, \emph{in degrees $k>\ov{p}(n)+1$, the Steenrod squares
on $H^k_{\TW,\ov{p}}(\Sigma X)$ coincide with the Steenrod squares on $H^k(\Sigma X)$,} 
which are the suspensions of the Steenrod squares on $X$.
\end{example}

We consider now the case of the Thom space of a vector bundle, $\R^m\to E\to B$.

\begin{example}[\emph{Steenrod squares on the intersection cohomology of a Thom space}]
Let $\R^m\to D_{E}\xrightarrow[]{g} B$ be the disk-bundle of associated sphere-bundle
$S^{m-1}\to S_{E}\xrightarrow[]{f} B$. 
The \emph{Thom space,} $\Th(E)$, is built from the disk-bundle along the process described in \propref{prop:isolated}.
 We filter $\Th(E)$ by the point of compactification. 
 Let $\ov{p}$ be a GM-perversity entirely determined in this case by the number $\ov{p}(n)$ with $n=\dim E$.
 \emph{In this example, we prove that the Steenrod squares on
 $H^*_{\TW,\ov{p}}(\Th(E))$ are entirely determined by the Steenrod squares on the base space
  and the Stiefel-Whitney classes of the bundle.
 }
 
 Denote by $c\in H^m(B)$ the Euler class and by $\theta\in H^m(\Th(E))$ the Thom class. 
 Let $j\colon D_{E}\to \Th(E)$ be the canonical map and recall that
 the Thom isomorphism, ${\mathcal Th}\colon H^{k-m}(B)\to H^k(\Th(E))\cong H^k(D_{E},S_{E})$, is defined by
${\mathcal Th}(\gamma)= g^*(\gamma)\cup \theta$.
 The Euler and the Thom  classes are connected by the two exact sequences,
 $$\xymatrix{
 \ldots\ar[r]&
 H^k(\Th(E))\ar[r]^-{j^*}&
 H^k(D_{E})\ar[r]&
 H^k(S_{E})\ar[r]&
 H^{k+1}(\Th(E))\ar[r]&
 \ldots\\
 \ldots\ar[r]&
 H^{k-m}(B)\ar[r]^-{-\cup c}\ar[u]^{{\mathcal Th}}&
 H^k(B)\ar[r]^-{f^*}\ar[u]^-{g^*}&
 H^k(S_{E})\ar[r]\ar@{=}[u]&
 H^{k+1-m}(B)\ar[r]\ar[u]^{\mathcal Th}&
 \ldots
 }$$
and $j^*(\theta)=g^*(c)$. 
From  \propref{prop:isolated}, we know that the complex, $\tN^*_{\ov{p}}(\Th(E))$, is quasi-isomorphic to
\begin{eqnarray*}
\cN^*&=&
N^*(B)\oplus_{N^*(S_{E})}\tau_{\leq \ov{p}(n)}N^*(S_{E})\\
&\cong&
N^{<\ov{p}(n)}(B)\oplus \{\alpha\in N^{\ov{p}(n)}(B)\mid df^*(\alpha)=0\}\\
&&\hskip 1.9cm \oplus (\ker(N^k(B)\xrightarrow[]{f^*}N^k(S_{E}))^{>\ov{p}(n)}.
\end{eqnarray*}
Thus, we recover (see \cite[Page~77]{MR2207421})  the intersection cohomology of the Thom space,
$$H_{\TW,\ov{p}}^k(\Th(E))=H^k(\cN)=\left\{
\begin{array}{lcl}
H^k(B)
&\text{if}&
k\leq \ov{p}(n),\\
(\im {(-\cup c)})^k
&\text{if}&
k=\ov{p}(n)+1,\\
H^{k-m}(B)\cong_{\mathcal Th}  H^k(\Th(E))
&\text{if}&
k>\ov{p}(n)+1.
\end{array}\right.$$

$\bullet$ \emph{In the case $k\leq \ov{p}(n)+1$, the Steenrod squares, $\sq^i\colon H_{\ov{p}}^k(\Th(E))\to H_{\cL(\ov{p},i)}^{k+i}(\Th(E))$,} coincide with the Steenrod squares, $\sq^i\colon H^k(B)\to H^{k+i}(B)$, cf. \remref{rem:coherenceofSq}. 

$\bullet$ \emph{Let $k>\ov{p}(n)+1$} and $\gamma\in H^{k-m}(B)$. The (classical) internal Cartan formula gives, 
\begin{eqnarray*}
\sq^j(g^*(\gamma)\cup\theta)&=&
\sum_{\ell=0}^j  \sq^{j-\ell}(g^*(\gamma))\cup \sq^{\ell}(\theta) \\
&=&
\sum_{\ell=0}^j  g^*(\sq^{j-\ell}(\gamma))\cup g^*(\omega_{\ell})\cup \theta,
\end{eqnarray*}
where the $\omega_{\ell}$'s are the Stiefel-Whitney classes of the fibration $f$, see \cite[Page 91]{MR0440554}.
Set  $\mu=g^*(\gamma)\cup \theta={\mathcal Th}(\gamma)\in H^k(\Th(E))$.
In this range of degrees, the Steenrod squares on $\Th(E)$, denoted by $\sq_{\Th}$, and the Steenrod squares on $B$, denoted by $\sq_{B}$, are related by
$$\sq^j_{\Th}(\mu)=\sum_{\ell=0}^j  g^*(\sq_{B}^{j-\ell}(\gamma)\cup \omega_{\ell})\cup \theta.$$
With the Thom isomorphism, ${\mathcal Th}\colon  H^{k-m}(B)\to H^k(\Th(E))$, the previous formula can be written as,
$$\sq^j_{\Th}(\mu)={\mathcal Th}\left(\sum_{\ell=0}^j \sq_{B}^{j-\ell}( {\mathcal Th}^{-1}(\mu))\cup \omega_{\ell}\right).
$$
\end{example}

%%%%%%%%%%%%%%%%%%%%%%%%
\section{Example of a fibration with fiber a cone}\label{sec:the example}

In this section, we construct an example showing the interest of the lifting of the image of $\sq^i$ to the perversity $\cL(\ov{p},i)$ instead of $2\ov{p}$. As the case of $\sq^1$ was analyzed in \cite{MR1014465}, we choose an example with $\sq^2$. 

\begin{proposition}\label{prop:theexample}
There exists a pseudomanifold $X$ and a GM-perversity $\ov{p}$, with an explicit \emph{non-trivial} perverse square,
$$\sq^2\neq 0\colon H_{\TW,\ov{p}}^6(X)\to H_{\TW,\cL(\ov{p},2)}^8(X),$$
whose composition with the canonical map
$H_{\TW,\cL(\ov{p},2)}^8(X)\to H_{\TW,2\ov{p}}^8(X)$
is zero.
\end{proposition}

\begin{proof}
To begin with, we describe the general strategy of the proof. The first step is the construction of a fibration,
$S^7\times S^4\to E \xrightarrow{\varphi} \C P(2)$,
with a non-trivial differential on a generator $a_{7}$ of $H^7(S^7\times S^4)$,
in the Serre spectral sequence. 
Secondly, we consider the fiberwise conification,
$c(S^7\times S^4)\to X\xrightarrow[]{\psi} \C P(2)$,
of the fibration $\varphi$.
The space $X$ is a pseudomanifold. A GM-perversity, $\ov{p}$, on $X$ is determined by the value $\ov{p}(12)=k$
and we denote it by $\ov{k}$.
(As $\ov{p}$ is a GM-perversity, we have $k\leq 10$.) 
In our fibration, depending on the value of $k$, the element $a_{7}$ is a class of $\ov{p}$-intersection or not;
more precisely, we get
$H^8_{\TW,\ov{k}}(X)\neq 0$ if $k=6$ and $H^8_{\TW,\ov{k}}(X)=0$ if $k=8$.
This property generates a non-trivial Steenrod square, $\sq^2\colon H^6_{\TW,\ov{p}}(X)\to H^8_{\TW,\cL(\ov{p},2)}(X)$,
such that the composite with the canonical map,
$H^6_{\TW,\ov{p}}(X)\xrightarrow[]{\sq^2} H^8_{\TW,\cL(\ov{p},2)}(X)\to H^8_{\TW,2\ov{p}}(X)$,
is the zero map. Details are as follows.

$\bullet$ First, we observe, from the cellular approximation theorem and the construction of $K(\Z,8)$, that the classifying map of the top class,
$\C P(2)\times S^4\to K(\Z,8)$,
lifts as a map $f\colon \C P(2)\times S^4\to S^8$. We denote by $p_{1}\colon E\to \C P(2)\times S^4$ the pullback of the Hopf fibration, $S^{15}\to S^8$, along $f$. We compose $p_{1}$ with the trivial fibration, $p_{2}\colon \C P(2)\times S^4\to \C P(2)$ and obtain a fibration
$$\varphi\colon E\to \C P(2),$$
whose  fiber, $F$, is $S^7\times S^4$. To show this last point, consider the next commutative diagram:
$$\xymatrix@=6pt{
F\ar[rr]\ar[dd]&&E\ar[rr]\ar[dd]^{p_{1}}&&S^{15}\ar[dd]\\
&\fbox{2}&&\fbox{3}&\\
S^4\ar[rr]\ar[dd]&&\C P(2)\times S^4\ar[rr]^-f\ar[dd]^{p_{2}}&&S^8\\
&\fbox{1}&&&\\
\ast\ar[rr]&&\C P(2)&&
}$$
The rectangle formed of \fbox{1} and \fbox{2} is a pullback. As \fbox{1} is a pullback, we deduce (\cite[Section III.4]{MR1712872}) that \fbox{2} is a pullback. Therefore, the rectangle formed of \fbox{2} and \fbox{3} is a  pullback and the triviality of the map $S^4\to S^8$ implies that $F$ is $S^7\times S^4$. 

We study now the Serre spectral sequence of the fibration $\varphi$. We denote by $a_{4}$, $a_{7}$ and $a_{7}\times a_{4}$ the generators of the reduced cohomology of $S^7\times S^4$ and by $x$ and $x^2$ the generators of the reduced cohomology of $\C P(2)$.
An inspection of the degrees in the differentials,
$d_{r}\colon E_{r}^{s,t}\to E_{r}^{s+r,t-r+1}$,
shows that the only differential which can be potentially non-trivial is
$$d_{4}\colon E_{4}^{0,7}=E_{2}^{0,7}=\F_{2} a_{7}\to E_{4}^{4,4}=E_{2}^{4,4}=\F_{2}(x^2\otimes a_{4}).$$
By definition of $S^7\to E\to \C P(2)\times S^4$ as a pullback of the Hopf fibration, we already know 
(\cite[Section III.4]{MR0045386})
that the top class $a_{7}$ of $S^7$ transgresses on the product $x^2\times a_{4}$. 
This gives $d_{4}(a_{7})=x^2\otimes a_{4}$ in the Serre spectral sequence of the fibration $\varphi\colon E\to \C P(2)$.

We continue with the determination of the image of the cohomology class $x\otimes a_{4}$ by $\sq^2$ in $H^*(\C P(2))\otimes H^*(S^7\times S^4)$. From the external Cartan formula, we have
$$\sq^2(x\otimes a_{4})=\sq^2(x)\otimes a_{4}+\sq^1(x)\otimes\sq^1(a_{4})+x\otimes \sq^2(a_{4}).$$
The last two terms are zero, for degree reasons. The equality $\sq^2(x)=x^2$ gives 
$$\sq^2(x\otimes a_{4})=x^2\otimes a_{4}.$$

$\bullet$ The second step is the fiberwise conification, $c(S^7\times S^4)\to X
\xrightarrow[]{\psi}
\C P(2)$, of the  fibration $\varphi$.
If $x\in \C P(2)$, we denote by $(S^7\times S^4)_{x}$ the fiber over $x$ and by $\tv_{x}$ the cone point of the cone
$c((S^7\times S^4)_{x})$.
A continuous section $\mu$ of $\psi$, defined by $\mu(x)=\tv_{x}$, identifies $\C P^2$ to a closed subspace of ${X}$.
We filter $X$ by $\emptyset\subset X_{0}=\C P(2)\subset X$. Observe that the singular set in $X$ is  $\C P(2)$ and that the link of a singular point is $S^7\times S^4$. 

Let $\ov{k}$ be a GM-perversity. 
The intersection cohomology, $H^*_{\TW,\ov{k}}(X)$, is the abutment 
(see \cite[Theorem 3.5]{MR2354985}) of a Serre spectral sequence with
$$\mbox{}_{\ov{k}}E_{2}^{r,s}=H^r(\C P(2))\otimes H^s_{\TW,\ov{k}}(c(S^7\times S^4)).$$
We may replace the right-hand term of this tensor product by its value and obtain
$$\mbox{}_{\ov{k}}E_{2}^{r,s}=H^r(\C P(2))\otimes H^{s}(S^7\times S^4),$$
if $s\leq k$ and 0 otherwise. The existence of a morphism, $E\to X$, over the identity on $\C P(2)$, gives a morphism of spectral sequences,
$(\mbox{}_{\ov{k}}E_{*}^{r,s},d_{*})\to  (E_{*}^{r,s},d_{*})$. From our previous determination of the Serre spectral sequence, 
$(E_{*}^{r,s},d_{*})$, associated to the fibration $\varphi\colon E\to \C P(2)$, we deduce that the differentials $d_{*}$ of 
$\mbox{}_{\ov{k}}E_{*}^{r,s}$ are zero, except $d_{4}(a_{7})=x^2\otimes a_{4}$, if $7\leq k$. 
Thus, in perversity $k<7$, as the class $a_{7}$ is not of $\ov{k}$-intersection, 
the class $x^2\otimes a_{4}$ survives and $H^8_{\TW,\ov{k}}(X)\neq 0$. 
But, if $k=8$, the class $a_{7}$ is of $\ov{8}$-intersection and kills the element $x^2\otimes a_{4}$ (which is the only element of degree 8 in the $E_{2}$-term). Thus $H^8_{\TW,\ov{8}}(X)=0$.

The square $\sq^2$, that we have previously determined, arises in the GM-perversity~$\ov{4}$ and we have
$$\sq^2\colon H^6_{\TW,\ov{4}}(X)=\F_{2}(x\otimes a_{4})\to H^8_{\TW,\ov{4}}(X)=\F_{2}(x^2\otimes a_{4}).$$
Observe that $\ov{6}=\cL(\ov{4},\ov{4}+2)$ is a GM-perversity and thus, with the argument above, $\sq^2$ still survives as map from $H^6_{\TW,\ov{4}}$ to $H^8_{\TW,\ov{6}}=H^8_{\TW,\ov{4}}$. 
But, for the GM-perversity $\ov{8}=2\,\times \ov{4}$, as $H^8_{\TW,\ov{8}}(X)=0$,  this square $\sq^2$ disappears if we express it as a map from
$H^6_{\TW,\ov{4}}$ to $H^8_{\TW,2\,\times\ov{4}}$.
\end{proof}

%%%%%%%%%%%%%%%%%%%%%%%
\section{Topological invariance of the Steenrod squares in intersection cohomology}\label{sec:topinvariance}

In the case of PL-pseudomanifolds, we know from \cite{MR761809} that the Steenrod squares are topological invariants,
as homomorphisms
$H^r_{\TW,\ov{p}}(X)\to H_{\TW,2\ov{p}}^{r+i}(X)$.
In this section, we prove that the lifting we have introduced before,
$\sq^i\colon H^r_{\TW,\ov{p}}(X)\to H^{r+i}_{\TW,\cL(\ov{p},i)}(X)$,
is also a topological invariant. The proof is based on the original combinatorial description of Steenrod squares made in 
\cite{MR0022071}.

\begin{theorem}\label{thm:topinvariance}
Let $X$ be an $n$-dimensional PL-pseudomanifold and $\ov{p}$ be a GM-perversity. Then, the Steenrod squares,
$\sq^i\colon H^*_{\TW,\ov{p}}(X)\to H^{*+i}_{\TW,\cL(\ov{p},i)}(X)$,
do not depend on the stratification of $X$.
\end{theorem}

\thmref{thm:topinvariance} is a direct consequence of \propref{prop:XandXqi} and \propref{prop:XandXsteenrod}. Before stating and proving these two results, we need to introduce  some material.
First, recall from \cite[Page 150]{MR800845} and \cite[Chapter~2]{IHGreg}, 
the existence of a PL-pseudomanifold, $X^*$, 
which is an intrinsic coarsest stratification of $X$, together with a stratified map, $\nu\colon X\to X^*$, 
defined by the identity map, see \cite[Definition A.18]{2012arXiv1205.7057C}. 
In \cite{MR800845}, H. King proves that $\nu$ induces a quasi-isomorphism 
between the Goresky-MacPherson chain (and cochain) complexes.  
Here we consider the map $\chi$, induced by $\nu$ between the Thom-Whitney complexes. 

\begin{proposition}\label{prop:XandXqi}
Let $X$ be an $n$-dimensional PL-pseudomanifold and $\ov{p}$ be a GM-perversity.
 Then the canonical map, $\nu\colon X\to X^*$, induces a quasi-isomorphism,\linebreak
$\chi\colon \tN^*_{\ov{p}}(X^*)\to \tN^*_{\ov{p}}(X)$.
\end{proposition}

\subsection{Construction of $\chi$, the local step}
Before giving the proof, we detail the \emph{construction of $\chi$,} based on the effect of $\nu\colon X\to X^*$ on filtered simplices of $X$.
Let  $\sigma\colon \Delta=\Delta^{j_{0}}\ast\cdots\ast\Delta^{j_{n}}\to X$ be a filtered simplex of $X$. 
Suppose that, 
\begin{itemize}
\item for some integer $0\leq i\leq n-1$, the set 
$\sigma(\Delta^{j_{0}}\ast\cdots\ast\Delta^{j_{i}})\backslash \sigma(\Delta^{j_{0}}\ast\cdots\ast\Delta^{j_{i-1}})$
 is included in an $i$-stratum of $X$ which ``disappears'' inside an $(i+1)$-stratum of~$X^*$, 
\item for the other indices, $\ell\neq i$, the corresponding strata of $\sigma(\Delta)$ stay unmodified.
\end{itemize}
Then, the filtered simplex $\sigma\colon \Delta=\Delta^{j_{0}}\ast\cdots\ast\Delta^{j_{n}}\to X$ becomes a filtered simplex of $X^*$,
$\nu\circ\sigma\colon \Delta(i)=
\Delta^{k_{0}}\ast\cdots\ast\Delta^{k_{n}}
\to X^*$, with
\begin{equation}\label{equa:deltai}
\left\{\begin{array}{lcl}
k_{\ell}=j_{\ell}&\text{ if }&\ell<i \text{ or } \ell >i+1,\\
k_{i}=-1&\text{ and } &k_{i+1}=j_{i}+j_{i+1}+1.
\end{array}\right.
\end{equation}
This process is called an \emph{elementary amalgamation.}
In general, the simplex $\nu\circ\sigma\colon \Delta\to X^*$ can be written as a filtered simplex after a finite number of elementary amalgamations. 
As we work with blow-ups,  we need to consider two cases, depending if $i+1=n$ or not. We write
$$\tN^*(\Delta)=N^*(c\Delta^{j_{0}})
\otimes\cdots\otimes
N^*(c\Delta^{j_{i}})\otimes N^*(c\Delta^{j_{i+1}})
\otimes\cdots\otimes
N^*(\Delta^{j_{n}})$$
and
$$\tN^*(\Delta(i))=\left\{
\begin{array}{ll}
N^*(c\Delta^{j_{0}})
\otimes\cdots\otimes
N^*(c\emptyset)\otimes N^*(c\Delta^{j_{i}+j_{i+1}+1})
\otimes\cdots\otimes
N^*(\Delta^{j_{n}}),
&
\text{ if }i\neq n-1,\\
N^*(c\Delta^{j_{0}})
\otimes\cdots\otimes
N^*(c\emptyset)\otimes N^*(\Delta^{j_{n-1}+j_{n}+1}),
&
\text{ if }i= n-1.
\end{array}\right.$$

We define below two morphisms,
$$\alpha\colon N^*(c\Delta^{a+b+1})\to N^*(c\Delta^a)\otimes N^*(c\Delta^b)
\text{ and }
\beta \colon N^*(c\Delta^{a+b+1})\to N^*(c\Delta^a)\otimes N^*(\Delta^b),$$
which correspond to the cases $i\neq n-1$ and $i=n-1$. 

Let $\tv$ be the cone point of $c\emptyset$. We use $\alpha$ and $\beta$ for the definition of a morphism
$\xi_{i}\colon \tN^*(\Delta(i))\to \tN^*(\Delta)$ as follows.
If $\Phi=\sum_{j}\Phi_{0,j}\otimes\cdots\otimes \Phi_{n,j}\in \tN^*(\Delta(i))$, we set
\begin{itemize}
\item for $i\neq n-1$,
\begin{equation}\label{equa:xii}
\xi_{i}(\Phi)=\sum_{j}\Phi_{i,j}([\tv])\cdot\Phi_{0,j}\otimes\cdots\otimes
\Phi_{i-1,j}\otimes \alpha(\Phi_{i+1,j})\otimes \Phi_{i+2,j}\otimes\cdots\otimes \Phi_{n,j},
\end{equation}
\item for $i= n-1$,
\begin{equation}\label{equa:xin}
\xi_{i}(\Phi)=\sum_{j}\Phi_{n-1,j}([\tv])\cdot\Phi_{0,j}\otimes\cdots\otimes\Phi_{n-2,j}\otimes \beta(\Phi_{n,j}).
\end{equation}
\end{itemize}
These $\xi_{i}$'s are the local ingredients used in the (global) definition of $\chi$, stated below.

\subsection{Construction of $\alpha\colon N^*(c\Delta^{a+b+1})\to N^*(c\Delta^a)\otimes N^*(c\Delta^b)$}
We define $\alpha$ by its values on the elements of a basis. 
If $L$ is one of the simplicial complexes, $c\Delta^a$, $c\Delta^b$ or $c\Delta^{a+b+1}$, we denote by $\{1_{F}\}$ the dual basis of $N^*(L)$ obtained from the basis of faces, $F$, of $L$.

If we represent by $F_{a}$ the faces of $\Delta^a$ and by $F_{b}$ the faces of $\Delta^b$, a face of $c\Delta^{a+b+1}$
is of the type $c(F_{a}\ast F_{b})$ or $F_{a}\ast F_{b}$, where $F_{a}$ and $F_{b}$ can also be the emptyset. 
A linear map $\alpha$ is entirely determined by
$$\left\{\begin{array}{lcll}
\alpha(1_{c(F_{a}\ast F_{b})})&=&
1_{cF_{a}}\otimes 1_{cF_{b}}, &\text{ the cases } F_{a}=\emptyset, F_{b}=\emptyset \text{ being included,}
\\
\alpha(1_{F_{a}\ast F_{b}})&=&
1_{cF_{a}}\otimes 1_{F_{b}}, &\text{ if } F_{b}\neq \emptyset, \text{ the case } F_{a}=\emptyset \text{ being included,}
\\
\alpha(1_{F_{a}})&=&
1_{F_{a}}\otimes 1_{\tv_{b}} + 1_{F_{a}}\otimes 1_{\mathbb V_{b}},
\end{array}\right.$$
where $1_{\mathbb V_{b}}$ is the sum of $1_{\tp}$ when $\tp$ runs in the set of vertices of $\Delta^b$ and $\tv_{b}$ is the
cone point of $c\Delta^b$.

\subsection{Construction of $\beta \colon N^*(\Delta^{a+b+1})\to N^*(c\Delta^a)\otimes N^*(\Delta^b)$} With the previous notation, the linear map $\beta$ is defined by
$$\left\{\begin{array}{lcll}
\beta(1_{F_{a}\ast F_{b}})&=&
1_{cF_{a}}\otimes 1_{F_{b}}, &\text{ if } F_{b}\neq \emptyset, \text{ the case } F_{a}=\emptyset \text{ being included,}
\\
\beta(1_{F_{a}})&=&
1_{F_{a}}\otimes 1_{\mathbb V_{b}}.
\end{array}\right.$$

These maps verify the next properties whose proofs are postpone after the proof of \propref{prop:XandXqi}.

\begin{lemma}\label{lem:alphabeta}
The two morphisms, $\alpha\colon N^*(c\Delta^{a+b+1})\to N^*(c\Delta^a)\otimes N^*(c\Delta^b)$ and\linebreak
 $\beta \colon N^*(\Delta^{a+b+1})\to N^*(c\Delta^a)\otimes N^*(\Delta^b)$,
 are compatible with the differentials and the restrictions to faces of $\Delta^a$ and $\Delta^b$.
\end{lemma}

\begin{lemma}\label{lem:ksii}
The morphism
$\xi_{i}\colon \tN^*(\Delta(i))\to \tN^*(\Delta)$ is compatible  with the differentials and the restrictions to faces of the $\Delta^{j_{\ell}}$'s. Moreover, it respects the perverse degree, i.e.,
$\xi_{i}(\tN^*_{\ov{p}}(\Delta(i)))
\subset
\tN^*_{\ov{p}}(\Delta)$, for any GM-perversity, $\ov{p}$.
\end{lemma}

\subsection{Construction of $\chi\colon \tN^*(X^*)\to \tN^*(X)$, the global step}

Let $\sigma\colon \Delta_{\sigma}=\Delta^{j_{0}}\ast\cdots\ast \Delta^{j_{n}}\to X$ be a filtered simplex of $X$, of blow-up 
$\tdelta_{\sigma}=c\Delta^{j_{0}}\times\cdots\times \Delta^{j_{n}}$. 
As we have noted before, the domain of the \emph{filtered} simplex,
$\nu\circ \sigma\colon \Delta_{\nu\circ\sigma}=\Delta^{k_{0}}\ast\cdots\ast \Delta^{k_{n}}\to X^*$,
 has a different decomposition, obtained by a succession of elementary amalgamations. 
 We denote by $\tdelta_{\nu\circ\sigma}$ the associated blow-up.
 
 These elementary amalgamations give a finite sequence of decompositions, $\Delta(i_{\ell})_{0\leq \ell\leq m}$,
 such that $\Delta(0)=\Delta^{j_{0}}\ast\cdots\ast \Delta^{j_{n}}=\Delta_{\sigma}$
 and
 $\Delta(m)=\Delta^{k_{0}}\ast\cdots\ast \Delta^{k_{n}}=\Delta_{\nu\circ\sigma}$.
  Two consecutive terms correspond to an elementary amalgamation, i.e.,
 $\Delta(i_{\ell})=\Delta^{x_{0}}\ast\cdots\ast\Delta^{x_{n}}$
 and
 $\Delta(i_{\ell +1})=\Delta^{y_{0}}\ast\cdots\ast\Delta^{y_{n}}$,
 with
 $$\left\{\begin{array}{lcl}
y_{u}=x_{u}&\text{ if }&u<i_{\ell} \text{ or } u >i_{\ell}+1,\\
y_{i_{\ell}}=-1&\text{ and } &y_{i_{\ell}+1}=x_{i_{\ell}}+x_{i_{\ell+1}}+1.
\end{array}\right.$$
Recall the map $\xi_{i_{\ell}} \colon \tN^*(\Delta(i_{\ell+1}))\to \tN^*(\Delta(i_{\ell}))$ defined in (\ref{equa:xii}) and (\ref{equa:xin}). We set,
$$\chi_{\sigma}=\xi_{i_{0}}\circ\cdots\circ\xi_{i_{m-1}}.$$
Finally, with \lemref{lem:alphabeta}, we have a map, $\chi\colon \tN^*(X^*)\to \tN^*(X)$, defined 
on $\sigma\colon \Delta\to X$ and $\Phi\in \tN^*(X^*)$, by
$$\chi(\Phi)_{\sigma}=\chi_{\sigma}(\Phi_{\nu\circ\sigma}).$$

\begin{proof}[Proof of \propref{prop:XandXqi}]
With \lemref{lem:ksii}, the previous map,
$\chi\colon \tN^*(X^*)\to \tN^*(X)$,
is a cochain map which restricts as
$\chi\colon \tN^*_{\ov{p}}(X^*)\to \tN^*_{\ov{p}}(X)$.

\smallskip
There exists also a map, $\ev_{N}\colon \tN^*(X)\to \hom(N_{*}^{\GM}(X),\F_{2})$, defined as follows.\\
For any $\Phi\in \tN^*(X)$ and 
$\sigma\colon \Delta^{j_{0}}\ast\cdots\ast\Delta^{j_{n}}\to X$,
with $\Phi_{\sigma}=\sum_{j}\Phi_{0,j}\otimes\cdots\otimes \Phi_{n,j}\in
N^*(c\Delta^{j_{0}})\otimes\cdots\otimes N^*(\Delta^{j_{n}})$, we set
$$\ev_{N} (\Phi)(\sigma)=\sum_{j}
\Phi_{0,j}([c\Delta^{j_{0}}])\cdot\ldots\cdot
\Phi_{n,j}([\Delta^{j_{n}}]).$$ 
Let $\ov{q}$ be the GM-perversity such that $\ov{p}+\ov{q}=\ov{t}$.
The canonical morphism, $\rho_{*}\colon C_{*}(-)\to N_{*}(-)$,
induces $\rho^*\colon \hom(N_{*}^{\GM,\ov{q}}(-),\F_{2})\to \hom(C_{*}^{\GM,\ov{q}}(-),\F_{2})$
and 
$\widetilde{\rho}\colon \tN^*_{\ov{p}}(-)\to \tC^*_{\ov{p}}(-)$.
The previous map, $\ev_{N}$  is connected with the morphism $\ev$ introduced in the proof of \lemref{lem:Upetits} by
$\rho^*\circ \ev_{N}=\ev\circ \widetilde{\rho}$.
As $\widetilde{\rho}$ and $\ev$ are quasi-isomorphisms (cf. \propref{prop:CandN} and \propref{prop:GMetTW}),
we know that the composite
$\rho^*\circ \ev_{N}$ is a quasi-isomorphism. 
Consider now the following diagram,
\begin{equation}\label{equa:petittableau}
\xymatrix{
\tN^*_{\ov{p}}(X^*)\ar[r]^-{\ev_{N}}\ar[d]_{\chi}&
\hom(N_{*}^{\GM,\ov{q}}(X^*),\F_{2})\ar[r]^{\rho^*}\ar[d]^{N^*(\nu)}&
\hom(C_{*}^{\GM,\ov{q}}(X^*),\F_{2})\ar[d]^{C^*(\nu)}\\
\tN^*_{\ov{p}}(X)\ar[r]^-{\ev_{N}}&
\hom(N_{*}^{\GM,\ov{q}}(X),\F_{2})\ar[r]^{\rho^*}&
\hom(C_{*}^{\GM,\ov{q}}(X),\F_{2}).
}\end{equation}
The right-hand square is  commutative by naturality of $\rho_{*}$. 
We prove now the commutativity of the left-hand one.
Let $\Phi\in \tN^*_{\ov{p}}(X^*)$ and 
$\sigma\colon \Delta_{\sigma}=\Delta^{j_{0}}\ast\cdots\ast \Delta^{j_{n}}\to X$ 
be a filtered simplex, of associated filtered simplex 
$\nu\circ\sigma\colon \Delta_{\nu\circ\sigma}\to X^*$.
We have to check
\begin{equation}\label{equa:commut}
(N^*(\nu)\circ \ev_{N}(\Phi))(\sigma) = \ev_{N}(\chi(\Phi))(\sigma).
\end{equation}

For a given $\sigma$, we can decompose $\nu$ in a finite number of elementary amalgamations 
and thus replace 
$\chi_{\sigma}\colon \tN^*(\Delta_{\nu\circ\sigma})\to
\tN^*(\Delta_{\sigma}) 
$
by $\xi_{i}\colon \tN^*(\Delta(i))\to \tN^*(\Delta)$, as defined in 
(\ref{equa:xii}) and (\ref{equa:xin}).
 Set 
$\Phi_{\nu\circ\sigma}=\sum_{j}\Phi_{0,j}\otimes\cdots\otimes \Phi_{n,j}\in \tN^*(\Delta(i))$
and suppose $i\neq n-1$.
By definition (\ref{equa:xii}), we have
\begin{eqnarray*}
\chi(\Phi)_{\sigma}&=&\chi_{\sigma}(\Phi_{\nu\circ\sigma})=\xi_{i}(\Phi_{\nu\circ\sigma})\\
&=&\sum_{j}\Phi_{i,j}([\tv])\cdot\Phi_{0,j}\otimes\cdots\otimes
\Phi_{i-1,j}\otimes \alpha(\Phi_{i+1,j})\otimes \Phi_{i+2,j}\otimes\cdots\otimes \Phi_{n,j}
\end{eqnarray*}
and the right-hand side of (\ref{equa:commut}) is equal to
\begin{eqnarray*}
\ev_{N}(\chi(\Phi))(\sigma)&=&\sum_{j}\Phi_{0,j}([c\Delta^{j_{0}}])\cdots
\Phi_{i-1,j}([c\Delta^{j_{i}}])\cdot \Phi_{i,j}([\tv])\cdot\\
&&\alpha(\Phi_{i+1,j})([c\Delta^{j_{i}}]\otimes [c\Delta^{j_{i+1}}])
\cdots \Phi_{n,j}([\Delta^{j_{n}}]).
\end{eqnarray*}
We determine now the left-hand side of (\ref{equa:commut}),
\begin{eqnarray*}
(N^*(\nu)\circ \ev_{N}(\Phi))(\sigma)
&=&
\ev_{N}(\Phi)(\nu\circ \sigma)\\
&=&
\sum_{j}\Phi_{0,j}([c\Delta^{j_{0}}])
\cdots
\Phi_{i-1,j}([c\Delta^{j_{i-1}}])\cdot \Phi_{i,j}([\tv])\cdot\\
&&
 \Phi_{i+1,j}([c\Delta^{j_{i}+j_{i+1}+1}])
\cdots
 \Phi_{n,j}([c\Delta(j_{n}]).
\end{eqnarray*}
Thus, the left-hand and the right-hand sides coincide by definition of $\alpha$. A similar argument gives the result when $i=n-1$.

\smallskip
We have established above that the two horizontal lines of the commutative square (\ref{equa:petittableau}) are 
quasi-isomorphisms. The right-hand vertical map is a quasi-isomorphism also (cf. \cite{MR800845}). Thus, 
$\chi\colon \tN^*_{\ov{p}}(X^*)\to \tN^*_{\ov{p}}(X)$
is a quasi-isomorphism.
\end{proof}

\begin{proof}[Proof of \lemref{lem:alphabeta}]
We \emph{consider first the map $\alpha$} and its behavior with restriction maps. Let $\nabla_{a}$ and $\nabla_{b}$ be faces of $\Delta^a$ and $\Delta^b$, respectively, including the cases $\nabla_{a}=\emptyset$ or $\nabla_{b}=\emptyset$. Then the following diagram commutes,
$$\xymatrix{
N^*(c\Delta^{a+b+1})\ar[r]^-{\alpha}\ar[d]_{\res}&
N^*(c\Delta^a)\otimes N^*(c\Delta^b)\ar[d]^{\res}\\
N^*(c\nabla^{a+b+1})\ar[r]^-{\alpha}&
N^*(c\nabla^a)\otimes N^*(c\nabla^b),
}$$
where the two vertical maps are given by the restriction map. To verify this assertion, 
we consider two faces, $F_{a}$ of $\nabla_{a}$ and $F_{b}$ of $\nabla_{b}$, 
and check the commutativity for the cochain $1_{c(F_{a}\ast F_{b})}$, the other cases being similar,
$$\res(\alpha(1_{c(F_{a}\ast F_{b})}))=\res(1_{cF_{a}}\otimes 1_{cF_{b}})=1_{cF_{a}}\otimes 1_{cF_{b}}=\alpha(1_{c(F_{a}\ast F_{b})})=
\alpha(\res(1_{c(F_{a}\ast F_{b})})).$$

Now, comes the differential. 
Let ${L}$ be a finite simplicial complex, endowed with a partial order of its vertices
 such that the vertices of any simplex are simply ordered. In the cone, $cL$, the cone point is the greatest element.
In the sequel, we adopt (see \cite[Page 292]{MR0022071}) the 
\subsection{Steenrod's convention:}\label{steenrod}  \emph{A symbol, as $F$, $G$, $\nabla$ will denote, ambiguously, either (1) a simplex of $L$ or (2) the array of vertices of the simplex ordered as in $L$, or (3) the orientation of the simplex determined by this order, or (4) the elementary cochain which attaches +1 to this oriented simplex and 0 to all others. The ambiguity can usually be resolved by examining the context in which the symbol is used.}

With this convention, the definition of 
$\alpha\colon N^*(c\Delta^{a+b+1})\to N^*(c\Delta^a)\otimes N^*(c\Delta^b)$ 
can be written as 
$$\left\{\begin{array}{lcll}
\alpha({c(F_{a}\ast F_{b})})&=&
{cF_{a}}\otimes {cF_{b}}, &\text{ the cases } F_{a}=\emptyset, F_{b}=\emptyset \text{ being included,}
\\
\alpha({F_{a}\ast F_{b}})&=&
{cF_{a}}\otimes {F_{b}}, &\text{ if } F_{b}\neq \emptyset, \text{ the case } F_{a}=\emptyset \text{ being included,}
\\
\alpha({F_{a}})&=&
{F_{a}}\otimes {\tv_{b}} + {F_{a}}\otimes {\mathbb V_{b}},
\end{array}\right.$$
where ${\mathbb V_{b}}$ denotes the sum of  the vertices of $\Delta^b$.
The definition of the coboundary with this convention is also specified in \cite[Page 296]{MR0022071}. 
Let $F_{a}=(a_{0},\ldots,a_{k})$ be a nonempty face of a simplex $\Delta^a$,  we denote by
$cF_{a}=(a_{0},\ldots,a_{k},\tv_{a})$ the face obtained from the adjunction of the cone point $\tv_{a}$. 
It is important to observe that, in this setting, the differential of a face $F$ (view as a cochain) depends on the simplicial complex in which we do the computation.
For instance, the differentials $\da$ in $\Delta^a$ and  $\dca$  in $c\Delta^a$ are linked by
$$\left\{\begin{array}{lcl}
\dca (cF_{a})&=& c(\da F_{a}),\\
\dca F_{a}&=&\da F_{a}+cF_{a},\\
\dca \tv_{a}&=& c{\mathbb V_{a}},
\end{array}\right.$$
where $\mathbb V_{a}$ is the sum of the vertices of $\Delta^a$. If $F_{b}$ is a nonempty face in $\Delta^b$, the differential $\dab$ in $\Delta^a\ast\Delta^b$ is defined by:
$$\left\{\begin{array}{lcl}
\dab (F_{a}\ast F_{b})&=&
(\da F_{a})\ast F_{b}+F_{a}\ast (\db F_{b}),\\
\dab F_{a}&=&
\da F_{a}+ F_{a}\ast {\mathbb V_{b}},\\
\dab F_{b}&=&
\mathbb V_{a}\ast F_{b}+\db F_{b}.
\end{array}\right.$$
The differential on $c(\Delta^a\ast \Delta^b)=(c\Delta^a)\ast\Delta^b$ can be deduced from the combination of the previous equalities; 
we denote it by $\dcab$.
We make uniform the notations by setting $\dcacb$ and $\dcatb$ for the product differentials
on $N^*(c\Delta^a)\otimes N^*(c\Delta^b)$ and $N^*(c\Delta^a)\otimes N^*(\Delta^b)$, respectively.
We verify now the compatibility of $\alpha$ with the differentials, by considering the various cases.
\begin{itemize}
\item \emph{Suppose $F_{a}\neq \emptyset$ and $F_{b}\neq \emptyset$.}
\begin{eqnarray*}
(\dcacb)(\alpha(c(F_{a}\ast F_{b})))
&=&
(\dcacb)(cF_{a}\otimes cF_{b})=
\dca (cF_{a})\otimes cF_{b}+cF_{a}\otimes \dcb (cF_{b})\\
&=&
c\da F_{a}\otimes cF_{b}+cF_{a}\otimes c\db F_{b}\\
&=&
\alpha(c((\da F_{a})\ast F_{b}))+\alpha(c(F_{a}\ast(\db F_{b})))
=\alpha(c(\dab(F_{a}\ast F_{b})))\\
&=&
\alpha(\dcab c(F_{a}\ast F_{b})).
\end{eqnarray*}
\begin{eqnarray*}
(\dcacb)\alpha(F_{a}\ast F_{b})
&=&
(\dcacb)(cF_{a}\otimes F_{b})=
(\dca cF_{a} )\otimes F_{b}+cF_{a}\otimes (\dcb F_{b})\\
&=&
(c\da F_{a})\otimes F_{b}+cF_{a}\otimes (\db F_{b})+cF_{a}\otimes cF_{b}\\
&=&
\alpha((\da F_{a})\ast F_{b})+\alpha(F_{a}\ast (\db F_{b}))+ \alpha(c(F_{a}\ast F_{b}))\\
&=&
\alpha(\dab (F_{a}\ast F_{b})+c(F_{a}\ast F_{b}))\\
&=&
 \alpha(\dcab (F_{a}\ast F_{b})).
\end{eqnarray*}
\item \emph{Suppose $F_{a}\neq \emptyset$ and $F_{b}= \emptyset$.}
\begin{eqnarray*}
(\dcacb)\alpha(cF_{a})
&=&
(\dcacb)(cF_{a}\otimes \tv_{b})=
(\dca cF_{a})\otimes \tv_{b}+cF_{a}\otimes \dcb\tv_{b}\\
&=&
(c\da F_{a})\otimes \tv_{b}+(cF_{a})\otimes c\mathbb V_{b}=
\alpha(c(\da F_{a}+F_{a}\ast \mathbb V_{b}))\\
&=&
\alpha(c(\dab F_{a}))=\alpha(\dcab cF_{a}).
\end{eqnarray*}
\begin{eqnarray*}
(\dcacb)\alpha(F_{a})
&=&
(\dcacb)(F_{a}\otimes \tv_{b}+F_{a}\otimes \mathbb V_{b})
= \dca F_{a}\otimes (\tv_{b} +\mathbb V_{b})+F_{a}\otimes \dcb (\tv_{b}+\mathbb V_{b})\\
&=&
(\da F_{a})\otimes (\tv_{b} +\mathbb V_{b})+cF_{a}\otimes \tv_{b}+cF_{a}\otimes \mathbb V_{b}+0\\
&=&
\alpha(\da F_{a})+\alpha(cF_{a})+\alpha(F_{a}\ast \mathbb V_{b})
=\alpha(\dca F_{a}+cF_{a})\\
&=&
\alpha(\dcab F_{a}).
\end{eqnarray*}
(We have used $\dca F_{a}=\da F_{a}+cF_{a}$ and
$\dcb(\tv_{b}+\mathbb V_{b})=0$. 
For the last one, observe that $\tv_{b}+\mathbb V_{b}$ is the non trivial cocycle in degree~0.)
\item \emph{Suppose $F_{a}= \emptyset$ and $F_{b}\neq \emptyset$.}
\begin{eqnarray*}
(\dcacb)\alpha(cF_{b})
&=&
(\dcacb)(\tv_{a}\otimes cF_{b})=c\mathbb V_{a}\otimes cF_{b}+\tv_{a}\otimes c\db F_{b}\\
&=&
\alpha(c(\mathbb V_{a}\ast F_{b}))+\alpha(c\db F_{b})=\alpha(c\dab F_{b})\\
&=&
\alpha(\dcab cF_{b}).
\end{eqnarray*}
\begin{eqnarray*}
(\dcacb)\alpha(F_{b})
&=&
(\dcacb)(\tv_{a}\otimes F_{b})= c\mathbb V_{a}\otimes F_{b}+ \tv_{a}\otimes \dcb F_{b}\\
&=&
c\mathbb V_{a}\otimes F_{b}+\tv_{a}\otimes \db F_{b}+\tv_{a}\otimes cF_{b}\\
&=&
\alpha(\mathbb V_{a}\ast F_{b})+\alpha(\db F_{b})+\alpha(cF_{b})\\
&=&
\alpha(\dcab F_{b}).
\end{eqnarray*}
\item \emph{Suppose $F_{a}= F_{b}=\emptyset$.}
\begin{eqnarray*}
(\dcacb)\alpha(c\emptyset)
&=&
(\dcacb)(\tv_{a}\otimes \tv_{b})=c\mathbb V_{a}\otimes \tv_{b}+\tv_{a}\otimes c\mathbb  V_{b}\\
&=&
\alpha(c\mathbb V_{a}+c\mathbb V_{b})=\alpha(c\mathbb V_{a+b+1})\\
&=&
\alpha(\dcab c\emptyset).
\end{eqnarray*}
\end{itemize}
\emph{As for the map $\beta \colon N^*(\Delta^{a+b+1})\to N^*(c\Delta^a)\otimes N^*(\Delta^b)$,} its description with Steenrod's convention writes,
$$\left\{\begin{array}{lcll}
\beta({F_{a}\ast F_{b}})&=&
{cF_{a}}\otimes {F_{b}}, &\text{ if } F_{b}\neq \emptyset, \text{ the case } F_{a}=\emptyset \text{ being included,}
\\
\beta({F_{a}})&=&
{F_{a}}\otimes {\mathbb V_{b}}.
\end{array}\right.$$
The proof of its compatibility with restriction maps is totally similar to the proof done for $\alpha$.
Therefore we are reduced to check the compatibility of $\beta$ with the differentials. As before, we list the different cases.
\begin{itemize}
\item \emph{Suppose $F_{a}\neq \emptyset$ and $F_{b}\neq \emptyset$.}
\begin{eqnarray*}
(\dcatb)(\beta(F_{a}\ast F_{b}))
&=&
(\dcatb)(cF_{a}\otimes F_{b})
= (\dca cF_{a})\otimes F_{b}+cF_{a}\otimes (\db F_{b})\\
&=&
(c\da F_{a})\otimes F_{b}+ cF_{a}\otimes (\db F_{b})\\
&=&
\beta((\da F_{a})\ast F_{b})+\beta (F_{a}\ast (\db F_{b}))\\
&=&
\beta(\dab (F_{a}\ast F_{b})).
\end{eqnarray*}
\item \emph{Suppose $F_{a}\neq \emptyset$ and $F_{b}= \emptyset$.}
\begin{eqnarray*}
(\dcatb)(\beta(F_{a}))
&=&
(\dcatb)(F_{a}\otimes \mathbb V_{b})=
(\dca F_{a})\otimes \mathbb V_{b}+ F_{a}\otimes (\db \mathbb V_{b})\\
&=&
(\da F_{a})\otimes \mathbb V_{b}+ (cF_{a})\otimes \mathbb V_{b}+0
=\beta(\da F_{a}+cF_{a})\\
&=&
\beta(\dab F_{a}).
\end{eqnarray*}
\item \emph{Suppose $F_{a}= \emptyset$ and $F_{b}\neq \emptyset$.}
\begin{eqnarray*}
(\dcatb)(\beta(F_{b}))
&=&
(\dcatb)(\tv_{a}\otimes F_{b})=
(c\mathbb V_{a})\otimes F_{b}+\tv_{a}\otimes (\db F_{b})\\
&=&
\beta(\mathbb V_{a}\ast F_{b})+\beta(\db F_{b})=\beta(\dab F_{b}).
\end{eqnarray*}
\end{itemize}
\end{proof}

\begin{proof}[Proof of \lemref{lem:ksii}]
The compatibilities with restriction maps and differentials being local, they are direct consequences of \lemref{lem:alphabeta}. We study now the behavior of $\xi_{i}$ with the perverse degrees.

We continue with the Steenrod's convention and begin with the expression of the perverse degree in this context. 
Let $F=F_{0}\otimes\cdots\otimes F_{n}$ be a tensor product of nonempty faces in 
$\widetilde{\Delta}=c\Delta^{j_{0}}\times\cdots\times c\Delta^{j_{n-1}}\times \Delta^{j_{n}}$.  In Steenrod's convention, we do not distinguish between $F$ and the tensor product of cochains, $1_{F_{0}}\otimes\cdots\otimes 1_{F_{n}}$.
We observe that, if a face $F_{k}$ of $c\Delta^{j_{k}}$ is not included in $\Delta^{j_{k}}$,
then the cochain $1_{F_{k}}$ restricts to 0 on the subcomplex $\Delta^{j_{k}}\times \{1\}$ of $c\Delta^{j_{k}}$.
Therefore, by \defref{def:tranversedegree}, the perverse degree of $F$ is given by
$$\|F_{0}\otimes\cdots \otimes F_{n}\|_{\ell}=\left\{
\begin{array}{cl}
-\infty&
\text{if } F_{n-\ell}\not\subset \Delta^{j_{n-\ell}},\\
|F_{n-\ell+1}|+\cdots+|F_{n}|&
\text{if } F_{n-\ell}\subset \Delta^{j{n-\ell}},
\end{array}\right.$$
for any $\ell \in \{1,\ldots,n\}$. A similar definition occurs for the blow-up $\widetilde{\Delta(i)}$ of $\Delta(i)$.

As $\xi_{i}$ is  compatible with the differentials, it is sufficient to prove that the image of a $\ov{p}$-admissible cochain is $\ov{p}$-admissible.
Let $\nabla=\nabla_{0}\otimes\cdots\otimes\nabla_{n}$ be a tensor product of faces of $\widetilde{\Delta(i)}$ such that
$$\|\nabla_{0}\otimes\cdots\otimes\nabla_{n}\|_{\ell}\leq \ov{p}(\ell), \text{ for any }\ell\in\{1,\ldots,n\}.$$
As we are dealing with $\Delta(i)$ (cf. (\ref{equa:deltai})), we have $\nabla_{i}=c\emptyset$ and,
with the notations of (\ref{equa:xii}), $\Phi_{ij}([\tv])=1$. Thus,
$$\xi_{i}(\nabla_{0}\otimes\cdots\otimes \nabla_{n})=\left\{
\begin{array}{cl}
\nabla_{0}\otimes\cdots\otimes \nabla_{i-1}\otimes 
\alpha(\nabla_{i+1})\otimes
\nabla_{i+2}\otimes\cdots\otimes \nabla_{n},&
\text{if }i\neq n-1,\\
\nabla_{0}\otimes\cdots\otimes \nabla_{n-2}\otimes \beta(\nabla_{n}),&
\text{if }i=n-1.
\end{array}\right.$$
The morphisms $\alpha$ and $\beta$ preserving the dimension of faces, it is sufficient to consider the following cases.
\begin{itemize}
\item Suppose $i\neq n-1$ and $n-\ell =i$. Then $\ell\neq 1$ and we have,
$$\begin{array}{lcl}
\|\xi_{i}(\nabla_{0}\otimes\cdots\otimes \nabla_{n})\|_{\ell}
&=&
\left\{\begin{array}{ll}
-\infty,&\!\!\!\!\!\!\!\!\!\!\!\!\!\!\!\!\!\!\!\!
 \text{ if }  \nabla_{i+1}=\nabla_{a}\ast \nabla_{b} \text{ with } \nabla_{b}\neq\emptyset, \\
&\!\!\!\!\!\!\!\!\!\!\!\!\!\!\!\!\!\!\!\!
\text{ or if }   \nabla_{i+1}=c(\nabla_{a}\ast\nabla_{b}),\\
\max(|\tv|,|\mathbb V|) +|\nabla_{n-\ell+2}|+\cdots+|\nabla_{n}|, &\text{ if } \nabla_{i+1}=\nabla_{a},\\
\end{array}\right.\\[.4cm]
&\leq&
\|\nabla_{0}\otimes\cdots\otimes \nabla_{n}\|_{\ell-1}\leq 
\ov{p}(\ell-1)\leq \ov{p}(\ell).
\end{array}
$$
We have used here that the perversity $\ov{p}$ is order-preserving and $|\tv|=|\mathbb V|=0$ in $c\Delta^{j_{n-\ell+1}}$.
\item Suppose $i\neq n-1$ and $n-\ell =i+1$. We have,
$$\begin{array}{lcl}
\|\xi_{i}(\nabla_{0}\otimes\cdots\otimes \nabla_{n})\|_{\ell}
&=&
\left\{
\begin{array}{cl}
-\infty&
\text{if } \nabla_{i+1}=c(\nabla_{a}\ast \nabla_{b}),\\
|\nabla_{n-\ell+1}|+\cdots+|\nabla_{n}|&
\text{if } \nabla_{i+1}=\nabla_{a}\ast\nabla_{b} \text{ or } \nabla_{i+1}=\nabla_{a},
\end{array}\right.\\[.4cm]
&\leq&
\|\nabla_{0}\otimes\cdots\otimes\nabla_{n}\|_{\ell}
\leq
\ov{p}(\ell).
\end{array}
$$
\item Suppose $i= n-1$ and $\ell =1$. We have,
$$\begin{array}{lcl}
\|\xi_{i}(\nabla_{0}\otimes\cdots\otimes \nabla_{n})\|_{1}
&=&
\left\{
\begin{array}{cl}
-\infty&
\text{if } \nabla_{n}=\nabla_{a}\ast\nabla_{b} \text{ with } \nabla_{b}\neq \emptyset,\\
|\mathbb V|&
\text{if } \nabla_{n}=\nabla_{a},
\end{array}\right.\\[.4cm]
&\leq &
0=\ov{p}(1).
\end{array}
$$
\end{itemize}
\end{proof}

\subsection{Steenrod's definition of $\text{cup}_{i}$-products in \cite{MR0022071}} 
Let ${L}$ be a finite simplicial complex, endowed with a partial order of its vertices
 such that the vertices of any simplex are simply ordered. 
 Let
 $F=(a_{0},\ldots,a_{k})$ and $G=(b_{0},\ldots,b_{\ell})$
 be two (ordered) simplices of ${L}$ and let $i\geq 0$ be an integer. The ordered pair,
 $(F,G)$,
 is called \emph{$i$-regular} if $F$ and $G$ have \emph{exactly} $(i+1)$  vertices in common, 
 $(c_{0},\ldots,c_{i})$,
 such that
 \begin{itemize}
 \item $c_{0}=b_{0}$,
 \item $c_{0}$ and $c_{1}$ are adjacent vertices in $F$,
 \item $\ldots$,
 \item $c_{j}$ and $c_{j+1}$ are adjacent in $F$ if $j$ is even and adjacent in $G$ if $j$ is odd,
 \item $\ldots$,
 \item $c_{i}$ is the last vertex of $F$ if $i$ is even and the last vertex of $G$ if $i$ is odd.
 \end{itemize}
 
 Denote by $F_{0}$ the face of $F$ spanned by its vertices lower than or equal to $c_{0}$ and by $F_{2j}$ the face of $F$ spanned by its vertices greater than or equal to $c_{2j-1}$ \emph{and} lower than or equal to $c_{2j}$, ($0<2j\leq i$).
 If $i$ is odd, let $F_{i+1}$ be the face of $F$ spanned by its vertices greater than or equal to $c_{i}$.
We do a similar decomposition for $G$, denoting by $G_{2j+1}$ ($1\leq 2j+1<i+1$) the face of $G$ spanned by its vertices greater than or equal to $c_{2j}$ \emph{and} lower than or equal to $c_{2j+1}$. If $i$ is even, let $G_{i+1}$ be the face spanned by the vertices greater than or equal to $c_{i}$. This gives the decompositions
 $$F=F_{0}\ast F_{2}\ast\cdots\ast F_{2s}\text{ and }
G=G_{1}\ast G_{3}\ast\cdots\ast G_{2s+(-1)^i},$$
with $2s=i$ if $i$ even and $2s=i+1$ if $i$ is odd.

Now, we denote by $G'_{2j+1}$ the face of $G_{2j+1}$ obtained by deleting the vertices $c_{2j}$ and $c_{2j+1}$. Moreover, if $i$ is even, let $G'_{i+1}$ be the face of $G_{i+1}$ obtained by deleting $c_{i}$. 

\begin{definition}\label{def:steenrodbysteenrod}
We define $F\cup_{i}G=0$ in the group of $(k+\ell -i)$-cochains, if the couple $(F,G)$ is not $i$-regular and, otherwise, by
$$F\cup_{i} G=F_{0}\ast G'_{1}\ast F_{2}\ast G'_{3}\ast\cdots \ast\left\{
\begin{array}{ll}
G'_{i+1}&\text{if } i \text{ even},\\
F_{i+1}&\text{if } i \text{ odd}.
\end{array}\right.$$
\end{definition}

\begin{example} We give an illustration of the cases $i$ even and $i$ odd in low dimensions.

1) The pair $(F=(a_{0},\ldots,a_{k}),G=(b_{0},\ldots,b_{\ell}))$ is $0$-regular, if $a_{k}=b_{0}$.
We write their vertices as follows,
$$\xymatrix{
F&\colon
& a_{0}\ar@{--}[r]
&a_{k}%
\ar@{=}[d]
&
\\
G&\colon
&&b_{0}%
\ar@{--}[r]
&b_{\ell}%
&
}$$
By definition,
$F\cup_{0}G=(a_{0},\ldots,a_{k},b_{1},\ldots,b_{\ell})%
$
is the (classical) cup-product.

2) The pair $(F,G)$ is $1$-regular, if they have two common vertices $(c_{0},c_{1})$ such that the vertices of $F$ and $G$ can be put in two lines, as follows
$$\xymatrix{
F&\colon
& a_{0}\ar@{--}[r]
&a_{k_{0}}=c_{0}\ar@{=}[d]
& a_{k_{0}+1}=c_{1}\ar@{=}[d]\ar@{--}[r]
&a_{k}\\
G&\colon
&&b_{0}=c_{0}\ar@{--}[r]
&b_{\ell}=c_{1}
&
}$$
By definition,
$F\cup_{1}G=(a_{0},\ldots,a_{k_{0}},b_{1},\ldots,b_{\ell-1}, a_{k_{0}+1},\ldots,a_{k})
$.

3) The pair $(F,G)$ is $2$-regular, if they have three common vertices $(c_{0},c_{1},c_{2})$ such that the vertices of $F$ and $G$ can be put in two lines as,
$$\xymatrix{
F&\colon&
a_{0}\ar@{--}[r]
&a_{k_{0}}=c_{0}\ar@{=}[d]
& a_{k_{0}+1}=c_{1}\ar@{=}[d]\ar@{--}[r]
&a_{k}=c_{2}\ar@{=}[d]
&
\\
G&\colon&
&b_{0}=c_{0}\ar@{--}[r]
&b_{\ell_{1}}=c_{1}
&b_{\ell_{1}+1}=c_{2}\ar@{--}[r]
&b_{\ell}
}$$
By definition,
$F\cup_{2}G=(a_{0},\ldots,a_{k_{0}},b_{1},\ldots,b_{\ell_{1}-1}, a_{k_{0}+1},\ldots,a_{k}, b_{\ell_{1}+2},\ldots,b_{\ell})
$.
\end{example}

For the convenience of the reader, we recall the next statement of Steenrod, written with the notations of this paper.

\begin{proposition}[{\cite[Theorem Page 295]{MR0022071}}]\label{prop:steenrodjoin}
The Steenrod squares verify the next properties,
\begin{enumerate}[(1)]
\item $F\cup_{i}(cG)=\left\{
\begin{array}{ll}
c(F\cup_{i}G),&
\text{if } i \text{ even,}\\
0,&\text{if } i \text{ odd,}
\end{array}\right.$
\item $(cF)\cup_{i}G=\left\{
\begin{array}{ll}
0,&
\text{if } i \text{ even,}\\
c(F\cup_{i}G),&\text{if } i \text{ odd,}
\end{array}\right.$
\item $(cF)\cup_{i}(cG)=c(F\cup_{i-1}G)$.
\end{enumerate}
\end{proposition}

The domains of the applications $\alpha$ and $\beta$ are the euclidean simplices
$\Delta^{a+b+1}=\Delta^a\ast\Delta^b$ and $c(\Delta^a\ast\Delta^b)$. Therefore, we need to study the $\text{cup}_{i}$-products in these complexes. The case of the cone, $c(\Delta^a\ast\Delta^b)$, can be deduced from the first one, $\Delta^{a+b+1}$, with \propref{prop:steenrodjoin}.
We order the vertices of $\Delta^{a+b+1}$ such that any vertex of $\Delta^a$ is lower to any vertex of $\Delta^b$. Also, the cone point, $\tv$,  is the greatest element of the set of vertices.
Recall from (\ref{equa:taucup}) the notation,
$$
F\cup_{i}^jG=\left\{\begin{array}{ll}
F\cup_{i}G&\text{ if } j \text{ is even,}\\
G\cup_{i}F&\text{ if } j \text{ is odd.}
\end{array}\right.
$$

\begin{lemma}\label{lem:cupijoin}
Let $F_{a}$, $G_{a}$ be (nonempty) faces of $\Delta^a$, $F_{b}$, $G_{b}$ be (nonempty) faces of $\Delta^b$. 
For any $i> 0$, we have,
in $\Delta^a\ast \Delta^b$,
\begin{equation}\label{equa:joincupi}
(F_{a}\ast F_{b})\cup_{i}(G_{a}\ast G_{b})=
\sum_{i_{1}+i_{2}=i-1} 
(F_{a}\cup_{i_{1}}G_{a})\ast (F_{b}\cup^{i_{1}+1}_{i_{2}}G_{b}).
\end{equation}
 Note that the right-hand side of the equality (\ref{equa:joincupi}) has at most one non-zero term.
\end{lemma}

If we set
$-\cup_{-1}-=0$,
the right hand side of (\ref{equa:joincupi}) is equal to zero in the case $i=0$.
Note also that $F_{a}$ and $G_{b}$ cannot have a common vertex, neither $F_{b}$ and $G_{a}$. 
Therefore, with the hypotheses of \lemref{lem:cupijoin}, the simplices $F_{a}\ast F_{b}$ and $G_{a}\ast G_{b}$ cannot have exactly
one vertex in common and respect the convention on the order of the vertices. 
As a consequence, the equality (\ref{equa:joincupi}) is also true for $i=0$, with the two sides equal to zero.

\begin{proof}
The $\text{cup}_{i}$-product of $F=F_{a}\ast F_{b}$ and $G=G_{a}\ast G_{b}$ is not zero only if 
$F$ and $G$ have $(i+1)$ vertices in common. 
Denote by $(x+1)$, with $0\leq x\leq i$,  the number of vertices in common for $F_{a}$ and $G_{a}$. 
Thus $F_{b}$ and $G_{b}$ have $(i-x)$ vertices in common.
We observe also that the only non-zero term of the right-hand side of (\ref{equa:joincupi}) corresponds to $i_{1}=x$.

\smallskip
\emph{Suppose $i$ even and $x$ odd.}
With the previous notations, we decompose,
\begin{eqnarray*}
F_{a}=F_{a,0}\ast\cdots\ast F_{a,x+1} &\text{and}& G_{a}=G_{a,1}\ast\cdots\ast G_{a,x},\\
F_{b}=F_{b,0}\ast\cdots\ast F_{b,i-x-1}&\text{and}& G_{b}=G_{b,1}\ast\cdots\ast G_{b,i-x}.
\end{eqnarray*}
Thus, we have
\begin{eqnarray*}
(F_{a}\ast F_{b})\cup_{i}(G_{a}\ast G_{b})&=&
F_{a,0}\ast G'_{a,1}\ast\cdots\ast G'_{a,x}\ast
F_{a,x+1}\ast F_{b,0}\ast G'_{b,1}\ast\cdots\ast G'_{b,i-x}\\
&=&
(F_{a} \cup_{x} G_{a})\ast (F_{b}\cup_{i-x-1} G_{b})\\
&=&
(F_{a} \cup_{x} G_{a})\ast (F_{b}\cup^{x+1}_{i-x-1} G_{b}).
\end{eqnarray*}

\smallskip
\emph{Suppose now $i$ even and $x$ even.}
We decompose
$$F_{a}=F_{a,0}\ast\cdots\ast F_{a,x} \text{ and }
G_{a}=G_{a,1}\ast\cdots\ast G_{a,x+1}.$$
Thus, we have
$$F_{a}\cup _{x}G_{a}= F_{a,0}\ast G'_{a,1}\ast\cdots\ast F_{a,x}\ast G'_{a,x+1}.$$
Note that $F_{a}\cup_{x}G_{a}$ contains all the vertices of $F\cup_{i}G$ belonging to $\Delta^a$. 
The first vertex in common between $F_{a}$ and $G_{a}$ is the first vertex of $G_{a}$.
The number of common points in $\Delta^a$ being the odd number $x+1$, the last vertex of
$\Delta^a$ in common must be the last vertex of $F_{a}$. 
Therefore, the first vertex of $\Delta^b$ in common is the first vertex of $F_{b}$.
(See \exemref{exam:xeven} for an illustration of this argument.)
Thus, for writing this final part of vertices
in $F\cup_{i}G$ as a $\text{cup}_{i-x-1}$-product, we have to decompose $F_{b}$ and $G_{b}$ as follows,
$$G_{b}=G_{b,0}\ast\cdots\ast G_{b,i-x} \text{ and }
F_{b}=F_{b,1}\ast\cdots\ast F_{b,i-x-1}.$$
We deduce
\begin{eqnarray*}
(F_{a}\ast F_{b})\cup_{i}(G_{a}\ast G_{b})&=&
F_{a,0}\ast G'_{a,1}\ast\cdots\ast G'_{a,x+1}\ast G_{b,0}\ast F'_{b,1}\ast\cdots\ast F'_{b,i-x-1}\ast G_{b,i-x}\\
&=&
(F_{a} \cup_{x} G_{a})\ast (G_{b}\cup_{i-x-1} F_{b})\\
&=&
(F_{a} \cup_{x} G_{a})\ast (F_{b}\cup^{x+1} _{i-x-1} G_{b}).
\end{eqnarray*}

\smallskip
\emph{In the case $i$ odd,} the conclusion is obtained with totally similar arguments.
\end{proof}

\begin{example}\label{exam:xeven}
We particularize with $x=2$ the argument done in the previous proof.
Let
$F_{a}\ast F_{b}=(f_{0}^a,\ldots,f^a_{\ell})\ast(f^b_{0},\ldots,f^b_{k})$
and
$G_{a}\ast G_{b}=(g_{0}^a,\ldots,g^a_{u})\ast(g^b_{0},\ldots,g^b_{v})$.
The following diagram represents
$(F_{a}\ast F_{b})\cup_{i}(G_{a}\ast G_{b})=
(F_{a}\cup_{2}G_{a})\ast (F_{b}\cup^3_{i-3}G_{b})=
(F_{a}\cup_{2}G_{a})\ast (G_{b}\cup_{i-3}F_{b})$.
$$\xymatrix{
F_{a}\ast F_{b}&\colon&
f^a_{0}\ar@{--}[r]
&%
\ar@{=}[d]
& 
\ar@{=}[d]\ar@{--}[r]
&f^a_{\ell}
\ar@{=}[d]
&&
f^b_{0}\ar@{=}[d]\ar@{--}[r]
&\ar@{=}[d]&
\\
G_{a}\ast G_{b}&\colon&
&
\ar@{--}[r]
&
&
\ar@{--}[r]
&
g^a_{u}\;g^b_{0}\ar@{--}[r]
&&\ar@{--}[r]&
}$$
\end{example}

\begin{proposition}\label{prop:XandXsteenrod}
Let $X$ be an $n$-dimensional PL-pseudomanifold, $\ov{p}$ and $\ov{q}$ be GM-perversities.
Then, the quasi-isomorphism, 
$\chi\colon \tN^*_{\bullet}(X^*)\to \tN^*_{\bullet}(X)$,
induced by $\nu\colon X\to X^*$ is compatible with the $\text{cup}_{i}$-products, i.e.,
$$\chi(\Phi\cup_{i}\Psi)=\chi(\Phi)\cup_{i}\chi(\Psi),$$
for any $i\geq 0$, $\Phi\in \tN^r_{\ov{p}}(X^*)$, $\Psi \in \tN^s_{\ov{q}}(X^*)$ and 
$\Phi\cup_{i}\Psi\in \tN^{r+s-i}_{\ov{p}\oplus\ov{q}}(X^{*})$ .
\end{proposition}

As $\text{cup}_{i}$-products on $\tN^*(-)$ are defined locally, it is sufficient to do the proof for an elementary amalgamation. Thus \propref{prop:XandXsteenrod} is a direct consequence of the next lemma.

\begin{lemma}\label{lem:alphabetacupi}
The two morphisms, $\alpha\colon N^*(c\Delta^{a+b+1})\to N^*(c\Delta^a)\otimes N^*(c\Delta^b)$ and\linebreak
 $\beta \colon N^*(\Delta^{a+b+1})\to N^*(c\Delta^a)\otimes N^*(\Delta^b)$,
 are compatible with the $\text{cup}_{i}$-products.
\end{lemma}

\begin{proof}
Consider the faces $F=F_{a}\ast F_{b}$ and $G=G_{a}\ast G_{b}$ of $\Delta^{a+b+1}$.

$\bullet$ Suppose first $F_{a}\neq \emptyset$, $F_{b}\neq \emptyset$, $G_{a}\neq \emptyset$, $G_{b}\neq \emptyset$ and begin with the map $\beta$. We have to prove,
\begin{equation}\label{equa:casbeta}
\beta((F_{a}\ast F_{b})\cup_{i}(G_{a}\ast G_{b}))=
\beta(F_{a}\ast F_{b})\cup_{i}\beta(G_{a}\ast G_{b}).
\end{equation}

From \lemref{lem:cupijoin} and the definition of $\beta$, we get
$$\beta((F_{a}\ast F_{b})\cup_{i}(G_{a}\ast G_{b}))=
\sum_{i_{1}+i_{2}=i-1}(c(F_{a}\cup_{i_{1}}G_{a})\otimes (F_{b}\cup^{i_{1}+1}_{i_{2}}G_{b}).$$
On the other side, we have
\begin{eqnarray*}
\beta(F_{a}\ast F_{b})\cup_{i}\beta(G_{a}\ast G_{b})
&=_{(1)}&
(cF_{a}\otimes F_{b})\cup_{i}(cG_{a}\otimes G_{b})\\
&=_{(2)}&
\sum_{k=0}^i (cF_{a}\cup_{k}cG_{a})\otimes 
(F_{b}\cup^k_{i-k}G_{b})\\
&=_{(3)}&
\sum_{k=1}^i c(F_{a}\cup_{k-1}G_{a})\otimes (F_{b}\cup^k_{i-k} G_{b})\\
&=&
\sum_{i_{1}+i_{2}=i-1}(c(F_{a}\cup_{i_{1}}G_{a}))\otimes  (F_{b}\cup^{i_{1}+1}_{i_{2}}G_{b}),
\end{eqnarray*}
where $=_{(1)}$ is the definition of $\beta$,  
$=_{(2)}$ comes from the structure of $\cE(2)$-algebra on a tensor product of $\cE(2)$-algebras, 
recalled in \secref{sec:E2algebras},
and $=_{(3)}$ is \cite[Formula (4.3)]{MR0022071}, recalled in \propref{prop:steenrodjoin}. 

\medskip
$\bullet$ With the same restriction, $F_{a}\neq \emptyset$, $F_{b}\neq \emptyset$, $G_{a}\neq \emptyset$, $G_{b}\neq \emptyset$, we study now the map~$\alpha$. 
The arguments are coming from \lemref{lem:cupijoin}, from the structure of $\cE(2)$-algebra on a tensor product and from
\propref{prop:steenrodjoin}, as before. In the sequel, we use them without an explicit recall.
We have only to study the cases where one of the faces contains the cone point, 
the other cases being already verified when we have considered the map $\beta$.
\begin{enumerate}[(i)]
\item Let $c(F_{a}\ast F_{b})$ and $c(G_{a}\ast G_{b})$ be faces of $c\Delta^{a+b+1}$. Then we have,
\begin{eqnarray*}
\alpha(c(F_{a}\ast F_{b})\cup_{i}c(G_{a}\ast G_{b}))
&=&
\alpha(c((F_{a}\ast F_{b})\cup_{i-1}(G_{a}\ast G_{b})))\\
&=&
\alpha\left(c\left(\sum_{i_{1}+i_{2}=i-2} (F_{a}\cup_{i_{1}}G_{a})\ast (F_{b}\cup^{i_{1}+1}_{i_{2}}G_{b})\right)\right)\\
&=&
\sum_{i_{1}+i_{2}=i-2}c(F_{a}\cup_{i_{1}}G_{a})\otimes c (F_{b}\cup^{i_{1}+1}_{i_{2}}G_{b}),
\end{eqnarray*}
and
\begin{eqnarray*}
\alpha(c(F_{a}\ast F_{b}))\cup_{i}\alpha(c(G_{a}\ast G_{b}))
&=&
(cF_{a}\otimes cF_{b})\cup_{i}(cG_{a}\otimes cG_{b})\\
&=&
\sum_{k=0}^i (cF_{a}\cup_{k}cG_{a})\otimes  (cF_{b}\cup^k_{i-k}cG_{b})\\
&=&
\sum_{k=1}^{i-1} c(F_{a}\cup_{k-1}G_{a})\otimes c (F_{b}\cup^k_{i-k-1}G_{b})\\
&=&
\sum_{i_{1}+i_{2}=i-2}
c(F_{a}\cup_{i_{1}}G_{a})\otimes c(F_{b}\cup^{i_{1}+1}_{i_{2}}G_{b}).
\end{eqnarray*}
The compatibility with $\text{cup}_{i}$-products is proved for these faces.
\item Let $c(F_{a}\ast F_{b})$ and $(G_{a}\ast G_{b})$ be faces of $c\Delta^{a+b+1}$. We have to prove that,
\begin{equation}\label{equa:casc}
\alpha(c(F_{a}\ast F_{b})\cup_{i}(G_{a}\ast G_{b}))=
\alpha(c(F_{a}\ast F_{b}))\cup_{i}\alpha(G_{a}\ast G_{b}).
\end{equation}
Observe first,
$\alpha(c(F_{a}\ast F_{b})\cup_{i}(G_{a}\ast G_{b})) =0 \text{ if } i \text{ \emph{even}}$, cf. \propref{prop:steenrodjoin}.
We study now the $\text{cup}_{i}$-product of the images by $\alpha$.
\begin{eqnarray*}
\alpha(c(F_{a}\ast F_{b}))\cup_{i}\alpha(G_{a}\ast G_{b})
&=&
(cF_{a}\otimes cF_{b})\cup_{i}(cG_{a}\otimes G_{b})\\
&=&
\sum_{k=0}^i (cF_{a}\cup_{k}cG_{a})\otimes (cF_{b}\cup^k_{i-k}G_{b}).
\end{eqnarray*}
We study the last right-hand side term in the case $i$ \emph{even.}
\begin{itemize}
\item[-] If $k$ is even, then
$(cF_{b}\cup^k_{i-k}G_{b})=cF_{b}\cup_{i-k}G_{b}=0$, since $i-k$ is even.
\item[-] If $k$ is odd, then
$(cF_{b}\cup^k_{i-k}G_{b})=G_{b}\cup_{i-k}cF_{b}=0$, since $i-k$ is odd.
\end{itemize}
Thus the equation (\ref{equa:casc}) is satisfied for $i$ \emph{even}.

Suppose now that  $i$ is \emph{odd.} The left-hand side of (\ref{equa:casc}) can be developed as,
\begin{eqnarray*}
\alpha(c(F_{a}\ast F_{b})\cup_{i}(G_{a}\ast G_{b}))
&=& 
\alpha(c((F_{a}\ast F_{b})\cup_{i} (G_{a}\ast G_{b})))\\
&=&
\alpha\left(c\left(\sum_{i_{1}+i_{2}=i-1}(F_{a}\cup_{i_{1}}G_{a})\ast  (F_{b}\cup^{i_{1}+1}_{i_{2}}G_{b})\right)\right)\\
&=&
\sum_{i_{1}+i_{2}=i-1}c(F_{a}\cup_{i_{1}}G_{a})\otimes c (F_{b}\cup^{i_{1}+1}_{i_{2}}G_{b}).
\end{eqnarray*}
We consider now the expression of the right-hand side of (\ref{equa:casc}) already obtained,
$$\sum_{k=0}^i (cF_{a}\cup_{k}cG_{a})\otimes (cF_{b}\cup^k_{i-k}G_{b}).$$
\begin{itemize}
\item[-] If $k$ is even, then
$(cF_{b}\cup^k_{i-k}G_{b})=cF_{b}\cup_{i-k}G_{b}=c(F_{b}\cup_{i-k}G_{b})$, since $i-k$ is odd.
\item[-] If $k$ is odd, then
$(cF_{b}\cup^k_{i-k}G_{b})=G_{b}\cup_{i-k}cF_{b}=c(G_{b}\cup_{i-k}F_{b})$, since $i-k$ is even.
\end{itemize}
In conclusion, we have proved,
$cF_{b}\cup^k_{i-k}G_{b}=c(F_{b}\cup_{i-k}^k G_{b})$
and
\begin{eqnarray*}
\alpha(c(F_{a}\ast F_{b}))\cup_{i}\alpha(G_{a}\ast G_{b})
&=&
\sum_{k=1}^i  c(F_{a}\cup_{k-1}G_{a})\otimes c(F_{b}\cup^k_{i-k}G_{b})\\
&=&
\sum_{i_{1}+i_{2}=i-1}c(F_{a}\cup_{i_{1}}G_{a})\otimes c(F_{b}\cup^{i_{1}+1}_{i_{2}}G_{b}).
\end{eqnarray*}
We have established the compatibility with $\text{cup}_{i}$-products in this case.
\item Let $(F_{a}\ast F_{b})$ and $c(G_{a}\ast G_{b})$  be  faces of $c\Delta^{a+b+1}$.
This situation is similar to the previous one.
\end{enumerate}

$\bullet$ We consider now the case where at least one of the subsets,
$F_{a}$, $F_{b}$, $G_{a}$, $G_{b}$,
 is the empty set and begin with the map $\beta$. 
 The verification follows the same routine than above but we cannot apply \lemref{lem:cupijoin} in this situation. 
 Therefore, we prove the compatibility with a direct computation of the two sides of the equality
 (\ref{equa:casbeta}).
 We list the different possibilities with the values of the left-hand side (\lhs) and of the right-hand side (\rhs). 
  If $F_{a}=F_{b}=\emptyset$ or $G_{a}=G_{b}=\emptyset$, 
 the expressions become trivial and we may focus on the cases below. 
 
 Before doing these verifications, we  note that $\mathbb V_{a}$ is a chain of vertices and, if $F_{a}\subset \Delta^a$ 
 is given, one (and only one) of theses vertices, say $a_{t}$, is the first vertex of $F_{a}$. This implies
 $\mathbb V_{a}\cup_{0}F_{a}=(a_{t})\cup_{0}F_{a}=F_{a}$. Similarly, we have
 $F_{a}\cup_{0}\mathbb V_{a}=F_{a}$ and $\mathbb V_{a}$ acts as a neutral element for $-\cup_{0}-$.
 Also, as $\tv_{a}\notin F_{a}$, we have $\tv_{a}\cup_{0}F_{a}=F_{a}\cup_{0}\tv_{a}=0$.

\begin{enumerate}[(1)]
\item $F_{a}=\emptyset$, $F_{b}\neq \emptyset$, $G_{a}=\emptyset$, $G_{b}\neq \emptyset$.\\
$\lhs=\beta(F_{b}\cup_{i}G_{b})=\tv_{a}\otimes (F_{b}\cup_{i}G_{b})$.\\
$\rhs=(\tv_{a}\otimes F_{b})\cup_{i} (\tv_{a}\otimes G_{b})=
(\tv_{a}\cup_{0}\tv_{a})\otimes (F_{b}\cup^0_{i}G_{b})=
\tv_{a}\otimes (F_{b}\cup_{i}G_{b})
$.
\item $F_{a}\neq\emptyset$, $F_{b}= \emptyset$, $G_{a}\neq\emptyset$, $G_{b}= \emptyset$.\\
$\lhs=\beta(F_{a}\cup_{i}G_{a})=(F_{a}\cup_{i}G_{a})\otimes (\tv_{b}+\mathbb V_{b})$.\\
$\rhs=(F_{a}\otimes(\tv_{b}+\mathbb V_{b}))\cup_{i}(G_{a}\otimes(\tv_{b}+\mathbb V_{b})=
(F_{a}\cup_{i}G_{a})\otimes ((\tv_{b}+\mathbb V_{b})\cup^i_{0}(\tv_{b}+\mathbb V_{b}))=
(F_{a}\cup_{i}G_{a})\otimes (\tv_{b} + \mathbb V_{b})$.
\item $F_{a}\neq\emptyset$, $F_{b}= \emptyset$, $G_{a}=\emptyset$, $G_{b}\neq \emptyset$.\\
$\lhs=\beta(F_{a}\cup_{i}G_{b})=0$.\\
$\rhs=(F_{a}\otimes (\tv_{b}+\mathbb V_{b}))\cup_{i}(\tv_{a}\otimes G_{b})=0$, because $F_{a}\cup_{k}\tv_{a}=0$ for any $k$.
\item $F_{a}=\emptyset$, $F_{b}\neq \emptyset$, $G_{a}\neq\emptyset$, $G_{b}= \emptyset$.\\
$\lhs=\beta(F_{b}\cup_{i}G_{a})=0$.\\
$\rhs=\beta(F_{b})\cup_{i}\beta(G_{a})=
(\tv_{a}\otimes F_{b})\cup_{i}(G_{a}\otimes(\tv_{b}+\mathbb V_{b}))=0$, because $\tv_{a}\cup_{k}G_{a}=0$ for any~$k$.
\item $F_{a}=\emptyset$, $F_{b}\neq \emptyset$, $G_{a}\neq\emptyset$, $G_{b}\neq \emptyset$.\\
$\lhs=\beta(F_{b}\cup_{i}(G_{a}\ast G_{b}))=0$, because $F_{b}\cap G_{a}=\emptyset$ and $G_{a}\neq \emptyset$.\\
$\rhs=\beta(F_{b})\cup_{i}\beta(G_{a}\ast G_{b})=(\tv_{a}\otimes F_{b})\cup_{i}(cG_{a}\otimes G_{b})=
(\tv_{a}\cup_{0}cG_{a})\otimes (F_{b}\cup_{i}G_{b})=0, \text{ because the cone point,} \,\tv_{a}, \text{is the greatest vertex}$.
\item $F_{a}\neq\emptyset$, $F_{b}= \emptyset$, $G_{a}\neq\emptyset$, $G_{b}\neq \emptyset$.\\
$\lhs=\beta(F_{a}\cup_{i}(G_{a}\ast G_{b}))=\left\{\begin{array}{ll}
\beta((F_{a}\cup_{i}G_{a})\ast G_{b})= c(F_{a}\cup_{i}G_{a})\otimes G_{b},&\text{if } i \text{ even},\\
0,&\text{if } i \text{ odd}.
\end{array}\right.$\\
$\rhs= \beta(F_{a})\cup_{i}\beta(G_{a}\ast G_{b})=
(F_{a}\otimes(\tv_{b}+\mathbb V_{b}))\cup_{i}(cG_{a}\otimes G_{b})=
(F_{a}\cup_{i}cG_{a})\otimes (\mathbb V_{b}\cup^i_{0}G_{b})=
\left\{\begin{array}{ll}
c(F_{a}\cup_{i}G_{a})\otimes G_{b},&\text{if } i \text{ even},\\
0,&\text{if } i \text{ odd}. 
\end{array}\right.$\\
The nullity when $i$ is odd comes from \propref{prop:steenrodjoin}.
\item $F_{a}\neq\emptyset$, $F_{b}\neq \emptyset$, $G_{a}=\emptyset$, $G_{b}\neq \emptyset$.\\
$\lhs=\beta((F_{a}\ast F_{b})\cup_{i}G_{b})=\beta(F_{a}\ast (F_{b}\cup_{i}G_{b}))=cF_{a}\otimes (F_{b}\cup_{i}G_{b})$.\\
$\rhs=\beta(F_{a}\ast F_{b})\cup_{i}\beta(G_{b})=
(cF_{a}\otimes F_{b})\cup_{i}(\tv_{a}\otimes G_{b})=
(cF_{a}\cup_{0}\tv_{a})\otimes (F_{b}\cup_{i}G_{b})=
cF_{a}\otimes (F_{b}\cup_{i}G_{b})$.
\item $F_{a}\neq\emptyset$, $F_{b}\neq \emptyset$, $G_{a}\neq\emptyset$, $G_{b}= \emptyset$.\\
$\lhs=\beta((F_{a}\ast F_{b})\cup_{i}G_{a})=\left\{
\begin{array}{ll}
0,&\text{if } i \text{ even},\\
\beta((F_{a}\cup_{i}G_{a})\ast F_{b})=c(F_{a}\cup_{i}G_{a})\otimes F_{b},&\text{if } i \text{ odd}.
\end{array}\right.$\\
$\rhs=\beta(F_{a}\ast F_{b})\cup_{i}\beta(G_{a})=
(cF_{a}\otimes F_{b})\cup_{i}(G_{a}\otimes (\tv_{b}+\mathbb V_{b}))=
(cF_{a}\cup_{i}G_{a})\otimes(F_{b}\cup^i_{0}\mathbb V_{b})=
\left\{\begin{array}{ll}
0,&\text{if } i \text{ even},\\
c(F_{a}\cup_{i}G_{a})\otimes F_{b},&\text{if } i \text{ odd},
\end{array}\right.$\\
with the argument already used in the case (6).

\end{enumerate}

$\bullet$ The end of the proof is concerned with the map $\alpha$ when at least one of the subsets, $F_{a}$, $F_{b}$, $G_{a}$, $G_{b}$ is the empty set.
Computations are similar to the previous ones.
\end{proof}
%

%\nocite{*}
%%%%%%%%%%%%%%%%
 %\bibliographystyle{unsrt}
%\bibliographystyle{abbrv}
%\bibliographystyle{plain}
%\bibliographystyle{smfalpha}
%\bibliographystyle{smfplain}
\bibliographystyle{amsplain}
%%%%%%%%%%%%%%%%%%%%
\bibliography{IntersectionSteenrod}

\end{document}